\documentclass[11pt]{amsart}

\usepackage{amssymb}
\usepackage{enumerate}
\usepackage{hyperref}

\input xy
\xyoption{all}

\newtheorem{theorem}{Theorem}[section]
\newtheorem{proposition}[theorem]{Proposition}
\newtheorem{lemma}[theorem]{Lemma}
\newtheorem{corollary}[theorem]{Corollary}
\theoremstyle{definition}
\newtheorem{definition}[theorem]{Definition}
\newtheorem{remark}[theorem]{Remark}
\newtheorem{example}[theorem]{Example}

\title[Weakly contractive iterated function systems...]{Weakly contractive iterated function systems 
		and beyond: A manual
}

\author{Krzysztof Le\'{s}niak, Nina Snigireva, Filip Strobin}
\address{K. Le\'sniak: 
Faculty of Mathematics and Computer Science, 
Nicolaus Copernicus University in Toru\'{n},
Chopina 12/18, 87-100 Toru\'{n}, Poland}
\email{much@mat.umk.pl}

\address{N. Snigireva: 
School of Mathematics and Statistics, University
College Dublin, Belfield, Dublin 4, Ireland}
\email{nina@maths.ucd.ie}
\address{F. Strobin: 
Institute of Mathematics, 
Lodz University of Technology, 
W\'{o}lcza\'{n}ska 215, 
90-924 \L\'{o}d\'{z}, Poland}
\email{filip.strobin@p.lodz.pl}

\date{}

\begin{document}

\begin{abstract}
We give a systematic account of iterated function systems (IFS) of weak contractions 
of different types (Browder, Rakotch, topological). 
We show that the existence of attractors and asymptotically stable invariant measures, 
and the validity of the random iteration algorithm (``chaos game"), can be 
obtained rather easily for weakly contractive systems. We show that the class 
of attractors of weakly contractive IFSs is essentially wider than the class 
of classical IFSs' fractals. On the other hand, we show that, in reasonable spaces, 
a typical compact set is not an attractor of any weakly contractive IFS.
We explore the possibilities and restrictions to break the contractivity barrier 
by employing several tools from fixed point theory: geometry of balls, 
average contractions, remetrization technique, ordered sets, and measures of noncompactness.
From these considerations it follows that while the existence of invariant sets 
and invariant measures can be assured rather easily for general iterated function systems
under mild conditions, to establish the existence of attractors 
and unique invariant measures is a substantially more difficult problem.
This explains the central role of contractive systems in the theory of IFSs.
\\ \\ \\
\end{abstract}
\maketitle

\tableofcontents

\section{Introduction}

The aim of the article is to show that 
invariant sets and invariant measures exist 
in iterated function systems (IFS) 
under fairly general conditions. 
Special attention is paid to systems 
of weakly contractive maps, because
invariant sets and invariant measures in such 
systems turn out to be unique and attracting.
A topological characterization of weakly contractive
systems is discussed which leads to the notion of
topologically contractive IFS. The prevalence of
topological fractals (attractors of topologically 
contractive IFSs) in the sense of Baire category 
is exhibited. A simple proof of a derandomized
(a.k.a. deterministic a.k.a. disjunctive) chaos game
for weakly contractive and topologically contractive
IFSs has been supplied. Limitations to various
generalizations of contractivity for IFSs, as well as
employment of metric fixed point theory methods,
have been highlighted.

Large part of this article, its core, can be 
summarized as follows: 
The sequences of successive iterates 
of sets and measures are convergent 
to invariant sets and invariant measures. 
While this premise comes as no surprise within 
the realm of contractive IFSs, it is less 
obvious for non-contractive IFSs, yet true 
if the convergence is understood in a suitably
weakened manner.

In this survey, we omit the theory of 
fractal dimension (cf. \cite{Falconer}), 
fractal interpolation (cf. \cite{Massopust}),
fractal compression (cf. \cite{Mendivil}) 
and analysis on fractals (cf. \cite{Strichartz}), 
and concentrate on fundamental questions of 
existence and uniqueness of attractors 
and invariant measure.
The presentation of the chosen aspects 
of fractal geometry is self-contained. 
Therefore, we elaborate on some 
aspects of topology, measure and order.

To facilitate proper understanding of the level 
of generality of the presented results, 
we employ the Vietoris topology instead of
the Hausdorff metric whenever needed (both
yield the same convergence of compacta)
and we use Radon measures instead of 
Borel measures (which is indifferent 
when measures live on 
a complete separable metric space).
The standard notions such as, 
Vietoris topology and Monge-Kantorovich metric, 
are recalled in the appendix 
for the convenience of a reader.

Our survey complements the existing surveys on IFSs 
such as \cite{Iosifescu}, \cite{Stenflo-average},  
\cite{BV-Developments}. Among many books devoted 
to fractal geometry we choose to cite  \cite{Edgar} 
and \cite{Mendivil} as they best cover the aspects 
of IFSs we are interested in in this survey.

Let us begin with some necessary definitions.

\begin{definition}
If $X$ is a set and $f:X\to X$, then we say 
that $x_*\in X$ is a \emph{fixed point} of 
$f$ provided $f(x_*)=x_*$.
If additionally $X$ is a metric (or topological)
space, then $x_*$ is called 
a \emph{contractive fixed point}, \emph{CFP} for short, 
if for every $x\in X$, the sequence of 
iterations $(f^k(x))$ converges to $x_*$.
\end{definition}

\begin{remark}\label{remark:CFP}
(1) A map having CFP is sometimes called 
a \emph{Picard operator}, cf. \cite{Petrusel}.

(2) A CFP is unique if the underlying 
space $X$ is at least Hausdorff. 
Notably, for the considerations of 
the weak convergence of the iterations 
of the Markov operator this may not be the case.

(3) If $x_*$ is a CFP of the $p$-fold composition 
$f^p$ of $f$, then $x_*$ is also a CFP of $f$. 
Indeed, let $x_*\in X$ be a CFP of $f^p$. 
Then for every $x\in X$ and 
$i=0,...,p-1$, we have 
$f^{pk+i}(x) = (f^p)^k(f^i(x)) 
\overset{k\to\infty}{\to} x_*$, 
which implies that $f^k(x)\to x_*$. 
(see also \cite{Granas} chap.I $\S$1.6 (A.1)
or \cite{Tarafdar} Remark 2.4 p.13.) 
\end{remark}

\begin{definition}
An \emph{iterated function system} 
${\mathcal{F}} = \{w_1,...,w_N\}$, \emph{IFS} in short, 
is a finite family of maps $w_i:X\to X$ 
acting on a (topological or metric) space $X$.
The \emph{Hutchinson operator} induced by 
an IFS ${\mathcal{F}}$, denoted without ambiguity 
again by ${\mathcal{F}}$, is the map ${\mathcal{F}}: 2^X\to 2^X$ 
acting on the power set $2^X$ of subsets 
of $X$ and given by the formula
\begin{equation*}
{\mathcal{F}}(S) = \overline{\bigcup_{i=1}^{N} w_i(S)}
\mbox{ for every } S\subseteq X.
\end{equation*}
A set $A_{*}\subseteq X$ is called 
\emph{$\mathcal{F}$-invariant}, if it is a fixed point 
of $\mathcal{F}$, i.e., $\mathcal{F}(A_{*})=A_{*}$. 
(When $\mathcal{F}$ is clear from the context we just speak of
invariant sets instead of $\mathcal{F}$-invariant sets.)
\end{definition}

\begin{remark}
(1) If not stated otherwise, for simplicity 
we will assume that if $\mathcal{F}$ is an IFS, 
then it consists of maps $w_1,...,w_N$. 
We also associate with $\mathcal{F}$ an alphabet
of symbols $I=\{1,...,N\}$ representing maps $w_i$.

(2) Closedness of an invariant set is built-in its definition.

(3) Invariant sets are often called \emph{self-similar},
especially when the IFS consists of similarities.
\end{remark}

Usually the Hutchinson operator is considered 
on some hyperspace of $X$. 
For most of the work we will consider 
${\mathcal{F}}:\mathcal{K}(X)\to \mathcal{K}(X)$ 
acting on the hyperspace $\mathcal{K}(X)$
of nonempty compact subsets of $X$, 
equipped either with the Hausdorff 
distance $d_{H}$ or the Vietoris topology. 
In case $X$ is a metric space,
the Hausdorff distance topology and 
the Vietoris topology
on ${\mathcal{K}}(X)$ coincide (see Appendix, 
Theorem \ref{th:HyperspaceCompleteCompact}). 
The restriction of $\mathcal{F}$ from $2^{X}$ to 
${\mathcal{K}}(X)$ is possible when each $w_i$ 
is continuous and then also 
$\mathcal{F}(K)=\bigcup_{i=1}^N w_i(K)$ for $K\in{\mathcal{K}}(X)$.

\begin{definition} 
The \emph{probabilistic IFS} $(\mathcal{F},\vec{p})$ 
is an IFS ${\mathcal{F}}=\{w_1,...,w_{N}\}$ of 
continuous maps acting on 
a Hausdorff topological space $X$, 
together with a vector 
$\vec{p}=(p_1,...,p_N)$ of positive weights, 
that is $\sum_{i=1}^{N} p_i=1$, $p_i>0$.
The \emph{Markov operator} associated with 
$(\mathcal{F},\vec{p})$, is a map 
$M:{\mathcal{M}}_{\pm}(X)\to{\mathcal{M}}_{\pm}(X)$ 
defined on the space ${\mathcal{M}}_{\pm}(X)$ of 
signed Radon measures on $X$, 
according to the formula
\begin{equation}\label{eq:operatorM}
M(\mu) = \sum_{i=1}^{N} 
p_i\cdot \mu\circ w_i^{-1}
\mbox{ for every } \mu\in{\mathcal{M}}_{\pm}(X), 
\end{equation}
where $\mu\circ w_i^{-1}$ is the push-forward 
of $\mu$ through $w_i$ 
(see Appendix, Section \ref{section:measures}). 
A measure $\mu_{*}\in{\mathcal{M}}_{\pm}(X)$ is called
\emph{invariant}, if it is a fixed
point of the Markov operator $M$ induced by the
probabilistic IFS $(\mathcal{F},\vec{p})$, i.e., 
$M\mu_{*}=\mu_{*}$.
\end{definition}

\begin{remark}\label{rem:federed}
(1) To underline that the Markov operator is induced 
by the probabilistic IFS $(\mathcal{F},\vec{p})$ we write  
$M_{(\mathcal{F},\vec{p})}$ in place of simple $M$.

(2) Invariant measures are also called 
\emph{self-similar}, or \emph{stationary}.

(3) For future use, it is worth to note here that, 
if $\mathcal{F}$ consists of continuous maps $w_i:X\to X$, 
then for all bounded Borel measurable 
functions $g:X\to {\mathbb{R}}$
and $p_i>0$, $i=1,...,N$,
the following holds:
\begin{equation}\label{feller}
\int_{X} g \;dM_{(\mathcal{F},\vec{p})}\mu=
\sum_{i=1}^{N} p_i\int_{X} g\circ w_i \;d\mu.
\end{equation}
(See Appendix, Proposition \ref{th:wtransport}.)
\end{remark}

The space of signed Radon measures is 
endowed with the weak topology. Usually 
the action of the Markov operator $M$ 
is considered on some subspace of ${\mathcal{M}}_{\pm}(X)$: 
${\mathcal{P}}(X)$ comprising Radon probability measures 
or  (if $X$ is a metric space) 
${\mathcal{P}}_1(X)$ comprising Radon probability measures with 
integrable distance. 
For more information see Appendix, 
Section \ref{section:measures}.

Given $I=\{1,...,N\}$, $N\geq 1$, 
the \emph{code space} $I^\infty$ is 
the Tikhonov product of 
countably many copies of $I$. 
We endow it with the Baire metric $d_B$
(see Section \ref{sec:code space} in Appendix 
for further information).

\begin{definition}\label{remark:canonical}
Consider the code space $I^\infty$, where $I=\{1,...,N\}$. 
For each $i\in I$, let $\tau_i:I^\infty\to I^\infty$ 
be defined by 
$\tau_i((\alpha_1,\alpha_2,...)):=(i,\alpha_1,\alpha_2,...)$. 
Then we call $\mathcal{T}= \{\tau_1,...,\tau_N\}$ 
the \emph{canonical IFS} on $I^\infty$. 
\end{definition}

\section{Metrically contractive iterated function systems}

\subsection{Taming a plethora of weak contractions}

In this section we make a short overview of generalizations 
of the Banach fixed point theorem for weak contractions. 

The classical Banach fixed point theorem from 1922 
states that if $X$ is a complete metric space and 
$f:X\to X$ is a \emph{Banach contraction} 
(that is, the Lipschitz constant $\operatorname{Lip}(f)<1$), 
then $f$ has the CFP. Since the beginning of 1960's 
there has been an effort to weaken the 
contractive assumptions in Banach's theorem. 

\begin{definition}
Given a function ${\varphi}:{\mathbb{R}}_+\to{\mathbb{R}}_+$, 
we say that $f$ is a \emph{${\varphi}$-contraction} 
(and that ${\varphi}$ is a~\emph{modulus of continuity} 
or a \emph{comparison function} for $f$), if
\begin{equation*}
d(f(x),f(y))\leq{\varphi}(d(x,y)) 
\mbox{ for } x,y\in X.
\end{equation*}
\end{definition}

Clearly, $f:X\to X$ is a Banach contraction 
iff $f$ is a ${\varphi}$-contraction for 
${\varphi}(t):=\lambda t$, where $\lambda\in(0,1)$. 
It turned out that much less can be 
assumed on ${\varphi}$.

\begin{definition}\label{def:contractions}
Given a metric space $X$, we say that 
$f:X\to X$ is
\begin{itemize}
\item[(i)] a \emph{Rakotch contraction}, 
provided $f$ is a ${\varphi}$-contraction 
for some comparison function ${\varphi}$ of the form 
${\varphi}(t):=\lambda(t)t$, where 
$\lambda:{\mathbb{R}}_+\to[0,1]$ is nonincreasing and 
$\lambda(t)<1$ for $t>0$;
\item[(ii)] a \emph{Browder contraction}, 
provided $f$ is a ${\varphi}$-contraction, where 
${\varphi}$ is nondecreasing right continuous 
and ${\varphi}(t)<t$ for $t>0$;
\item[(iii)] an \emph{Edelstein contraction}, 
provided $f$ satisfies
\begin{equation*}
d(f(x),f(y))<d(x,y) \mbox{ for all } 
x,y\in X,\; x\neq y.
\end{equation*}
\end{itemize}
\end{definition}

\begin{remark}
As we develop the theory of 
weakly contractive iterated function systems 
we will introduce some additional types of contractions 
which are broad generalizations of weak contractions. 
Namely, we will introduce Matkowski contractions 
in Remark \ref{remmatkowski}, 
eventual contractions and Tarafdar contractions 
in Section \ref{sec:EC}, 
as well as generalizations of weakly contractive IFSs: 
topologically contractive 
in Definition \ref{def:TIFS}, 
average contractive 
in Remark \ref{rem:average}, 
and average Rakotch contractive 
in Definition \ref{def:averageRakotch}.
\end{remark}

Clearly, each Banach contraction is 
a Rakotch contraction and each Browder 
contraction is Edelstein's. Also, as 
we show in Lemma \ref{prop1} below, each 
Rakotch contraction is a Browder contraction. 
The converse implications do not hold. 
For example, the map $f(x)=\sin x$ on $X={\mathbb{R}}$
is a Rakotch contraction but not Banach's.

F. Browder in 1968 proved that each Browder
contraction on a complete metric space 
has the CFP, generalizing an earlier 
result of E. Rakotch from 1962, who proved 
the CFP for Rakotch contractions; e.g., 
\cite{Granas} chap.I $\S$1.6 (B2).
On the other hand, M. Edelstein in 1962 
showed that each Edelstein contraction 
on a compact space $X$ has the CFP; e.g.
\cite{Granas} chap.I $\S$1.6 (A.7) and (B.7).
Although Edelstein's result fails on 
arbitrary complete metric spaces (e.g., 
for $f(x):=x+e^{-x}$ on $X=[0,\infty)$, see 
also \cite{Granas} chap.I $\S$1.6 (A.7) (c)),
it looks like a generalization of 
the Browder theorem on compact spaces. 
This impression breaks down according 
to the following folklore observation 
(the proof can be found 
in \cite{Jachymski3}; it relies on 
a characterization of Rakotch contractions 
given in Lemma \ref{prop1}).

\begin{proposition}\label{prop}
Assume that $X$ is compact and $f:X\to X$. 
Then $f$ is an Edelstein contraction 
iff $f$ is a Rakotch contraction.
\end{proposition}

\begin{remark}\label{remmatkowski}
(1) It is worth to observe that if 
${\varphi}:{\mathbb{R}}_+\to{\mathbb{R}}_+$ is nondecreasing, right
continuous and ${\varphi}(t)<t$ for $t>0$, then
\begin{equation}\label{matkowski}
\lim_{k\to\infty}{\varphi}^k(t)=0
\mbox{ for } t>0;
\end{equation}
e.g., \cite{Granas} chap.I $\S$1.6 (B.2).
A map $f:X\to X$ is called a 
\emph{Matkowski contraction}, 
if $f$ is a ${\varphi}$-contraction for some
nondecreasing function ${\varphi}$ which 
satisfies \eqref{matkowski}.
In 1975 J. Matkowski proved that 
each Matkowski contraction on a complete 
metric space has the CFP, generalizing 
the result of Browder. However, as was 
proved by J. Jachymski (see \cite{JJ}), 
the second iteration of a Matkowski 
contraction is Browder's, so Matkowski 
theorem follows from Browder's (recall 
Remark \ref{remark:CFP}). In fact, the 
proof given in \cite[Theorem 3]{Jachymski2}
shows more: if $f,g:X\to X$ are Matkowski 
contractions, then their composition 
$f\circ g$ is a Browder contraction.

(2)
Monotonicity of ${\varphi}$ is 
an important ingredient in the definition 
of Matkowski's contraction.
Wicks in his definition of a \emph{reduction} 
(\cite{Wicks}) assumes only (\ref{matkowski}), 
without monotonicity, but in his proofs 
he actually uses monotonicity. 
An example of a ${\varphi}$-contraction, 
for ${\varphi}$ satisfying only (\ref{matkowski}), 
on a complete space without a fixed point 
can be found in \cite{Mat3}.
\end{remark}

In the literature we can find many other 
(and essentially weaker) contractive-type
conditions which guarantee the existence of 
the CFP of a map - we refer the interested
reader to the survey \cite{JJ} by J. Jachymski
and I. J\'o\'zwik, in which a detailed 
discussion on various types of contractive
conditions and mutual relationships between 
them is given (if not stated otherwise,
presented results in this section can also be found there). 
We restricted ourselves here to 
probably the most important ones, which in
addition can be defined in a simple and 
natural way. In fact, Proposition \ref{prop}
shows that for compact spaces $X$, all these 
weak contractions are Rakotch
contractions. Moreover, as was proved by 
J. Matkowski and R. W\c egrzyk \cite{Mat2}, 
if the underlying space $X$ is 
\emph{metrically convex} (in particular, 
if $X$ is a Banach space), then $f$ is 
a Rakotch contraction if and only if 
$f$ is a $\varphi$-contraction for 
a~comparison function that satisfies
$\varphi(t)<t$ for $t>0$. Thus, 
Rakotch contractions are sufficient 
to explain other fixed point theorems 
in many natural cases.

A bit surprisingly, allowing for a change of 
the underlying metric, most such 
generalizations can be deduced from 
the Banach theorem, as was shown, 
for example, by P. Meyers in 1967
(e.g., \cite{Granas} chap.I $\S$1.7 p.24; 
note that there is a misprint in the 
formulation of Meyers' theorem: 
neighbourhoods $V$ and $U$ are 
misplaced in condition (ii) below).
\begin{theorem}[Meyers' remetrization theorem]\label{theorem:meyers}
Let $f$ be a continuous selfmap of 
a completely metrizable space $X$. 
Assume that there exists $x_0\in X$ 
such that
\begin{itemize}
\item[(i)] for any $x\in X$, 
$\lim_{k\to\infty} f^k(x)=x_0$ 
(i.e., $x_0$ is the CFP of $f$);
\item[(ii)] there exists an open neighbourhood 
$U$ of $x_0$ such that for every open
neighbourhood $V$ of $x_0$, there is 
$n\in{\mathbb{N}}$ such that $f^k(U)\subset V$ 
for all $k\geq n$.
\end{itemize}
Then there is an admissible complete 
metric $\rho$ on $X$ such that $f$ 
is a Banach contraction with respect to $\rho$.
\end{theorem}

\begin{remark}
Each Browder contraction satisfies (i) and (ii)
in Theorem \ref{theorem:meyers}
(just take $U:=B(x_0,1)$ to be the open ball). 
Thus the Browder fixed point theorem essentially 
follows from the Banach fixed point principle.
\end{remark}

\begin{remark}
Conditions (i) and (ii) of 
Theorem \ref{theorem:meyers} are satisfied, 
if $X$ is compact and the intersection
$\bigcap_{k\in{\mathbb{N}}}f^k(X)$ is a singleton, 
i.e., when $f$ is Tarafdar's topological 
contraction (\cite{Tarafdar} chap.2.7.1 p.88);
see also Section \ref{sec:EC}.
\end{remark}

\begin{remark}
We note that Meyers’ remetrization theorem 
is only applicable to an iterated function system (IFS) 
consisting of a single map. 
In this case, we obtain an IFS 
consiting of a Banach contraction. 
In general, one cannot remetrize a given space 
so that several maps are Banach contractions 
with respect to the same metric 
(see \cite{Ka2} Section 1.4). 
To deal with the situation when an IFS 
consists of more than one map 
we will prove remetrization theorem 
for topologically contractive IFSs (TIFS)
which we will discuss in Section \ref{sec:TIFS}. 
In this case, we obtain an IFS 
consisting of weak contractions. 
Therefore, for an IFS consisting of a single map 
Meyers’ remetrization theorem gives a stronger result.
We also note that the remetrization theorem 
for TIFS is only used in Section \ref{sec:remetrizeTIFS}. 
In order to obtain results for TIFS from 
a weakly contractive case via 
the remetrization theorem for TIFS 
one needs to assume that the underlying
space $X$ is metrizable 
(cf. Remark \ref{rem:chgameRemetrize}). 
However, in general, TIFS can be defined 
on nonmetrizable spaces (Example \ref{ex:nonmetrizTIFS}). 
\end{remark}

We end this section with two lemmas which 
will be useful later, when dealing with 
iterated function systems consisting of 
weak contractions. The first one gives a nice
characterization of Rakotch contractions 
(see \cite[Theorem 1]{JJ} 
for more characterizations like this one).

\begin{lemma}\label{prop1}
Let $X$ is a complete metric space and 
$f:X\to X$. The following conditions are 
equivalent:
\begin{itemize}
\item[(i)] $f$ is a Rakotch contraction;
\item[(ii)] $f$ is a ${\varphi}$-contraction for 
some concave and strictly increasing ${\varphi}$ 
such that ${\varphi}(t)<t$ for $t>0$;
\item[(iii)] for every $\delta>0$, there 
exists $\lambda<1$ such that if $x,y\in X$ 
satisfy $d(x,y)\geq \delta$, then 
$d(f(x),f(y))\leq\lambda d(x,y)$.
\end{itemize}
\end{lemma}
The second lemma is technical and states 
some relationship between different 
types of comparison functions ${\varphi}$ 
suitable for Rakotch contractions. 
It follows directly from \cite[Lemma 1]{JJ}.
\begin{lemma}\label{lemma:rakotch}
Let ${\varphi}:{\mathbb{R}}_+\to{\mathbb{R}}_+$ be defined by 
${\varphi}(t)=\lambda(t)t$ for some 
nonincreasing $\lambda:{\mathbb{R}}_+\to[0,1]$ 
with $\lambda(t)<1$ for $t>0$. Then 
there exists a concave strictly increasing 
function $\tilde{{\varphi}}:{\mathbb{R}}_+\to{\mathbb{R}}_+$ such 
that ${\varphi}(t)\leq\tilde{{\varphi}}(t)<t$ for $t>0$. 
\end{lemma}

\subsection{Attractors for weakly contractive IFSs}

\begin{definition}
W say that an IFS $\mathcal{F}$ on a metric space $X$ 
is \emph{Banach} (\emph{Rakotch, Browder,
Edelstein} respectively) \emph{contractive} 
if it consists of Banach (Rakotch, Browder, 
Edelstein respectively) contractions. 
We will refer to all such IFSs as 
\emph{weakly contractive} IFSs.
\end{definition}

The classical Hutchinson-Barnsley theorem states that:
\begin{theorem}\label{HB1}
If $\mathcal{F}=\{w_1,...,w_N\}$ is a Banach contractive IFS 
on a complete metric space $X$, then 
the Hutchinson operator $\mathcal{F}:{\mathcal{K}}(X)\to{\mathcal{K}}(X)$ 
has the CFP, i.e., 
there is a unique set $A_{\mathcal{F}}\in{\mathcal{K}}(X)$ satisfying
two conditions
\begin{enumerate}
\item[(i)] invariance:
$A_{\mathcal{F}}=\mathcal{F}(A_{\mathcal{F}})=\bigcup_{i=1}^{N} w_i(A_{\mathcal{F}})$,
and
\item[(ii)] attractivity:
for every $K\in{\mathcal{K}}(X)$, the sequence of 
iterations $\mathcal{F}^k(K)$ converges to $A_{\mathcal{F}}$ 
(with respect to the Hausdorff metric).
\end{enumerate}
\end{theorem}

The above theorem suggests the following definition:

\begin{definition}\label{def:Attractor}
Given an IFS $\mathcal{F}$ on a Hausdorff topological space $X$
(in particular, a metric space $X$),
a nonempty compact set $A_{\mathcal{F}}$, 
which is the CFP (contractive fixed point) 
of the Hutchinson operator $\mathcal{F}:{\mathcal{K}}(X)\to{\mathcal{K}}(X)$, 
will be called an \emph{attractor} generated by $\mathcal{F}$.
In other words, $A_{\mathcal{F}}$ fulfills conditions
(i) and (ii) from Theorem \ref{HB1}.
The convergence of sets in (ii) has to be understood 
in the Vietoris topology on ${\mathcal{K}}(X)$. 
\end{definition}

\begin{remark}
If the Hutchinson operator is continuous 
(see Proposition \ref{th:ContinuityF}), then 
the invariance condition (i) follows from 
the attractivity condition (ii) 
in Definition \ref{def:Attractor}. 
\end{remark}

\begin{example}
Let $\mathcal{T}$ be the canonical IFS
on the code space $I^\infty$, 
cf. Definition \ref{remark:canonical}. 
It is easy to see that $\mathcal{T}$ 
is a Banach contractive IFS and 
$I^\infty$ is its attractor. 
\end{example}

Theorem \ref{HB1} follows easily from 
the Banach fixed point theorem, 
as the Hutchinson operator turns out 
to be a Banach contraction on the 
hyperspace $({\mathcal{K}}(X),d_{H})$, provided that 
the underlying IFS is Banach contractive. 
Many known generalizations of this result 
are based on the same idea, that is, 
they rely on proving that the Hutchinson 
operator satisfies the same contractive 
condition as maps from the IFS. 
In many papers we can find
particular cases of the following:

\begin{lemma}\label{ficont}
Let $\mathcal{F}$ be an IFS on a metric space $X$ 
that consists of ${\varphi}$-contractions 
for some nondecreasing function ${\varphi}$. 
Then the Hutchinson operator 
$\mathcal{F}:{\mathcal{K}}(X)\to{\mathcal{K}}(X)$ is a ${\varphi}$-contraction.
\end{lemma}
\begin{proof}
Assume first that $w:X\to X$ is 
a ${\varphi}$-contraction and choose $K,S\in{\mathcal{K}}(X)$ 
and $x_0\in K$. There is $y_0\in S$ such that 
$d(x_0,y_0)=\inf\{d(x_0,y):y\in S\}$. We have
\begin{equation*}
\inf_{y\in S} d(w(x_0),w(y)) \leq
d(w(x_0),w(y_0))\leq {\varphi}(d(x_0,y_0))= 
{\varphi}\left(\inf_{y\in S} d(x_0,y)\right)
\leq {\varphi}(d_H(K,S)).
\end{equation*}
Hence
\begin{equation*}
e(w(K),w(S))=
\sup_{x\in K} \inf_{y\in S} d(w(x),w(y)) 
\leq {\varphi}(d_H(K,S)).
\end{equation*}
By symmetry $e(w(S),w(K)) \leq {\varphi}(d_H(K,S))$,
which implies $d_H(w(K),w(S))\leq{\varphi}(d_H(K,S))$.

Now, if $\mathcal{F}=\{w_1,...,w_N\}$ consists of
${\varphi}$-contractions, then by the above and 
known properties of the Hausdorff metric, 
we have
\begin{eqnarray*}
d_H(\mathcal{F}(K),\mathcal{F}(S))= 
d_H\left(\bigcup_{i=1}^{N} w_i(K), 
\bigcup_{i=1}^{N} w_i(S)\right) \\
\leq \max_{i=1,...,N} d_H(w_i(K),w_i(S))
\leq {\varphi}(d_H(K,S)) 
\mbox{ for } K,S\in{\mathcal{K}}(X).
\end{eqnarray*}
\end{proof}

As an immediate corollary, we get the 
following extension of Theorem \ref{HB1}:

\begin{theorem}\label{th1}
If $\mathcal{F}$ is an IFS on a complete metric 
space $X$ that consists of Browder [Banach,
Rakotch, Matkowski, respectively] contractions,
then the Hutchinson operator $\mathcal{F}:{\mathcal{K}}(X)\to{\mathcal{K}}(X)$
is a Browder [Banach, Rakotch, Matkowski,
respectively] contraction and hence 
it has the CFP, i.e., $\mathcal{F}$ generates 
the unique compact invariant set 
which is an attractor.
\end{theorem}
\begin{proof}
In view of Lemma \ref{ficont}, we only 
have to observe that if $\mathcal{F}=\{w_1,...,w_N\}$ 
consists of Browder [Banach, Rakotch, 
Matkowski, respectively] contractions $w_i$ 
with respective comparison functions ${\varphi}_i$, 
then all $w_i$'s are ${\varphi}$-contractions 
for the same function ${\varphi}$ with suitable
properties. Namely,
${\varphi}:=\max\{{\varphi}_1,...,{\varphi}_N\}$ 
meets the desired conditions.
\end{proof}

\begin{remark}\label{rem:generalization}
In the literature there are given many further
generalizations of Theorem \ref{HB1} in the spirit 
of Theorem \ref{th1}. 
Here we want to point out two further directions.

(1) We can extend the thesis of Theorem \ref{th1} 
by showing that for any nonempty closed 
and bounded set $B\subset X$ 
(denote by $\mathcal{CB}(X)$ the family 
of such sets; note that the Hausdorff metric 
on $\mathcal{CB}(X)$ is complete provided 
$X$ is complete - see Appendix), 
the sequence of iterations $\mathcal{F}^k(B)$ 
converges to $A_\mathcal{F}$. To see it, observe that 
we can easily extend Lemma \ref{ficont} 
by considering the Hutchinson operator 
$\mathcal{F}:\mathcal{CB}(X)\to\mathcal{CB}(X)$, 
under additional assumption that ${\varphi}$ 
is right continuous. 
Thus $\mathcal{F}:\mathcal{CB}(X)\to\mathcal{CB}(X)$ 
is a Browder contraction if the IFS $\mathcal{F}$ is 
(at least) Browder contractive and we are done. 
If $\mathcal{F}$ is  Matkowski contractive, 
then using the mentioned result 
from \cite{Jachymski2}, the second iteration 
$\mathcal{F}^2:\mathcal{CB}(X)\to\mathcal{CB}(X)$ 
is Browder's, so has the CFP and hence also
$\mathcal{F}:\mathcal{CB}(X)\to\mathcal{CB}(X)$ 
has the CFP.

(2) Let $\mathcal{F}$ be an IFS consisting of 
infinitely many weak contractions.
Moreover, let us assume that 
the weak contractivity of $\mathcal{F}$
is \emph{uniform} in the sense that 
all mappings from $\mathcal{F}$ are ${\varphi}$-contractions 
for a common comparison function ${\varphi}$ 
with appropriate properties.
Assume additionally that the induced 
Hutchinson operator $\mathcal{F}$ transforms compact sets 
onto compact sets, i.e., 
$\mathcal{F}({\mathcal{K}}(X))\subseteq {\mathcal{K}}(X)$.
Then the same reasoning as in 
Lemma \ref{ficont} shows that 
$\mathcal{F}$ is a ${\varphi}$-contraction 
and thus it has the CFP. 
Note that the assumption that all 
mappings are ${\varphi}$-contractions for 
the same function ${\varphi}$ is important.
For instance the IFS $\mathcal{F}:=\{w_t:t\in(0,1)\}$, 
where $w_t(x)=tx$, $x\in[0,1]$, does 
not generate a unique invariant set 
(each set $[0,a]$, where $a\in(0,1]$, 
is $\mathcal{F}$-invariant). 
See also \cite{Wicks}
and \cite{Mendivil} chap.2.6.4.1.

(3) Finally, we can mix the above two approaches 
and obtain a result for closed and bounded sets 
for infinite IFSs consisting of weak contractions.
\end{remark}

\begin{remark}
While a weakly contractive IFS induces 
a weakly contractive Hutchinson operator
(Theorems \ref{HB1} and \ref{th1}), 
an IFS comprising non-contractive maps 
can induce a Hutchinson operator with CFP, 
see Example \ref{ex:irrot}. This opens 
the gate to a whole world of non-contractive 
IFSs, e.g., \cite{BV-Developments}. 
\end{remark}

\subsection{Invariant measures for weakly contractive IFSs}

The existence of a unique invariant measure 
for weakly contractive IFSs 
has been established in 1996 
independently by A. Fan \cite{AihuaFan}, 
who employed ergodic theory techniques,
and A. Edalat \cite{Edalat}, 
who used order theory 
with the aim of studying some aspects 
of the computation theory for IFSs. 
Since then it has been rediscovered several times, 
e.g., \cite{Arbieto}, \cite{Okamura}.

We provide an elementary proof of this fact 
for Rakotch contractive IFSs. 
Our proof is much like the one given 
by K. Okamura in \cite{Okamura}. 
However, he restricted the discussion to 
measures supported on compact $\mathcal{F}$-invariant set, 
so our result is more general. 
Let us also note that a more straightforward 
proof for Banach contractive IFSs can be 
found in \cite{Mendivil}, but it seems
to be hard to adjust that approach 
to weakly contractive IFSs 
(see also the discussion in \cite{Okamura}).

Let $(X,d)$ be a complete metric space 
and let $\mathcal{P}_1(X)$ be the set of all 
Radon probability measures on $X$ 
with integrable distance $d$, 
i.e., measures $\mu$ such that for some 
(equivalently - for any) $x_0\in X$, 
the integral $\int_{X} d(x,x_0)\;d\mu(x)<\infty$.  
We endow $\mathcal{P}_1(X)$ with 
the Monge-Kantorovitch metric $d_{MK}$
which is complete and has the property 
that the convergence of measures 
with respect to $d_{MK}$ implies their weak convergence
(see Appendix, Lemma \ref{lemma:measuresonmetric}
(i) and (ii)).

\begin{theorem}\label{th:weakcontrMarkov}
Let $(\mathcal{F},\vec{p})$ be a probabilistic IFS
consisting of Rakotch contractions which
act on a~complete metric space $X$. 
Let $M:{\mathcal{M}}_{\pm}(X)\to {\mathcal{M}}_{\pm}(X)$ be 
the Markov operator induced by $(\mathcal{F},\vec{p})$.
Then
\begin{itemize}
\item[(a)] $M(\mathcal{P}_1(X)) \subseteq
\mathcal{P}_1(X)$;
\item[(b)] $M:\mathcal{P}_1(X)\to\mathcal{P}_1(X)$
is a Rakotch contraction with respect to 
the Monge-Kantorovitch metric;
\item[(c)] $M$ has the CFP in $\mathcal{P}_1(X)$,
i.e., there exists a unique 
$\mu_{*}\in\mathcal{P}_1(X)$ s.t.
$M(\mu_{*})=\mu_{*}$ and 
for every $\mu\in\mathcal{P}_1(X)$, 
the sequence of iterations $M^k(\mu)$ 
converges weakly to $\mu_{*}$.
\end{itemize}
\end{theorem}
\begin{proof}
(a) 
Using equality \eqref{feller} from 
Remark \ref{rem:federed} and 
Proposition \ref{th:wtransport} 
we easily infer that since $\mathcal{F}$ consists 
of Lipschitz maps, then $M(\mu)\in{\mathcal{P}}_1(X)$ 
for every $\mu\in{\mathcal{P}}_1(X)$.

(b) 
Let $\mathcal{F}=\{w_1,...,w_{N}\}$ and
$\vec{p}=(p_1,...,p_N)$. 
By Lemma \ref{lemma:rakotch} we can assume 
that all maps $w_i$ are ${\varphi}_i$-contractions for
concave strictly increasing functions ${\varphi}_i$ 
which satisfy ${\varphi}_i(t)<t$ for $t>0$. 
It is easy to see that then  
${\varphi}(t):= \sum_{i=1}^{N} p_i{\varphi}_i(t)$, 
$t\geq 0$, is also strictly increasing 
and concave, and ${\varphi}(t)<t$ for $t>0$.

By Lemma \ref{lemma:measuresonmetric}, 
the Monge-Kantorovitch metric $d_{MK}$  
has the following description:
\begin{equation*}
d_{MK}(\mu,\eta)= \min\left\{\int_{X\times X} 
d(x,y)\;d\lambda:
\lambda\in \Lambda(\mu,\eta)\right\},
\end{equation*}
where $\Lambda(\mu,\eta)$ consists 
of all Radon probability measures $\lambda$ 
on $X\times X$ such that the projection 
of $\lambda$ on the first and the second 
coordinate equals $\mu$ and $\eta$, respectively.

Take $\lambda\in \Lambda(\mu,\eta)$ so that 
$d_{MK}(\mu,\eta)= 
\int_{X\times X} d(x,y)\;d\lambda(x,y)$ 
and let 
$\tilde{\lambda}:=\sum_{i=1}^{N} 
p_i\lambda\circ (w_i,w_i)^{-1}$. 
Then for every Borel $B\subset X$, we have
\begin{eqnarray*}
\tilde{\lambda}(B\times X) = \sum_{i=1}^{N} 
p_i \lambda(w_i^{-1}(B)\times w_i^{-1}(X)) =
\\
= \sum_{i=1}^{N} p_i\lambda(w_i^{-1}(B)\times X) = 
\sum_{i=1}^{N} p_i\mu(w_i^{-1}(B)) = M(\mu)(B).
\end{eqnarray*}
Similarly we can show that 
$\tilde{\lambda}(X\times B) = M(\eta)(B)$. 
Hence $\tilde{\lambda}\in \Lambda(M(\mu),M(\eta))$. 

Observe now that 
$\tilde{\lambda}= 
M_{(\tilde{\mathcal{F}},\vec{p})}(\lambda)$, where
$M_{(\tilde{\mathcal{F}},\vec{p})}$ is the Markov
operator induced by the probabilistic 
IFS $\tilde{\mathcal{F}} = \{(w_i,w_i):i=1,...,N\}$
with vector of weights $\vec{p}$.
Therefore we can employ equality \eqref{feller} 
from Remark \ref{rem:federed}
alongside the concavity of $\varphi_i$'s
to obtain
\begin{eqnarray*}
d_{MK}(M(\mu),M(\eta))\leq
\int_{X\times X} d(x,y) \;d\tilde{\lambda} =
\\
= \sum_{i=1}^{N} p_i\int_{X\times X} 
d(w_i(x),w_i(y))\;d\lambda \leq 
\sum_{i=1}^{N} p_i\int_{X\times X} 
\varphi_i(d(x,y))\;d\lambda \leq 
\\
\leq \sum_{i=1}^{N} p_i\cdot {\varphi_i} 
\left(\int_{X\times X} d(x,y)\;d\lambda\right) = 
\sum_{i=1}^{N} p_i 
{\varphi_i}\left(d_{MK}(\mu,\eta)\right) =
\varphi(d_{MK}(\mu,\eta)).
\end{eqnarray*}

Item (c) is immediate from (b). 
\end{proof}

\begin{remark}
It turns out that the support of the CFP $\mu_{*}$
equals to the attractor of $\mathcal{F}$; 
see Theorem \ref{theorem:invariant} (ii) and 
Theorem \ref{th:invariantsupp}.
\end{remark}

\section{Topologically contractive iterated function systems}\label{sec:TIFS}

\subsection{Notation and basic definitions}

Let $\mathcal{F}=\{w_1,...,w_N\}$ be an IFS on $X$.
We find the following notational conventions
very useful for the whole section. 
We fix the alphabet $I=\{1,...,N\}$.
Given a finite word 
$\alpha=(\alpha_1,...,\alpha_k)\in I^k$, 
$k\in{\mathbb{N}}$, we set
\begin{equation}\label{eq:fibercomposition}
w_{\alpha}= w_{\alpha_{1}}\circ ... 
\circ w_{\alpha_{k}}.
\end{equation}
(Note carefully the order of composition 
typical for symbolic dynamics.)
In particular, we can use that notation
for a composition of $w_i$'s along a finite
prefix $\alpha_{\vert k}$ of an infinite
word $\alpha\in I^{\infty}$ and write
$w_{\alpha_{\vert k}}$.

\begin{definition}\label{def:TIFS}
We say that an IFS $\mathcal{F}$ on a topological 
space $X$ is a \emph{topologically contractive
iterated function system} (TIFS), 
if it consists of continuous maps and the 
following two conditions hold:
\begin{itemize}
\item[(i)] $\mathcal{F}$ is 
\emph{compactly dominated}
in the sense that for every $K\in{\mathcal{K}}(X)$, 
there exists $C\in{\mathcal{K}}(X)$ such that 
$K\subseteq C$ and $\mathcal{F}(C)\subseteq C$;
\item[(ii)] for every $C\in{\mathcal{K}}(X)$ with 
$\mathcal{F}(C)\subseteq C$ and every sequence 
$\alpha\in I^\infty$, the set 
$\bigcap_{k\in{\mathbb{N}}} w_{\alpha_{\vert k}}(C)$ 
is a~singleton.
\end{itemize}
\end{definition}
\begin{remark}\label{remark:compactdominacy}
If the space $X$ is compact or if 
$\mathcal{F}$ has an attractor $A_\mathcal{F}$, then 
condition (i) in Definition \ref{def:TIFS} 
is fulfilled for free.
To see it in the second case, observe 
that for every $K\in{\mathcal{K}}(X)$, the set 
\[
C:= A_\mathcal{F}\cup K\cup\bigcup_{n\in{\mathbb{N}}}\mathcal{F}^n(K)
= \overline{\bigcup_{n=0}^{\infty} \mathcal{F}^n(K)}\]
is compact, $K\subseteq C$ and $\mathcal{F}(C)\subseteq C$.
\end{remark}

A careful reader will notice that 
the compact dominance property appears 
implicitely in many discussions about IFSs. 
It is a form of localizing
(or, in other words, trapping) an attractor. 
It allows to reduce the discussion about the dynamics 
of the IFS to a compact set.

\begin{remark}
Definition \ref{def:TIFS} was introduced 
by A. Mihail in 2012 (cf. \cite{Mih3}) 
under the name 
\emph{topologically iterated function system}
and, independently, by A.V. Tetenov in 2010 
(see \cite{Tetenov2017} and the references therein) 
under the name 
\emph{self-similar topological structure
satisfying condition (P)}. 
However, its particular versions were 
considered earlier. 
In the case when $X$ is a compact metric space, 
TIFSs were called \emph{weakly hyperbolic IFSs} 
by A. Edalat \cite{Edalat} and 
\emph{point fibred IFSs} 
by B. Kieninger \cite{Kieninger}.
A. Kameyama \cite{Ka2} called 
a compact topological space as 
a \emph{topological self similar set}, 
if there exists an IFS $\mathcal{F}$ on $X$ 
comprising continuous maps together with 
a continuous surjection
$\pi:I^\infty\to X$ such that 
$\pi\circ\tau_i=w_i\circ \pi$ 
for every $i\in I$. 
Propositions \ref{prop:TIFSvsKameyama} and
\ref{prop:TIFSvsKieninger} explain
how the notions of a~topological self 
similar set, point-fibred IFS and TIFS are
interrelated.
\end{remark}

Kameyama's definition of a topologically 
self similar set (\cite{Ka2}) involves 
a very important concept of the coding map. 
We restate it as a separate definition.

\begin{definition}
Let $\mathcal{F}=\{w_1,...,w_N\}$ be an IFS on 
a metric (or topological) space $X$
and let $\mathcal{T}=\{\tau_1,...,\tau_N\}$ 
be the canonical IFS on the code space 
$I^{\infty}$;
see Definition \ref{remark:canonical}.
The continuous map $\pi:I^\infty\to X$ 
satisfying $w_i\circ \pi=\pi\circ \tau_i$ 
for every $i\in I$, 
will be called the \emph{coding map}.  
\end{definition}

\begin{remark}
The name ``coding map" for $\pi$, used frequently 
in the literature  (e.g., \cite{MauldinUrbanski}),
has not been standardized so far. 
Some other names are 
the \emph{address map} (\cite{Mendivil}), 
the \emph{coordinate map} (\cite{Kieninger}),
and the \emph{projection from the code space} 
(e.g., J. Geronimo, Ch. Bandt 
--- personal communication).
\end{remark}

\begin{remark}\label{newremfs}
Assume that $X$ is a Hausdorff topological space. 
According to Kameyama's definition, $X$ is topological self similar set 
if there exists a surjective coding map for some IFS 
$\mathcal{F}$ consisting of continuous selfmaps of $X$ 
(as $\pi$ is continuous and $I^\infty$ is compact, 
it automatically implies compactness of $X$). 
On the other hand, if a coding map $\pi$ is not surjective, 
then its image $\pi(I^\infty)$ is a topological self similar set. 
Indeed, as $\pi(I^\infty)$ is $\mathcal{F}$-invariant:
\begin{equation*}
\pi\left(I^\infty\right)=
\pi\left(\bigcup_{i=1}^N\tau_i\left(I^\infty\right)\right)=
\bigcup_{i=1}^N\pi\circ\tau_i\left(I^\infty\right)=
\bigcup_{i=1}^Nw_i\circ\pi\left(I^\infty\right)=
\bigcup_{i=1}^Nw_i\left(\pi\left(I^\infty\right)\right),
\end{equation*}
we just have to consider the IFS 
$\tilde{\mathcal{F}}:=\{w_{i\vert \pi(I^\infty)}:i=1,...,N\}$ 
on $\pi(I^\infty)$ consisting of restrictions 
of $w_i$'s to $\pi(I^\infty)$.
\end{remark}

\subsection{Existence of attractors and the coding map}

The following result (essentially given
in \cite{Mih3}) shows that TIFSs 
generate attractors which are projections 
of the code space via coding maps.

\begin{theorem}\label{HB2}
Assume that $\mathcal{F}=\{w_1,...,w_N\}$ is 
a topologically contractive IFS on 
a Hausdorff topological space $X$.
\begin{itemize}
\item[(i)] (Symbolic conjugation) 
There exists a coding map $\pi:I^\infty\to X$, 
i.e., a continuous map for which 
$w_i\circ \pi=\pi\circ \tau_i$ 
for $i\in I=\{1,...,N\}$.
\item[(ii)] The set $\pi(I^\infty)$ is 
the attractor of $\mathcal{F}$ (i.e., the 
contractive fixed point of the operator $\mathcal{F}$).
\item[(iii)] For every $\alpha\in I^\infty$ 
and $C\in{\mathcal{K}}(X)$ with $\mathcal{F}(C)\subseteq C$, the value 
$\pi(\alpha)$ is the only element of the intersection 
$\bigcap_{k\in{\mathbb{N}}}w_{\alpha_{\vert k}}(C)$.
\item[(iv)] For every $K\in{\mathcal{K}}(X)$ and 
$\alpha\in I^\infty$, the sequence 
$(w_{\alpha_{\vert k}}(K))$ converges to
$\{\pi(\alpha)\}$ with respect to the Vietoris topology.
\item[(v)] (Williams' formula) The attractor 
$\pi(I^\infty)$ is the closure of the set of all 
fixed points of maps $w_\alpha$, 
$\alpha\in I^{<\infty}$.
\item[(vi)] (Metrizability) The attractor 
$\pi(I^\infty)$ is metrizable.
\end{itemize}
\end{theorem}
\begin{proof}
We first prove (i) and (iii).
Take $C\in{\mathcal{K}}(X)$ so that $\mathcal{F}(C)\subseteq C$ 
and $\alpha\in I^\infty$. By definition, the set
$\bigcap_{k\in{\mathbb{N}}}w_{\alpha_{\vert k}}(C)$ 
is singleton. Its unique element, denoted by
$\pi(\alpha)$, does not depend on 
the choice of $C$, because if $C_1,C_2\in {\mathcal{K}}(X)$
satisfy $\mathcal{F}(C_i)\subseteq C_i$ for $i=1,2$, 
then by the compact dominance 
we can find $\tilde{C}\in{\mathcal{K}}(X)$ 
such that $C_1,C_2\subseteq \tilde{C}$ 
and $\mathcal{F}(\tilde{C})\subseteq \tilde{C}$.

Now we show that the constructed map 
$\pi:I^\infty\to X$ is continuous. 
Fix $\alpha\in I^\infty$ and 
take any open neighborhood $U$ of $\pi(\alpha)$. 
Since the sequence $(w_{\alpha_{\vert k}}(C))$ 
is a decreasing sequence of compact sets 
whose intersection contains just $\pi(\alpha)$, 
we see that there exists $k\in{\mathbb{N}}$ such that
$w_{\alpha_{\vert k}}(C)\subseteq U$. 
Then if  $\beta\in I^\infty$ agrees with 
$\alpha$ on the first $k$ coordinates, we get
\begin{equation*}
\pi(\beta)\in w_{\alpha_{\vert k}}(C)\subseteq U
\end{equation*}
and $\pi$ is continuous. 

Now take any $\alpha\in I^\infty$ and observe that
\begin{equation*}
\{w_i(\pi(\alpha))\}= 
w_i\left(\bigcap_{k\in{\mathbb{N}}} 
w_{\alpha_{\vert k}}(C)\right) \subseteq 
\bigcap_{k\in{\mathbb{N}}} w_i\circ w_{\alpha_{\vert k}}(C) = 
\{\pi(i\hat{\;}\alpha)\} =
\{\pi(\tau_i(\alpha))\},
\end{equation*}
so we get $w_i(\pi(\alpha))=\pi(\tau_i(\alpha))$ 
and the proofs of (i) and (iii) are finished.

Now we move to (ii). 
Similarly, as in Remark \ref{newremfs}, 
we can show that $\pi\left(I^\infty\right)$ 
is $\mathcal{F}$-invariant and compact. 

We will show that $A:= \pi\left(I^\infty\right)$ 
is the CFP of $\mathcal{F}$. Choose any nonempty 
and compact set $K$ and let $C\in{\mathcal{K}}(X)$ 
be such that $\mathcal{F}(C)\subseteq C$ and $K\subset C$.
Fix any nonempty and open sets 
$U_0,U_1,...,U_k\subseteq X$ with 
$A\subseteq U_0$ 
and $A\cap U_i\neq\emptyset$ for $i=1,...,k$ 
(that is, we establish a neighbourhood of $A$ 
in the Vietoris topology). 
We have to find $n_0$ such that for 
$n\geq n_0$, $\mathcal{F}^n(K)\subseteq U_0$ and
$\mathcal{F}^n(K)\cap U_i\neq\emptyset $ for $i=1,...,k$.
Given $n\in{\mathbb{N}}$, set
\[
P_n:=\{\alpha\in I^\infty: 
w_{\alpha_{\vert n}}(C) \subseteq U_0\}.
\]
Clearly, $(P_n)$ is an increasing sequence of 
open sets and (by the inclusion $A\subset U_0$ 
and definition of a~TIFS) we have that 
$I^\infty=\bigcup_{n\in{\mathbb{N}}} P_n$. 
Hence by compactness of $I^\infty$, 
there is $n_0\in{\mathbb{N}}$ such that 
$I^\infty=P_{n_0}$ and thus for $j\geq n_0$,
\begin{equation*}
\mathcal{F}^j(K)\subseteq \mathcal{F}^{n_0}(C) \subseteq
\bigcup_{\alpha\in I^{n_0}} w_\alpha(C) 
\subseteq U_0.
\end{equation*}  
Now choose $i=1,...,k$ and 
$\alpha\in I^\infty$ with 
$\pi(\alpha)\in U_i$. 
There exists $n_i\in{\mathbb{N}}$ such that 
$w_{\alpha_{\vert n}}(C)\subseteq U_i$ 
for $n\geq n_i$ and hence 
$w_{\alpha_{\vert n}}(K)\subseteq \mathcal{F}^n(K)\cap U_i$.
All in all, it is enough to take 
$k_0:= \max\{n_0,n_1,...,n_k\}$.

To see (iv), choose any 
$\alpha\in I^\infty$ and $K\in{\mathcal{K}}(X)$, 
and find $C\in{\mathcal{K}}(X)$ with $K\subseteq C$ 
and $\mathcal{F}(C)\subseteq C$. By (iii), 
$w_{\alpha_{\vert k}}(C)\to \{\pi(\alpha)\}$. 
As for any $k\in{\mathbb{N}}$, 
$w_{\alpha_{\vert k}}(K)\subseteq 
w_{\alpha_{\vert k}}(C)$, it also holds 
$w_{\alpha_{\vert k}}(K)\to \{\pi(\alpha)\}$.

Now we observe (v). 
Choose $\beta\in I^{<\infty}$ and 
let $x_\beta$ be the unique fixed point 
of $w_\beta$. Using (iv) for a compact set
$\{x_\beta\}$ and a sequence 
${\alpha}:=\beta\hat{\;}\beta\hat{\;}\beta\hat{\;}...$, 
we have that 
$w_{{\alpha}_{\vert k}}(x_\beta)\to \pi({\alpha})$. 
But the latter sequence has a constant 
subsequence whose unique element equals 
$x_\beta$, and therefore 
$x_\beta=\pi(\alpha)\in A_\mathcal{F}$. 
Now choose any $\beta\in I^\infty$ and 
any open neighbourhood $U$ of $\pi(\beta)$. 
By the continuity of $\pi$, there is $k$ 
such that if $\alpha\in I^\infty$ agrees 
with $\beta$ on the first $k$ coordinates, 
then $\pi(\alpha)\in U$. Hence taking 
$\alpha:= (\beta_{\vert k})\hat{\;}
(\beta_{\vert k})\hat{\;}...$, 
we get $\pi(\alpha)\in U$. 
By the earlier observation, $\pi(\alpha)$ 
is the fixed point of $w_{\beta_{\vert k}}$ 
and the result follows.

Finally,  the attractor $\pi(I^\infty)$ is 
metrizable as a continuous image of 
a compact metric space 
(cf. \cite{engelking} chap.3.7 p.182 
and Theorem 4.4.17) , so we get (vi).
\end{proof}

By Theorem \ref{HB2}, attractors of 
topologically contractive IFS are always 
metrizable. On the other hand, the underlying 
space $X$ for a TIFS can be nonmetrizable. 

\begin{example}[Nonmetrizable TIFS]\label{ex:nonmetrizTIFS}
Let us endow $[0,1]$ with the Euclidean topology 
and $(1,2]$ with any nonmetrizable topology.
Let $X$ be the disjoint union of 
$[0,1]$ and $(1,2]$. 
Define $w_1(x)=\frac{1}{2}\cdot\min\{1,x\}$ and
$w_2(x)=w_1(x)+\frac{1}{2}$. 
Then $w_1,w_2$ are continuous 
(as they are so on each of disjoint 
subspaces $[0,1]$ and $(1,2]$). 
For every compact $K\subseteq X$, the set 
$C:=K\cup[0,1]$ is compact, contains $K$ and 
$w_1(C)\cup w_2(C)\subseteq C$. This verifies 
point (i) of Definition \ref{def:TIFS}.
Furthermore, for every
$\alpha\in\{1,2\}^\infty$, 
the intersection
$\bigcap_{k\in{\mathbb{N}}} w_{\alpha_{\vert k}}(X)
=\bigcap_{k\in{\mathbb{N}}} w_{\alpha_{\vert k}}([0,1])$ 
is a singleton. This verifies 
point (ii) of Definition \ref{def:TIFS}.
Thus $\mathcal{F}=\{w_1,w_2\}$ is a TIFS.
\end{example}

The next two results show that 
Kameyama's topological self similar sets 
and Kieninger's point-fibred attractors 
are nothing else but attractors of TIFSs.

\begin{proposition}\label{prop:TIFSvsKameyama}
Let $X$ be a compact Hausdorff topological space and
$\mathcal{F}$ be an IFS on $X$ consisting of continuous maps.
The following conditions are equivalent:
\begin{itemize}
\item[(i)] there exists a surjective coding map 
$\pi:I^\infty\to X$ adjusted to $\mathcal{F}$ (i.e., 
$X$ is Kameyama's topological self similar set);
\item[(ii)] $\mathcal{F}$ is a topologically contractive
IFS and $X$ is its attractor.
\end{itemize}
\end{proposition}
\begin{proof}
Implication (ii)$\Rightarrow$(i) follows 
directly from Theorem \ref{HB2}. 
To see (i)$\Rightarrow$(ii), 
assume that $\mathcal{F}$ is an IFS on $X$ for which 
there exists a surjective coding map 
and choose $\alpha\in I^\infty$. 
By induction we can show that 
for every $k\in{\mathbb{N}}$, 
$w_{\alpha_{\vert k}}\circ \pi =
\pi\circ \tau_{\alpha_{\vert k}}$. 
Hence
\begin{equation*}
\bigcap_{k\in{\mathbb{N}}} w_{\alpha_{\vert k}}(X)=
\bigcap_{k\in{\mathbb{N}}} w_{\alpha_{\vert k}}(\pi(I^\infty))=
\bigcap_{k\in{\mathbb{N}}} \pi(\tau_{\alpha_{\vert k}}(I^\infty))=
\pi\left(\bigcap_{k\in{\mathbb{N}}}\tau_{\alpha_{\vert k}}(I^\infty)\right)
=\pi(\{\alpha\}),
\end{equation*}
where the penultimate equality holds as
$(\tau_{\alpha_{\vert k}}(I^\infty))$ is 
a decreasing sequence of compact sets. 
The result follows.
\end{proof}

\begin{proposition}\label{prop:TIFSvsKieninger}
Let $\mathcal{F}$ be an IFS on a Hausdorff topological 
space $X$ consistsing of continuous maps. 
The following conditions are equivalent:
\begin{itemize}
\item[(a)] $\mathcal{F}$ is a TIFS;
\item[(b)] $\mathcal{F}$ has the attractor $A_\mathcal{F}$ which is 
point-fibred in the sense of Kieninger, that
is point (ii) of Definition \ref{def:TIFS}
holds true for $C=A_\mathcal{F}$.
\end{itemize}
\end{proposition}
\begin{proof}
Implication (a)$\Rightarrow$(b) follows from 
Theorem \ref{HB2}. We will prove the opposite one.

First note that, by Remark \ref{remark:compactdominacy}, 
$\mathcal{F}$ is compactly dominated, because it has the attractor.

Now, let $\mathcal{F}=\{w_1,...,w_N\}$. Choose any
$C\in{\mathcal{K}}(X)$ so that $\mathcal{F}(C)\subseteq C$ and 
$\alpha=(\alpha_i)_{i=1}^{\infty}\in I^\infty$. 
Since $A_\mathcal{F}$ is the attractor of $\mathcal{F}$, 
we have $\mathcal{F}^k(C)\to A_\mathcal{F}$ and, in turn, 
$A_\mathcal{F}=\bigcap_{k\in{\mathbb{N}}}\mathcal{F}^k(C)$. 
Now if $k>j>1$, then 
\begin{equation*}
w_{\alpha_{\vert k}}(C)=
w_{\alpha_{\vert j}}(w_{(\alpha_{j+1},...,\alpha_k)}(C))
\subseteq w_{\alpha_{\vert j}}(\mathcal{F}^{k-j}(C)). 
\end{equation*}
Hence
\begin{equation*}
\bigcap_{k\in{\mathbb{N}}} w_{\alpha_{\vert k}}(C)=
\bigcap_{k>j} w_{\alpha_{\vert k}}(C) \subseteq
\bigcap_{k>j} w_{\alpha_{\vert j}}(\mathcal{F}^{k-j}(C))=
w_{\alpha_{\vert j}}\left(\bigcap_{k>j}\mathcal{F}^{k-j}(C)\right)=
w_{\alpha_{\vert j}}(A_\mathcal{F}).
\end{equation*}
Since $j$ was taken arbitrarily and 
$A_\mathcal{F}\subseteq C$, we get 
$\bigcap_{k\in{\mathbb{N}}}w_{\alpha_{\vert k}}(C)= 
\bigcap_{k\in{\mathbb{N}}}w_{\alpha_{\vert k}}(A_\mathcal{F})$.
As $A_\mathcal{F}$ is a point fibred attractor, 
the proof is finished.
\end{proof}

Finally, we show that Browder contractive IFSs 
on complete spaces are topologically contractive, 
whence the thesis of Theorem \ref{HB2} 
holds for such IFSs. 

\begin{corollary}\label{cor:tifs}
Let $\mathcal{F}$ be a Browder contractive IFS on a complete
metric space $X$. Then $\mathcal{F}$ is a~TIFS.
\end{corollary}
\begin{proof}
Condition (i) of Definition \ref{def:TIFS} 
follows, due to Remark \ref{remark:compactdominacy}, 
from the existence of an attractor of $\mathcal{F}$.

Now, if $C\in{\mathcal{K}}(X)$ satisfies $\mathcal{F}(C)\subseteq C$ 
and $\alpha\in I^\infty$, then the sequence
$(w_{\alpha_{\vert k}}(C))$ is a decreasing 
sequence of compact sets such that 
$\operatorname{diam}(w_{\alpha_{\vert k}}(C))\leq
\varphi^k(\operatorname{diam}(C))$ for all $k\in{\mathbb{N}}$, 
where $\varphi$ is a~comparison function 
common for all maps comprising $\mathcal{F}$. 
Since $\varphi^k(\operatorname{diam}(C))\to 0$, 
the Cantor theorem assures that the intersection
$\bigcap_{k\in{\mathbb{N}}}w_{\alpha_{\vert k}}(C)$ 
is a singleton. Thus condition (ii) 
of Definition \ref{def:TIFS} is satisfied.
\end{proof}

\begin{remark}
It is worth pointing out that although 
Theorem \ref{HB2} implies Theorem \ref{th1}
(point (ii) states that $\pi(I^\infty)$ 
is the CFP of $\mathcal{F}$), we needed 
to establish Theorem \ref{th1} directly. 
Theorem \ref{th1} is employed to justify 
the compact dominance of $\mathcal{F}$ in the course
of proving Theorem \ref{HB2}. 
If the space $X$ was already compact,
we could have omitted such a detour.
\end{remark}

\subsection{Existence of an invariant measure}

The next result extends Theorem \ref{th:weakcontrMarkov}
to the setting of TIFSs. 
Let $(\mathcal{F},\vec{p})$ be a probabilistic IFS 
consisting of continuous maps acting
on a Hausdorff topological space $X$.
Let $M_{(\mathcal{F},\vec{p})}$ be the Markov operator
corresponding to $(\mathcal{F},\vec{p})$.
Let $\mathcal{P}(X)$ be the set of 
all Radon probability measures on $X$. 
By Proposition \ref{th:wtransport},
$M_{(\mathcal{F},\vec{p})}(\mu)\in\mathcal{P}(X)$
for every $\mu\in\mathcal{P}(X)$.
Thus we can consider the Markov operator
restricted to $\mathcal{P}(X)$.

\begin{theorem}\label{theorem:invariant}
Let $(\mathcal{F},\vec{p})$ be a probabilistic 
topologically contractive IFS on a 
Hausdorff topological space $X$. 
Then 
\begin{enumerate}
\item[(i)] the Markov operator 
$M:\mathcal{P}(X)\to\mathcal{P}(X)$ 
corresponding to $(\mathcal{F},\vec{p})$
has the (not necessarily unique!) 
CFP $\mu_{*}$, i.e.,
$M(\mu_{*})=\mu_{*}$ is an invariant
Radon probability measure and
$M^n(\mu)$ converges weakly to $\mu_{*}$
for all $\mu\in\mathcal{P}(X)$;
\item[(ii)] $\mu_{*}$ supports the attractor
$A_{\mathcal{F}}$ of $\mathcal{F}$, i.e., 
$\operatorname{supp}(\mu_{*})=A_{\mathcal{F}}$;
\item[(iii)] $\mu_{*}$ is unique up to 
Radon probability measures with 
compact supports and, 
if additionally $X$ is normal, 
$\mu_{*}$ is unique up to all 
Radon probability measures.
\end{enumerate}
\end{theorem}
\begin{proof}
\textbf{Step 1.}
Existence of an invariant measure 
for the canonical IFS on a code space.

Set $I:=\{1,...,N\}$ and let $\mu_{b}$ 
be the Bernoulli measure on $I^\infty$
corresponding to the vector of weights
$\vec{p}= (p_1,...,p_{N})$;
cf. Appendix, Section \ref{section:measures}).
This means that $\mu_{b}$ is 
the unique Borel measure such that 
for every cylinder 
$A_{(\alpha_1,...,\alpha_k)} := 
\{(\alpha_1,...,\alpha_k)\}\times I^\infty$, 
$\alpha_1,...,\alpha_k\in I$, the following holds
\begin{equation*}
\mu_{b}(A_{(\alpha_1,...,\alpha_k)}) 
= p_{\alpha_1}\cdots p_{\alpha_k}.
\end{equation*}
Letting $M_{(\mathcal{T},\vec{p})}$ to be 
the Markov operator induced by
a probabilistic canonical IFS 
$(\mathcal{T},\vec{p})$ on $I^\infty$, 
for every 
$(\alpha_1,...,\alpha_k)\in I^k$, $k\geq 2$, 
we have
\begin{eqnarray*}
M_{(\mathcal{T},\vec{p})} (\mu_{b}(A_{(\alpha_1,...,\alpha_k)})) = 
\sum_{i=1}^{N} p_i\cdot \mu_{b}(\tau_i^{-1}(A_{(\alpha_1,...,\alpha_k)}) =
\\
= p_{\alpha_1}\cdot \mu_{b}(A_{(\alpha_2,...,\alpha_k)})=
\mu_{b}(A_{(\alpha_1,...,\alpha_k)}).
\end{eqnarray*}
This implies that 
$\mu_{b} = M_{(\mathcal{T},\vec{p})}(\mu_{b})$. 

\textbf{Step 2.} Existence of an invariant measure
for arbitrary TIFS.

Fix a probabilistic TIFS $(\mathcal{F},\vec{p})$, 
where $\mathcal{F}=\{w_1,...,w_N\}$.
Define the measure 
$\mu_{*}:= \mu_{b}\circ \pi^{-1}$, 
where $\pi:I^{\infty}\to X$ is the coding map
and $\mu_{b}$ is the Bernoulli measure 
on the code space. 
As $\mu_{b}$ is Radon, so is $\mu_{*}$ 
(see Proposition \ref{th:wtransport}). 
By the invariance of $\mu_{b}$ established 
in Step 1 and the definition of $\pi$, 
for every Borel $B\subseteq X$ the following holds 
\begin{eqnarray*}
M(\mu_{*})(B)=
\sum_{i=1}^{N} p_i\cdot\mu_{*}(w_i^{-1}(B))=
\sum_{i=1}^{N} p_i\cdot 
\mu_{b}(\pi^{-1}(w_i^{-1}(B)))=
\\
= \sum_{i=1}^{N} p_i\cdot 
\mu_{b}(\tau_i^{-1}(\pi^{-1}(B))) =
\mu_{b}(\pi^{-1}(B))=\mu_{*}(B).
\end{eqnarray*}
Thus $\mu_{*}$ is an invariant Radon 
probability measure for $(\mathcal{F},\vec{p})$. 
Moreover, Proposition \ref{prop:support}(c) 
together with Theorem \ref{HB2}(ii) 
imply that 
\[
\operatorname{supp}(\mu_{*}) =
\overline{\pi(\operatorname{supp}{(\mu_{b})})} = 
\pi(I^\infty)= A_{\mathcal{F}}.
\]

\textbf{Step 3.} The measure $\mu_{*}$ is CFP.

We show that the invariant measure $\mu_{*}$ 
defined in Step 2 is the CFP of the Markov 
operator $M:\mathcal{P}(X)\to \mathcal{P}(X)$ 
corresponding to $(\mathcal{F},\vec{p})$. 
Take any measures $\mu,\eta\in\mathcal{P}(X)$ 
and a continuous and bounded function $g:X\to{\mathbb{R}}$. 
It is enough to show that
\begin{equation}\label{inv1}
\lim_{n\to\infty}
\left|\int_{X} g \;dM^n(\mu) -
\int_{X} g \;dM^n(\eta)\right| =0.
\end{equation}
Indeed, choosing $\mu_{*}$ for $\eta$, 
and using its invariance, 
we see that \eqref{inv1} implies that 
$M^n(\mu)\to \mu_{*}$ weakly. 

Fix $\varepsilon>0$. 
Since $\mu,\eta$ are Radon, 
there is a compact set $C$ such that 
$\mu(X\setminus C), 
\eta(X\setminus C) <
\frac{\varepsilon}{8 P}$, 
where $P:=\sup\{|g(x)|:x\in X\}+1$. 
Switching if needed to some bigger compact set, 
we can assume that $\mathcal{F}(C)\subseteq C$.
As $g$ is continuous at each point of $C$, 
we can find a finite open cover $U_1,...,U_m$ of $C$ 
such that for every $i=1,...,m$ and $x,y\in U_i$, 
we have $|g(x)-g(y)|<\frac{\varepsilon}{4}$. 
Now if $k\in{\mathbb{N}}$, then we set 
\begin{equation*}
A_k:= \{\alpha\in I^\infty: 
w_{\alpha_{\vert{k}}}(C)\subseteq U_i 
\mbox{ for some } i=1,...,m\}.
\end{equation*}
Clearly, each $A_k$ is open and, by properties 
of a TIFS, the family $A_k$, $k\in{\mathbb{N}}$, 
is a cover of $I^\infty$. 
By compactness of $I^{\infty}$ and the fact that 
the sequence $(A_k)$ is increasing, 
we can find $k_0\in{\mathbb{N}}$ such that $I^\infty=A_{k}$ 
for $k\geq k_0$.  
Now let $\alpha\in I^k$ and find $i=1,...,m$ 
so that $w_i(C)\subseteq U_i$. 
Choosing any $x_0\in U_i$, we have
\begin{eqnarray*}
\left|\int_{X} g\circ w_\alpha\;d\mu -
\int_{X} g\circ w_\alpha\;d\eta\right| \leq 
\left|\int_{X} g\circ w_\alpha\;d\mu-g(x_0)\right| +
\left|g(x_0)-\int_{X} g\circ w_\alpha\;d\eta\right| =
\\
= \left|\int_{X}g\circ w_\alpha\;d\mu - 
\int_{X} g(x_0)\;d\mu\right| + 
\left|\int_{X} g(x_0)\;d\eta - 
\int_{X} g\circ w_\alpha\;d\eta\right| \leq 
\\
\leq \int_{X} |g\circ w_\alpha-g(x_0)| \;d\mu +
\int_{X} |g(x_0)-g\circ w_\alpha|\;d\eta\leq 
\\
\leq \int_{X\setminus C} 2{ P}\;d\mu +
\int_{C} \frac{\varepsilon}{4}\;d\mu +
\int_{X\setminus C} 2{ P}\;d\eta +
\int_{C} \frac{\varepsilon}{4}\;d\eta\leq \varepsilon.
\end{eqnarray*}
Hence by formula \eqref{feller} in Remark 
\ref{rem:federed} and an easy inductive argument:
\begin{eqnarray*}
\left|\int_{X} g\;dM^n(\mu) - 
\int_{X} g\;dM^n(\eta)\right|=
\left|\sum_{\alpha\in I^k} p_\alpha
\int_{X} g\circ w_\alpha\;d\mu - 
\sum_{\alpha\in I^k} p_\alpha 
\int_{X} g\circ w_\alpha\;d\eta\right| \leq 
\\
\leq \sum_{\alpha\in I^k} p_\alpha 
\left|\int_{X} g\circ w_\alpha\;d\mu - 
\int_{X} g\circ w_\alpha\;d\eta\right| <\varepsilon,
\end{eqnarray*}
where $p_\alpha=p_{\alpha_1}\cdots p_{\alpha_k}$, 
$\alpha=(\alpha_1,...,\alpha_k)\in I^k$.

\textbf{Step 4.} Uniqueness of ${ \mu_*}$.

If $X$ is normal, then the weak topology 
on $\mathcal{P}(X)$ is Hausdorff (see Appendix,
Proposition \ref{prop:WeakTopIsHausdorff}).
Therefore, the CFP $\mu_{*}$ of 
$M_{(\mathcal{F},\vec{p})}$ is necessarily unique. 

Now let us consider the case of a general, 
not necessarily normal space $X$. 
Assume that $\mu\neq\eta$ are two invariant measures
with compact supports. 
By the compact dominance of $\mathcal{F}$, 
we can find a~compact set $C\subseteq X$ 
so that $\mathcal{F}(C)\subset C$ and 
\begin{equation}\label{eq:supportsinC}
\operatorname{supp}(\mu) \cup \operatorname{supp}(\eta)\subseteq C.
\end{equation}
Let $\mu_{C},\eta_{C}$ be the restrictions 
of $\mu,\eta$, respectively, 
to the Borel $\sigma$-algebra $\mathcal{B}(C)$. 
Consider the restriction of $\mathcal{F}$ to $C$,
$\mathcal{F}\vert C:=\{w_{1}\vert C,...,w_{N}\vert C\}$
and the induced Markov operator 
$M_{(\mathcal{F}{\vert C},\vec{p})}$.
Then for every $B\in \mathcal{B}(C)$,
\begin{eqnarray*}
M_{(\mathcal{F}{\vert C},\vec{p})}(\mu_{C})(B) =
\sum_{i=1}^{N} p_i\cdot 
\mu_C({(w_{i}\vert C)}^{-1}(B)) = 
\sum_{i=1}^{N} p_i\cdot 
\mu(C\cap w_{i}^{-1}(B)) =
\\
= \sum_{i=1}^{N} p_i\cdot 
\mu(w_{i}^{-1}(B)) =\mu(B)=\mu_C(B),
\end{eqnarray*}
because of Lemma \ref{lem:Radonsupport}(ii),
inclusion \eqref{eq:supportsinC} 
and simple set-algebra
\[
(w_{i}\vert C)^{-1}(B) = C\cap w_{i}^{-1}(B).
\]
This means that $\mu_C$ is (clearly Radon) 
invariant measure for $\mathcal{F}\vert C$. 
By symmetry, $\eta_C$ is also 
an invariant measure for $\mathcal{F}\vert C$. 
It remains to show that $\mu_C\neq \eta_C$ 
as this will give a contradiction 
with the ``normal'' case. 
By our assumption that $\mu\neq\eta$, 
there is a Borel set $B\subseteq X$ 
such that $\mu(B)\neq\eta(B)$.  
Thanks to Lemma \ref{lem:Radonsupport}(ii) 
and \eqref{eq:supportsinC}, we have
\begin{equation*}
\mu_C(B\cap C) = \mu(B\cap C) = 
\mu(B\cap C)+\mu(B\setminus C) = 
\mu(B)\neq \eta(B)=\eta_C(B\cap C).
\end{equation*} 
Hence $\mu_C\neq \eta_C$ and the result follows.
\end{proof}

\begin{remark}
For a metric version of Theorem 
\ref{theorem:invariant} see 
\cite{Mendivil} Theorem 2.67.
\end{remark}

\subsection{Remetrization of topologically contractive IFSs}\label{sec:remetrizeTIFS}

Turning back to remetrization Theorem 
\ref{theorem:meyers}, it is natural 
to ask whether, having a TIFS $\mathcal{F}$ on 
a metrizable space $X$, there exists an 
admissisble complete metric $d$ in $X$ 
making $\mathcal{F}$ a Banach contractive IFS. 
Such a question for topological self similar 
sets was considered by Kameyama in \cite{Ka2}. 
He proved that 
given a topological self similar 
set $X$ and an appropriate IFS $\mathcal{F}$,
we can always define 
a family of pseudometrics $d_\mathcal{F}^a$, 
for $a=(a_1,...,a_N)\in(0,1)^N$, with the 
property that $\operatorname{Lip}_{d_\mathcal{F}^a}(w_i)\leq a_i$ 
for $i=1,...,N$, and the following conditions 
are equivalent:
\begin{itemize}
\item[(K-1)] $d_\mathcal{F}^a$ is admissible metric 
for some $a\in(0,1)^N$;
\item[(K-2)] there is an admissible metric 
$d$ on $X$ such that $\operatorname{Lip}_d(w_i)<1$ 
for every $i=1,...,N$.
\end{itemize}

Kameyama also gave an example of a topological 
self similar system for which condition (K-1) 
does not hold (\cite{Ka2} Section 1.4). 
Hence the answer to his question is negative.
However, unexpectedly, it turns out 
that things are different if we do not 
insist that the remetrized IFS has 
to comprise Banach contractions 
and we agree to use weak contractions in place
of Banach ones 
(below $w_\varnothing=\operatorname{id}_X$ 
stands for the identity function).

\begin{theorem}\label{metrizability}
Let $\mathcal{F}=\{w_1,...,w_N\}$ be a 
topologically contractive IFS on 
a metric space $(X,d)$. Let
$(a_n)_{n=0}^{\infty}$ be a strictly increasing 
sequence of reals such that $1\leq a_n\leq 2$ 
for all $n\geq 0$. Define
\begin{equation}\label{rem}
\hat{d}(x,y) := \max\{a_k d(w_\alpha(x), 
w_\alpha(y)):k=0,1,2,...,\;\alpha\in I^k\}
\mbox{ for } x,y\in X.
\end{equation}
Then
\begin{itemize}
\item[(i)] $\hat{d}$ is admissible metric on $X$;
\item[(ii)] $\hat{d}$ is complete provided $d$ is complete;
\item[(iii)] $\mathcal{F}$ is Edelstein contractive with respect to $\hat{d}$;
\item[(iv)] $\mathcal{F}$ is Rakotch contractive with respect to
$\hat{d}$ provided that, additionally, 
for every $\alpha\in I^\infty$, the set 
$\bigcap_{k\in{\mathbb{N}}} w_{\alpha_{\vert k}}(X)$ 
is a singleton. 
\end{itemize}
\end{theorem}

Before we give a proof, let us observe 
that the above result yields a natural
characterization of topological self 
similar sets. 

\begin{corollary}\label{cor:selfsimilar}
A topological space $X$ is a topological 
self similar set if and only if $X$ is 
homeomorphic to the attractor of 
some weakly contractive IFS.
\end{corollary}

\begin{proof}(of Theorem \ref{metrizability})
We first observe that $\hat{d}$ is well defined. 
Take distinct points $x,y\in X$. By definition, 
we can find a compact set $C$ containing $x,y$ 
and such that $\mathcal{F}(C)\subseteq C$. 
Then, using similar reasonings to those 
in the proof of Theorem \ref{theorem:invariant}, 
we can find $k_0$ such that for every $k\geq k_0$ 
and $\alpha\in I^k$, it holds 
$\operatorname{diam}(w_{\alpha}(C))<\frac{1}{3}d(x,y)$. 
(Indeed, we just have to put 
$A_k:=\{\alpha\in I^\infty: 
\operatorname{diam}(w_{\alpha_{\vert k}}(C)) 
< \frac{1}{3}d(x,y)\}$). 
Hence for every $k\geq k_0$ and $\alpha\in I^k$, 
it holds 
\begin{equation*}
a_k d(w_\alpha(x),w_\alpha(y))\leq 
a_k \operatorname{diam}(w_\alpha(C))<\frac{2}{3}d(x,y)\leq 
\frac{2}{3} \hat{d}(x,y)<\hat{d}(x,y).
\end{equation*}
Thus $\hat{d}$ is well defined.

Now we will prove (i). It is easy to see that 
$\hat{d}$ is a metric. Symmetry and triangle 
inequality are immediate. A less trivial implication 
$\hat{d}(x,y)=0\;\Rightarrow\;x=y$ 
follows from the inequality $d\leq\hat{d}$. 
From the same inequality we see that 
to prove admissibility of $\hat{d}$, 
it is enough to show that if 
$d(x_n,x)\to 0$, then $\hat{d}(x_n,x)\to 0$. 
Fix any $\varepsilon>0$, and choose $k_0$ 
such that for any $k\geq k_0$, $n\in{\mathbb{N}}$ and 
$\alpha\in I^k$, we have 
$d(w_\alpha(x_n),w_\alpha(x))<\frac{1}{2}\varepsilon$ 
(this can be done in a similar way as earlier; 
observe that $\{x_n:n\in{\mathbb{N}}\}\cup\{x\}$ is compact). 
Now since the family 
$\{w_\alpha:\alpha\in I^k,\;k\leq k_0\}$ 
is finite and consists of continuous functions, 
we can find $k_1\in{\mathbb{N}}$ with the property that
$d(w_\alpha(x_n),w_\alpha(x))<\frac{1}{2}\varepsilon$ 
for all $\alpha\in I^k$, $k\leq k_0$ and $n\geq k_1$. 
Hence for every $n\geq k_1$, 
$\hat{d}(x_n,x)< \varepsilon$ which proves 
that $\hat{d}(x_n,x)\to 0$. 

Now we move to (ii). Assume that $d$ is complete 
and let $(x_n)$ be a $\hat{d}$-Cauchy sequence. 
As $d\leq\hat{d}$, it is also $d$-Cauchy, 
hence $d$-convergent. Since $d$ and $\hat{d}$ 
are equivalent, the sequence $(x_n)$ 
is also $\hat{d}$-convergent.

We postpone (iii) for a moment and jump to prove (iv). 
We will show that each $w_i$ satisfies 
condition (iii) from Lemma \ref{prop1}. 
Take $\delta>0$ and choose $k_0\in{\mathbb{N}}$ 
such that for all $k\geq k_0$ and $\alpha\in I^k$, 
the following inequality holds: 
$\operatorname{diam}(w_\alpha(X))<\frac{1}{4}\delta$. 
Set 
\begin{equation*}
\lambda:= 
\max\left\{\frac{a_{k}}{a_{k+1}}: 
k\leq k_0\right\} <1.
\end{equation*}
Now choose $x,y\in X$ with $\hat{d}(x,y)\geq \delta$ 
and $i=1,...,N$. For $\alpha\in I^k$, 
we see that if $k\geq k_0$, then
\begin{equation*}
a_{k} d(w_\alpha(w_i(x)),w_\alpha(w_i(x))) \leq 
2\cdot\frac{1}{4}\delta \leq \frac{1}{2}\hat{d}(x,y),
\end{equation*}
and if $k\leq k_0$, then 
\begin{equation*}
a_{k} d(w_\alpha(w_i(x)),w_\alpha(w_i(y)))=
\frac{a_k}{a_{k+1}}\cdot a_{k+1} 
d(w_\alpha(w_i(x)),w_\alpha(w_i(y))) \leq
\frac{a_k}{a_{k+1}} \hat{d}(x,y)\leq 
\lambda \hat{d}(x,y).
\end{equation*}
Therefore
\begin{equation*}
\hat{d}(w_i(x),w_i(y))\leq 
\max\left\{\frac{1}{2},\lambda\right\} 
\hat{d}(x,y).
\end{equation*}
Overall $w_i$ is a Rakotch contraction with respect 
to $\hat{d}$. This gives (iv).

Finally we show (iii). Let $x,y\in X$, $x\neq y$. 
By compact dominance of $\mathcal{F}$ we can find 
a compact set $C$ such that $x,y\in C$ 
and $\mathcal{F}(C)\subseteq C$. 
Then using (iv) for the IFS 
$\mathcal{F}\vert C=\{w_{1}\vert C,...,w_{N}\vert C\}$ 
we deduce that 
$\hat{d}(w_i(x),w_i(y))<\hat{d}(x,y)$.
The result follows.
\end{proof}

\begin{remark}
The above remetrization theorem shows that 
Theorem \ref{HB2} for completely metrizable spaces 
can be deduced from its version for 
Rakotch contractive IFSs 
(which is given for example in \cite{Ha}).
Indeed, although IFSs consisting of Edelstein
contractions may fail to generate attractors, 
we can always switch to some suitably 
large compact set $C$ with $\mathcal{F}(C)\subseteq C$ 
and use the fact that an Edelstein contaction 
on a compact space is always Rakotch. 
\end{remark}

Formula \eqref{rem} was conceived by 
R. Miculescu and A. Mihail in \cite{MM3}.
They employed it to metrize Kameyama's
topological self similar sets. Their 
technique of proof required some additional 
properties of the sequence $(a_n)$. 
In \cite{BKNNS}, Banakh et al. considered 
also several more restrictive 
topological-type contractive conditions 
on $\mathcal{F}$ and proved that $\hat{d}$ has 
better properties in such cases. 
For example, IFSs which satisfy 
the sufficient condition in point (iv) 
of Theorem \ref{metrizability} 
are called \emph{globally contractive}. 
Moreover, they extended the whole discussion 
to Tikhonov (a.k.a completely regular) spaces. 
It is known that the Tikhonov topology 
is generated by appropriate family of pseudometrics, 
say $\mathcal{D}$, and Banakh et.al. obtained 
counterpart of Theorem \ref{th1} 
for IFSs consisting of weakly contractive maps with respect to 
pseudometrics from $\mathcal{D}$. 
Additionally they proved that the formula \eqref{rem} 
can be successfully adjusted to that general 
pseudometric setting and hence Theorem \ref{HB2} 
(at least the existance of the CFP) can be explained 
by its pseudometric version.
Other remetrization results in this spirit can be found 
for example in \cite{MM4}.

\section{Baire genericity of attractors}

A natural question arises whether 
the class of attractors of weakly contractive IFSs 
is essentially wider than the class of attractors 
of Banach contractive ones. 
Since attractors of Banach contractive IFSs 
have necessarily finite Hausdorff dimension, 
it is enough to find a weakly contractive IFS 
which admits an infinite dimensional attractor. 
As shown by D. Dumitru \cite{Du}, 
each Peano continuum $X$ with a free arc 
(i.e., with an arc which is open) is 
the attractor of some TIFS. 
Hence, by Corollary \ref{cor:selfsimilar}, 
$X$ is homeomorphic to the attractor 
of some weakly contractive IFS. 
On the other hand, if $X$ has also 
infinite topological dimension, 
then it has infinite Hausdorff dimension 
with respect to any admissible metric, 
so it is not homeomorphic to an attractor of 
a Banach contractive IFS.

Interestingly, one can look for suitable examples 
in the finite dimension as we do below.

\begin{example}[\cite{NFM}]
An arc $S$ on the (complex) plane, 
called the \emph{snake} and defined as follows
\begin{equation*}
S:= \{0\}\cup
\left\{\frac{1}{n}e^{i\alpha}: n\in{\mathbb{N}}, \;\alpha\in 
\left(\frac{\pi}{2},2\pi\right)\right\} \cup 
\left\{re^{i\alpha}: n\in{\mathbb{N}}, 
\;r\in\left[\frac{1}{n+1},\frac{1}{n}\right],
\;\alpha=\frac{\pi}{2}(n\operatorname{mod}2)\right\},
\end{equation*}
is the attractor of some weakly contractive IFS 
but it is not an attractor of any Banach 
contractive IFS. (The latter observation follows 
from a nice criterion given in \cite{San}.) 
\end{example}

Although many continua turn out to be attractors 
of IFSs if we allow to employ weak contractions 
instead of restricting ourselves to 
Banach contractions, not all compact sets 
are attractors of weakly contractive IFSs. 
Indeed, every connected attractor must 
be locally connected (see 
Theorem \ref{th:HataConnectedness} or \cite{Ha}). 
Even more, attractors of weakly contractive IFSs
are exceptional sets among compacta.

\begin{theorem}\label{theorem:Baire}
Let $X$ be a complete separable metric space 
without isolated points (a.k.a. perfect Polish space).
Typical nonempty and compact subset 
of $X$ is not an attractor of any weakly
contractive iterated function system. Formally,
the collection of attractors of weakly contractive
IFSs on $X$ is of the first category 
in the hyperspace ${\mathcal{K}}(X)$ of compacta.
\end{theorem}
\begin{proof}
\textbf{Step 1.} Construction of nowhere dense sets.

For any nonempty open set $D\subset X$, define
\begin{equation*}
C_{D}:=\{K\in{\mathcal{K}}(X):\emptyset\neq K\cap D 
\subseteq f(K\setminus D)\; \mbox{ for some } 
f:K\to X\; \mbox{ with }\operatorname{Lip}(f)\leq 1\}.
\end{equation*}
We will show that $C_{D}$ is 
nowhere dense in ${\mathcal{K}}(X)$. 

Take $r>0$ and $K\in C_{D}$. 
Since $X$ has no isolated points and $D$ is open, 
we can find a finite set $K'\subseteq X$ 
such that $d_H(K',K)<\frac{r}{2}$, 
$D\cap K'=\{x_1,...,x_k\}$ and 
$K'\setminus D=\{y_1,...,y_j\}$ for some distinct 
points $x_1,...,x_k,y_1,...,y_j$ with $k>j$. 
Set 
\begin{eqnarray*}
\varepsilon:= 
\min\left\{r,\delta,\min\{d(x_i,x_l):
i\neq l\}\right\}>0,
\\
\varepsilon':=\frac{\varepsilon}{4},
\end{eqnarray*}
where $\delta>0$ is chosen so that 
$B(x_i,\delta)\subseteq D$ for $i=1,...,k$. 
It remains to prove that 
\begin{equation*}
B_H(K',{\varepsilon}') \subseteq 
B_H(K,r)\setminus C_{D},
\end{equation*} 
where $B_H(\cdot,\cdot)$ stands for an open ball 
in the hyperspace $({\mathcal{K}}(X),d_{H})$.

Clearly, if $S\in {\mathcal{K}}(X)$ satisfies 
$d_H(S,K')<{\varepsilon}'$,
then $d_H(S,K)\leq d_H(S,K')+d_H(K',K)<r$. 
Further we will see that also $S\notin C_{D}$. 

First observe that 
$S\setminus D\subseteq 
\bigcup_{i=1}^{j} B(y_i,\varepsilon')$. 
Indeed, take $x\in S\setminus D$. 
If $x\in B(x_i,\varepsilon')$ for some $i=1,...,k$, 
then $x\in D$ by the choice of $\varepsilon'$ 
and $\delta$, which is a contradiction. 
Thus $x\in B(y_i,\varepsilon')$ for some $i$ 
and we are done. 

Next, take any $f:S\to X$ with $\operatorname{Lip}(f)\leq 1$ 
and for each $i=1,...,k$, find a point 
$\tilde{x}_i\in S$ such that 
$d(x_i,\tilde{x}_i)<{\varepsilon}'$. 
In particular $\tilde{x}_1,...,\tilde{x}_k\in D$.
Moreover, for every $i\neq l$
\begin{equation*}
d(\tilde{x}_i,\tilde{x}_l) > 
d(x_i,x_l) - 2{\varepsilon}'\geq 2{\varepsilon'}.
\end{equation*}
Since $\{\tilde{x}_1,...,\tilde{x}_k\}\subseteq S\cap D$,
$S\setminus D\subseteq \bigcup_{i=1}^{j} 
B(y_i,\varepsilon')$, and $k>j$, 
it is enough to show that each 
$f(B(y_i,\varepsilon')\cap S)$ 
can contain at most one $\tilde{x}_l$. 
Supposing this is not the case, 
we would get 
\begin{equation*}
2\varepsilon'< d(\tilde{x}_i,\tilde{x}_l)\leq 
\operatorname{diam}(f(B(y_n,{\varepsilon}')\cap S))\leq
\operatorname{diam}(B(y_n,{\varepsilon}'))\leq 2{\varepsilon}'
\end{equation*}
for some $i\neq l$ and $n$, which is a contradiction. 
All in all, $S\cap D$ is not contained in 
$f(S\setminus D)$ and $S\notin C_{D}$. 
Thus we have proved that $C_{D}$ is nowhere dense in ${\mathcal{K}}(X)$.
  
\textbf{Step 2.} Construction of a residual set $\mathcal{A}$.

Let us fix a countable basis $\mathcal{B}$ 
of the topology of the space $X$ and 
consider the collection of sets 
$\{C_D: D\in\mathcal{B}\}$ as defined in Step 1.
Define 
\begin{equation*}
\mathcal{A}:= \operatorname{Perf}(X)\setminus 
\bigcup_{D\in \mathcal{B}} C_D,
\end{equation*}
where $\operatorname{Perf}(X)$ is the collection of 
all nonempty compact subsets of $X$
which are perfect (i.e., without isolated points). 
By Lemma \ref{lemma:cantor} and the fact that 
each Cantor space has no isolated points, 
the set $\operatorname{Perf}(X)$ is residual in ${\mathcal{K}}(X)$. 
Hence $\mathcal{A}$ is also residual.

\textbf{Step 3.} Nonattractors.

In the final part of the proof we establish 
that every element of the collection $\mathcal{A}$ 
defined in Step 2, is not an attractor 
of any weakly contractive IFS.  
To do so, we will prove that for any 
$K\in\mathcal{A}$ and for any Edelstein 
(hence Rakotch) contraction $f:K\to K$, 
the image $f(K)$ is nowhere dense in $K$. 
Actually, it suffices to show that $f(K)$,
being closed in $K$, has empty interior in $K$. 

Choose any open set $U\subset X$ 
so that $U\cap K\neq \emptyset$. 
Then $\operatorname{diam}(U\cap K)>0$.
By Lemma \ref{prop1}, we can find 
$\lambda<1$ so that
$d(f(x),f(y))\leq \lambda d(x,y)$ whenever 
$d(x,y)\geq \delta:=\frac{1}{2}\operatorname{diam}(U\cap K)$. 
Hence
\begin{equation*}
\operatorname{diam}(f(K\cap U))\leq 
\max\{\delta,\lambda\operatorname{diam}(K\cap U)\} 
<\operatorname{diam}(K\cap U).
\end{equation*}
Thus there exists $D\in\mathcal{B}$ such that
$\emptyset \neq D\cap K\subseteq 
(K\cap U)\setminus f(K\cap U)$. 
By definition, $K\notin C_{D}$, so we have that
$(K\cap D)\nsubseteq f(K\setminus D)$. 
Therefore
\begin{eqnarray*}
(K\cap U)\setminus f(K)= [(K\cap U)\setminus
f(K\cap U)]\cap[(K\cap U)\setminus 
f(K\setminus U)]\supseteq 
\\
\supseteq (K\cap D)\cap [(K\cap D)
\setminus f(K\setminus D)]=
(K\cap D)\setminus f(K\setminus D)
\neq\emptyset
\end{eqnarray*}
and hence $f(K)$ has empty interior. 
The result follows.
\end{proof}

\begin{remark}
The proof of Theorem \ref{theorem:Baire} 
given above extends some ideas from 
D'Aniello and Steele's paper\cite{DS2}, 
where it was assumed that $X$ is 
a finite dimensional unit cube.
\end{remark}

\begin{remark}
An alternative proof of 
Theorem \ref{theorem:Baire} was given 
by Balka and M\'{a}th\'{e} in \cite{BaM}. 
They defined the so-called \emph{balanced} 
sets and proved that typical sets 
are balanced while  a balanced set cannot 
be an attractor of a weakly contractive IFS. 
Moreover, they extended the thesis of 
Theorem \ref{theorem:Baire} by showing 
that if $X$ is separable and complete, 
then typical nonempty and compact set 
is either finite or it is not an attractor 
of any weakly contractive IFS. 
Finally, let us note that to learn 
how much prevalent are fractals 
(attractors with possibly fractional dimension), 
one is led to study typical dimension 
of sets and measures, 
e.g.,  \cite{BaM}, \cite{DS2}, \cite{DS3} and \cite{Myjak}.
\end{remark}

Theorem \ref{theorem:Baire} stated that it is rare
for a set to be an attractor. Quite opposite
is true when we demand only that the set at hand
is homeomorphic to an attractor. 
Lemma \ref{lemma:cantor} implies that 
typical subset of a complete metric space 
without isolated point is homeomorphic 
to the Cantor ternary set, a~classical example 
of an attractor of a Banach contractive IFS. 
We state this observation as a~separate theorem.

\begin{theorem}
Typical nonempty and compact
subset of a complete perfect metric space 
is homeomorphic to the attractor of 
a Banach contractive iterated function system. 
\end{theorem}

For a list of examples that distinguish 
different classes of IFS attractors 
we refer the reader to the survey paper \cite{NFM} 
on counterexamples in the IFS theory.

\section{Deterministic chaos game for weakly contractive IFSs}\label{sec:ChaosGame}

In the present section we exhibit a simple 
proof of the deterministic version of 
the chaos game algorithm.
The main tool is the coding map. 
Derandomization of the algorithm is due to 
contractivity and employment of algorithmically
random sequences --- disjunctive sequences.

\begin{definition}(\cite{CaludeStaiger}).
Let $I$ be a finite set of symbols (alphabet).
A sequence 
$\sigma= (\sigma_n)_{n=1}^{\infty}\in I^{\infty}$ 
is called \emph{disjunctive}, whenever 
each finite word $\alpha\in I^k$, $k\geq 1$, 
appears in $\sigma$, i.e.,
$(\sigma_{n},...,\sigma_{n+k-1}) = \alpha$ 
for some $n\geq 1$.
\end{definition} 

\begin{example}(Champernowne sequence).
Let $I=\{1,2,...,N\}$, $1<2<...<N$, $N\geq 1$.
The following sequence is disjunctive:
\[
1,2,...,N,\;1,1,1,2,...,1,N,\,2,1,2,2,...,2,N,
...,N,1,N,2,...,N,N,...
\]
It is created by writing in the lexicographic order 
first all symbols from $I$, 
then all $2$-letter words over $I$, 
then $3$-letter words etc.
This simple sequence is very regular in 
a probabilistic manner --- it is Borel normal 
(in base $N$),
cf. \cite{Knill} Corollary 2.9.2.
\end{example}

A disjunctive sequence is random in the qualitative manner:
it contains all possible finite sequences.
A disjunctive sequence is also chaotic, because:
(i) it is not almost periodic (a.k.a. uniformly recurrent), 
cf. \cite{AlmostPeriodic};
(ii) the Bernoulli shift 
$\vartheta: I^{\infty}\to I^{\infty}$,
$\vartheta(\sigma_{1},\sigma_{2},\sigma_{3}...) 
= (\sigma_{2},\sigma_{3}...)$, 
generates a dense orbit in the code space
$I^{\infty}$ when the orbit 
$\{\vartheta^{k}(\sigma)\}_{k=0}^{\infty}$ 
starts at a disjunctive sequence 
$\sigma=(\sigma_{1},\sigma_{2}...)$,
cf. \cite{BV-Developments}.

Disjunctive sequences are prevalent in both 
topological and measure-theoretic sense.
The set of sequences which are not disjunctive
is: (i) $\sigma$-porous, in particular 
it is of the first Baire category,
in the code space $(I^{\infty},d_{B})$
endowed with the Baire metric;
(ii) null with respect to the Bernoulli measure $\mu_{b}$ 
in $I^{\infty}$; e.g., \cite{CaludeStaiger}.
Moreover, many discrete stochastic processes, 
like nondegenerate Bernoulli and Markov processes, 
generate disjunctive outcomes with probability $1$. 
A sufficient condition for a~general chain to 
yield disjunctive sequences almost surely provides

\begin{theorem}[\cite{Les2014BAustMS} Theorem 3.4, Lemma 3.5]\label{th:DisjunctiveProcessViaMinorant}
Let $(Z_n)_{n=1}^{\infty}$ be a (not necessarily stationary) 
stochastic process with states in a finite set $I$.
Suppose that $Z_n$ satisfies 
\begin{equation*}
\Pr(Z_n = \sigma_n \,|\, Z_{n-1}=\sigma_{n-1},
{\ldots},Z_1=\sigma_1) \geq {p_n}>0
\end{equation*}
for all ${{\sigma}_1,{\ldots},{\sigma}_n\in I}$ 
and $n\geq 1$.
Let the minorant $(p_n)_{n=1}^{\infty}$ satisfy
\begin{equation}\label{eq:minorant}
\lim_{n\to\infty} 
\frac{{p}_{n}^{-1}}{n^c} =0 
\mbox{ for every } c>0.
\end{equation}
Then $Z_n$ generates a disjunctive sequence
$\sigma_n$ with probability $1$.
\end{theorem}

(In the above 
$\Pr(Z_n=\sigma_n\,|\, 
Z_{n-1}=\sigma_{n-1}, {\ldots},Z_1=\sigma_1)$,
understood as $\Pr(Z_1=\sigma_1)$ when $n=1$,
denotes the conditional probability 
that $Z_n=\sigma_n$ occurs if earlier 
have occured already
$Z_{n-1}=\sigma_{n-1}, ..., Z_1=\sigma_1$;
cf. \cite{Knill}.)

\begin{example}[\cite{BV-Developments} Definition 4.1; \cite{Les2014BAustMS} Examples 3.2 and 3.6]
The following sequences fulfill \eqref{eq:minorant}: 
$p_n\equiv \operatorname{const.} >0$,
$p_n= (\log n)^{-b}$ for some $b>0$.
The following sequences do not fulfill \eqref{eq:minorant}:
$p_n= n^{-b}$,
$p_n= \sin(n^{-b})$, where $b>0$.
\end{example}

Let ${\mathcal{F}}= \{w_1,...,w_N\}$ be an IFS 
on a Hausdorff topological space $X$. 
The sequence $(x_n)_{n=0}^{\infty}$ given by iteration
\begin{equation}\label{eq:orbit}
\left\{\begin{array}{l}
x_0 \in X,\\
x_n = w_{\sigma_n}(x_{n-1}), n\geq 1,\\
\end{array}\right.
\end{equation}
is called an \emph{orbit} starting at $x_0$ and driven by 
$(\sigma_n)_{n=1}^{\infty}\in I^{\infty}$. 
The \emph{$\omega$-limit set} of the orbit $x_n$ is 
defined to be
\begin{equation*}
\omega((x_n)) := \bigcap_{m=0}^{\infty} 
\overline{\{x_n: n\geq m\}}.
\end{equation*}

\begin{theorem}[Deterministic chaos game]\label{th:ChaosGame}
Let ${\mathcal{F}}= \{w_1,...,w_N\}$ be an IFS consisting of Browder contractions on a complete metric space $(X,d)$. 
Let $A$ be an attractor of ${\mathcal{F}}$. Then for every $x_0\in X$ and 
any disjunctive driver 
$\sigma\in\{1,...N\}^{\infty}$
the orbit $x_n$ starting at $x_0$ and 
driven by $\sigma$ (according to \eqref{eq:orbit}) 
recovers $A$, i.e., $\omega((x_n))=A$.  
\end{theorem}
\begin{proof}
For a given driver $\sigma$ and two starting 
points $x_0\in X$ and $a_0\in A$, 
the orbits $x_n= w_{\sigma_n}(x_{n-1})$ and 
$a_n= w_{\sigma_n}(a_{n-1})$ are getting 
closer to each other, i.e.,
\[
d(x_n,a_n)\leq {{\varphi}}^n(d(x_0,a_0)) \to 0,
\]
where ${\varphi}$ is a module of continuity common
for all maps $w_i$,
${\varphi}^n$ stands for the $n$-fold composition of 
${\varphi}$ and $\lim_{n\to\infty} {\varphi}^n(t)=0$
due to \cite{Granas} chap.I $\S$1.6 (B.2) p.19
(see also Remark \ref{remmatkowski}). 
In consequence
\begin{equation}\label{eq:XreducetoA}
\omega((x_n))=\omega((a_n)).
\end{equation}

By Theorem \ref{HB2} and Corollary \ref{cor:tifs}
the dynamics of 
${\mathcal{F}}|A = \{w_{1}{|}_A,...,w_{N}{|}_A\}$ 
on $A$ is conjugated to the dynamics of 
$\{\tau_{1},...,\tau_{N}\}$ on $I^{\infty}$, 
where $I=\{1,...,N\}$, 
$\tau_i:I^{\infty}\to I^{\infty}$,
$\tau_i(\alpha)= i\hat{ }\alpha$ for $i\in I$,
$\alpha\in I^{\infty}$.
The conjugation is established by the coding map
$\pi:(I^{\infty},d_{B})\to (A,d)$ via commutations
\begin{equation}\label{eq:commute}  
\pi\circ \tau_i= w_i\circ \pi,\; i\in I,
\end{equation}
where $d_{B}$ is the Baire metric 
and $d$ is a metric in $A$ induced from $X$.

\begin{figure}
\[
\xymatrix{
{\varsigma_{n-1} = (\sigma_{n-1},\dots,\sigma_{1}) \widehat{\;} \alpha} 
   \ar@{|->}[r]^-{\pi}    
   \ar@{|->}[d]_{{\tau}_{\sigma_n}} 
& 
{a_{n-1}} 
    \ar@{|->}[d]^{w_{\sigma_n}}
\\
{\varsigma_n = (\sigma_{n},\dots,\sigma_{1}) \widehat{\;} \alpha} 
    \ar@{|->}[r]^-{\pi} 
& 
{a_n}
}
\]
\caption{Action of the chaos game on 
the attractor $A$ (on the right) and 
on the code space $I^{\infty}$ (on the left)
conjugated via the coding map $\pi$.}\label{fig:diagram}
\end{figure}

In view of \eqref{eq:XreducetoA}, 
it is enough to show that 
if $a_0\in A$ and the orbit $a_n$ is driven by 
a disjunctive sequence 
$\sigma=(\sigma_n)_{n=1}^{\infty}$, 
then $\omega((a_n))=A$.
Let us fix a starting point $a_0$ 
together with its address
$\alpha= (\alpha_n)_{n=1}^{\infty}\in I^{\infty}$, i.e., $a_0=\pi(\alpha)$. 
Thanks to \eqref{eq:commute} we have
\begin{eqnarray*}
a_n= w_{\sigma_n}\circ...\circ w_{\sigma_1}(a_0)=
w_{\sigma_n}\circ...\circ w_{\sigma_1}(\pi(\alpha))= 
w_{\sigma_n}\circ...\circ \pi\circ \tau_{\sigma_1}(\alpha)= \\
= \pi \circ \tau_{\sigma_n}\circ... 
\circ\tau_{\sigma_1}(\alpha_{1},\alpha_{2},\dots)=
\pi(\sigma_{n},\dots,\sigma_{1},\, \alpha_{1},\alpha_{2}\dots);
\end{eqnarray*}
see also Figure \ref{fig:diagram}.
Since $\sigma$ is disjunctive, 
for any given $m\geq 1$, 
the prefixes of addresses
\[
\varsigma_n := (\sigma_{n},\dots,\sigma_{1}) \widehat{\;\;} \alpha,\; n\geq m,
\]
exhaust all possible finite words. 
(A desired prefix appears as a reversed sequence 
on the far enough position in $\sigma$.) 
Hence, on calling Lemma \ref{lem:denseincodespace}  
each set $\{\varsigma_n: n\geq m\}$, $m\geq 1$, 
is dense in $(I^{\infty},d_{B})$. Finally
\[
\omega((a_n))=\bigcap_{m=0}^{\infty} \overline{\{a_n:n\geq m\}} =
\bigcap_{m=0}^{\infty} 
\pi(\overline{\{\varsigma_n : n\geq m\}}) = \pi(I^{\infty})=A,
\]
where the second equality follows from the closedness
of the continuous map $\pi$ on a compact set. 
\end{proof}

Deterministic chaos game is also valid for topologically contractive IFSs.

\begin{theorem}\label{th:ChaosGame4tifs}
Let ${\mathcal{F}}= \{w_1,...,w_N\}$ be a TIFS on 
a Hausdorff topological space $X$. 
Let $A$ be an attractor of ${\mathcal{F}}$. 
Then for every $x_0\in X$ and 
any disjunctive driver 
$\sigma\in\{1,...N\}^{\infty}$
the orbit $x_n$ starting at $x_0$ and 
driven by $\sigma$ 
(according to \eqref{eq:orbit}) 
recovers $A$, i.e., $\omega((x_n))=A$.  
\end{theorem}
\begin{proof}
It is sufficient to prove \eqref{eq:XreducetoA}.
Since TIFS admits a compact attractor and 
a conjugation with a canonical IFS
via coding map (Theorem \ref{HB2}), 
the rest follows as in the proof of 
Theorem \ref{th:ChaosGame}. 

Let $x_n, a_n$ be two orbits 
with the same driver $(\sigma_n)_{n=1}^{\infty}$,
one starting at $x_0$, and the other at 
$a_0\in A$.
We are going to show that 
$\omega((x_n))\subseteq \omega((a_n))$. 
The reverse inclusion will follow by symmetry.
By Definition \ref{def:TIFS} (i): 
\begin{equation}\label{eq:nakladkaC}
\mbox{there exists a compact }\; 
C\supseteq \{x_0\}\cup A  \;\mbox{ s.t. }\; 
\mathcal{F}(C)\subseteq C.
\end{equation}
In particular, 
$\omega((x_n))\cup \omega((a_n)) \subseteq C$.
Therefore, by restricting
$\mathcal{F}$ to $C$ if necessary, we can assume further 
that $X$ is compact.
(The $\omega$-limit set of an orbit of $\mathcal{F}$ 
which lies in $C$ coincides with 
the $\omega$-limit set of that orbit understood
as an orbit of the restricted system
$\mathcal{F}\vert C=\{w_1\vert C,..., w_N\vert C\}$.)

Let $y\in \omega((x_n))$. Fix open $U\ni y$ and
arbitrarily large number $m\in{\mathbb{N}}$. Shrink $U$
to open $V\ni y$, $\overline{V}\subseteq U$.
Let $C$ be chosen according to \eqref{eq:nakladkaC}. 
By Definition \ref{def:TIFS} (ii) 
there exists $k_0$ s.t. for all 
$(\alpha_1,...,\alpha_k)\in\{1,...,N\}^k$, 
$k\geq k_0$, it holds
\begin{equation}\label{eq:UValternative}
w_{\alpha_1}\circ ... \circ w_{\alpha_k}(C)
\subseteq U \mbox{ or }
w_{\alpha_1}\circ ... \circ w_{\alpha_k}(C)
\subseteq X\setminus \overline{V}.
\end{equation}
(Inward composed maps yield a singleton.)
Since $y\in\omega((x_n))$, we can pick 
$k\geq \max\{k_0,m\}$ so that $x_k\in V$. 
Recalling that 
$x_k= w_{\sigma_k}\circ ... \circ w_{\sigma_1}(x_0)$,
$x_0\in C$, this implies
\[
V\cap w_{\sigma_k}\circ ... \circ w_{\sigma_1}(C) 
\neq \emptyset.
\]
Therefore the second alternative in 
\eqref{eq:UValternative} is false, and we have
$w_{\sigma_k}\circ ... \circ w_{\sigma_1}(C)\subseteq U$.
Recalling that 
$a_k= w_{\sigma_k}\circ ... \circ w_{\sigma_1}(a_0)$,
$a_0\in C$, shows that
\[
a_k\in w_{\sigma_k}\circ ... \circ w_{\sigma_1}(C)\subseteq U
\mbox{ for some } k\geq m.
\]
Therefore, $y\in\omega((a_n))$, 
because $m$ and $U$ were arbitrary.
\end{proof}

\begin{remark}\label{rem:chgameRemetrize}
If $X$ is metrizable, then the proof of 
Theorem \ref{th:ChaosGame4tifs} can 
be simplified on the basis of 
Theorem \ref{metrizability} on remetrizability.
One just takes a suitable compact $C\supseteq A$ 
with $\mathcal{F}(C)\subseteq C$ and 
considers the restriction of $\mathcal{F}$ to $C$, 
$\mathcal{F}\vert{C}=\{w_{1}\vert{C},...,w_{N}\vert{C}\}$. 
Then there is a complete metric $d$ so making 
$w_{i\vert C}$ Rakotch contractive all at once.
\end{remark}

\begin{remark}
The generation of the random orbit \eqref{eq:orbit}
can be viewed as an action of a skew-product system
(a.k.a. cocycle a.k.a. random dynamical system) 
over the Bernoulli shift 
$\vartheta: I^{\infty}\to I^{\infty}$ with fiber $X$;
cf. \cite[chap.14]{Cheban}, \cite[chap.2.1]{Arnold} 
(see also \cite{Barrientos} and 
\cite[Remark 4.1.5(b)]{Kieninger} 
with the footnote on p.84).
\end{remark}

Let us comment upon the essence and history
of the chaos game.
The chaos game algorithm works in such a way 
that to recover the attractor
\[
A=\bigcap_{m=0}^{\infty} 
\overline{\bigcup_{n=m}^{\infty} {\mathcal{F}}^n(\{x_0\})},
\]
instead of building the full tree
$T = \bigcup_{n=0}^{\infty} {\mathcal{F}}^n(\{x_0\})$
of iterations of the Hutchinson operator ${\mathcal{F}}$,
\[
{\mathcal{F}}^n(\{x_0\})=  \{w_{\sigma_n}\circ ... \circ w_{\sigma_1}(x_0): 
(\sigma_1,...,\sigma_n)\in I^n\},
\] 
see Figure \ref{fig:tree}, it is enough 
to climb along a single, 
yet sufficiently complex branch of 
the tree $T$; namely, it is enough 
to follow a disjunctive orbit
$x_n=w_{\sigma_n}\circ ... \circ w_{\sigma_1}(x_0)$.
(Note the order of composition along 
a word, which is different from that 
in \eqref{eq:fibercomposition}.)

\begin{figure}
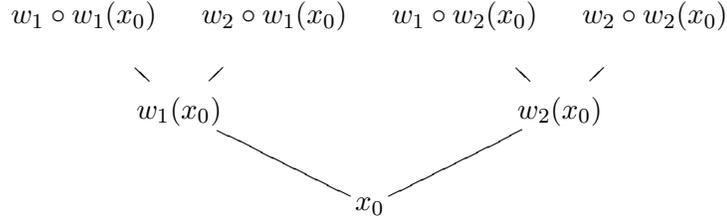

\[
\xygraph{
[]*+[o]{x_0}
  (-[ull] *+=[o]{w_{1}(x_0)}
      (-[ul] *+=[o]{w_{1}\circ w_{1}(x_0)}
      ,-[ur] *+=[o]{w_{2}\circ w_{1}(x_0)})
   ,-[urr] *+=[o]{w_{2}(x_0)}
      (-[ul] *+=[o]{w_{1}\circ w_{2}(x_0)}
      ,-[ur] *+=[o]{w_{2}\circ w_{2}(x_0)})
   )
}
\]
\caption{Part of a full tree of iterations 
of $\mathcal{F}=\{w_1,w_2\}$ up to a second level: 
$\bigcup\limits_{n=0}^{2} \mathcal{F}^n(\{x_0\})$.}\label{fig:tree}
\end{figure}

The proof of the deterministic chaos game 
which we have provided above was not 
presented anywhere so far. Nevertheless, 
it is much in the spirit of 
\cite{Edgar} Theorem 5.1.3; see also \cite{Mar}. 
The ideas behind it are buried in some papers 
from 1990s, cf. \cite{Goodman} and \cite{Hoggar}.
Formally, the deterministic chaos game was stated
for the first time in \cite{BV-Developments}
(without a proof, for strongly-fibred systems).

Recently it was shown that the probabilistic 
chaos game algorithm (i.e., when the orbit 
is driven by a stochastic process) works 
for very general noncontractive IFSs, 
cf. \cite{BLesR2016chaos}. Since many 
stochastic processes generate disjunctive 
sequences, the probabilistic version
of the chaos game readily follows from 
its deterministic version. 
One could hope that it is always the case and 
the probabilistic chaos game is just a corollary 
to the deterministic chaos game.
Notably, the deterministic chaos game algorithm 
is valid for two large classes of systems: 
IFSs with a strongly-fibred attractor (e.g., 
\cite{BV-Developments}) and IFSs comprising
nonexpansive maps (cf. \cite{Les2015chaos}).
Unfortunately, there exist IFSs for which the 
probabilistic chaos game works while 
deterministic --- fails, cf. \cite{Barrientos}.

\section{Between contractive and non-contractive realm}

We are going to exhibit some types of
contractive iterated function systems (IFS) 
which suffer various deficiencies 
making them less flexible proposals 
for generalized contractive IFSs 
than topologically contractive IFSs (TIFS).

\subsection{Eventual contractions}\label{sec:EC}
Given a complete metric space $X$, we say 
that a continuous map $w:X\to X$ is 
an \emph{eventual contraction} if 
for some $p\in{\mathbb{N}}$, the $p$-fold composition 
$w^p$ is a Banach contraction (e.g.,
\cite{Granas} chap.I $\S$1.6 (A.1)).
This concept has proven to be handy 
in some circumstances,
e.g. in the theory of integral equations 
\cite{KolmogorovFomin} chap.II $\S$15.V.

Clearly, if $X$ is compact, then 
an eventual contraction is a 
\emph{topological contraction in the sense 
of Tarafdar}, i.e., the intersection 
$\bigcap_{k\in{\mathbb{N}}}w^k(X)$ is a singleton; 
cf. \cite{Tarafdar} Definition 2.6 and 
Theorem 2.36 on p.88. Noting that in 
such cases $\mathcal{F}=\{w\}$ is a topologically
contractive IFS, it is natural to ask 
whether all IFSs consisting of Tarafdar's
contractions or eventual contractions are 
TIFS or, at least, they generate unique 
invariant sets. The answer is negative as 
the following example shows.

\begin{example}[Non-TIFS of topological contractions]\label{ex:NonTIFSofTarafdar}
Let $w_i:[0,1]\to[0,1]$, $i=1,2$, be defined by 
$w_1(x)=\max\left\{\frac{1}{2},1-x\right\}$ and 
$w_2(x)=\min\left\{\frac{1}{2},1-x\right\}$.
Then $w_1,w_2$ are continuous and 
$w_{1}^2([0,1])=w_{2}^2([0,1])=
\left\{\frac{1}{2}\right\}$, 
so they are continuous Tarafdar's contractions
(even eventual contractions) 
acting on a compact space. 
On the other hand, for every 
$t\in\left[0,\frac{1}{2}\right]$, we have
\begin{equation*}
w_1([t,1-t])\cup w_2([t,1-t])=
\left[\frac{1}{2},1-t\right]\cup
\left[t,\frac{1}{2}\right] =[t,1-t].
\end{equation*}
Hence the IFS $\mathcal{F}:=\{w_1,w_2\}$ does not 
generate a unique nonempty compact invariant set 
and, in particular, $\mathcal{F}$ is not topologically 
contractive.
\end{example}
\begin{remark}
Observe that in Example \ref{ex:NonTIFSofTarafdar}
we have $|w_1(x_1)-w_1(x_2)| = |x_1-x_2|$
for some $x_1\neq x_2$ (namely 
$\frac{1}{2} \leq x_1< x_2\leq 1$).
This means that a topologically contractive map 
need not be a weak contraction under any 
good comparison function ${\varphi}$ for the original metric 
despite it is a ${\varphi}$-contraction after suitable
change of a metric (by either the Miculescu-Mihail 
or the Meyer remetrization theorem; 
cf. Theorems \ref{theorem:meyers} and \ref{metrizability}).
\end{remark}

Whether an IFS comprising eventually contractive
maps yields a topologically contractive IFS
depends on the interaction between the maps,
not only upon the individual maps. 
We look closer at this phenomenon below.

A map $u:X\to X$ on a complete metric space 
$(X,d)$ is called \emph{bi-Lipschitz} provided
there exist two coefficients 
$\lambda(u)\geq \kappa(u)>0$ s.t.
\begin{equation*}
\kappa(u)\cdot d(x_1,x_2) \leq 
d(u(x_1),u(x_2)) \leq 
\lambda(u)\cdot d(x_1,x_2)
\mbox{ for } x_1,x_2\in X.
\end{equation*}
In other words $\lambda(u)$ is a 
(not necessarily minimal) Lipschitz 
constant of $u$, and $1/\kappa(u)$ is 
a Lipschitz constant of $u^{-1}:u(X)\to X$, 
the inverse of $u$. Obviously 
$\lambda(u_1\circ u_2)\leq 
\lambda(u_1) \cdot \lambda(u_2)$
and $\kappa(u_1\circ u_2)\geq 
\kappa(u_1) \cdot \kappa(u_2)$ for
bi-Lipschitz $u_1,u_2:X\to X$.

Let us define for $t_1,t_2\geq 0$,
$\ell^p(t_1,t_2)=(t_1^p+t_2^p)^{\frac{1}{p}}$
when $1\leq p<\infty$ and 
$\ell^p(t_1,t_2)=\max\{t_1,t_2\}$
when $p=\infty$. The function 
$\ell^p:{\mathbb{R}}_{+}\times{\mathbb{R}}_{+}\to{\mathbb{R}}_{+}$ is 
continuous, monotone in each variable, and 
obeys many other properties; for instance
$\ell^p(t_1,t_2)=\ell^p(t_2,t_1) \geq
\ell^{\infty}(t_1,t_2)$.

\begin{example}[Eventually contractive IFS]\label{ex:EC}
Let $(X,d)$ be a complete metric space 
containing at least two points.
We endow $X\times X$ with
the $\ell^p$-product metric 
\begin{equation*}
d_{X\times X}((x_1,y_1), (x_2,y_2)) =
\ell^p(d(x_1,x_2), d(y_1,y_2))
\mbox{ for } x_1,y_1, x_2,y_2\in X.
\end{equation*}
(See also \cite{Les2007Lifsh} for more 
general product metrics.) So metrized 
product $X\times X$ is a complete space.

Let $v_i,u_i:X\to X$, $i=1,2$, 
be bi-Lipschitz maps.
Let $r:X\times X\to X\times X$, 
$r(x,y)=(y,x)$ for $x,y\in X$.
We define on $X\times X$ an IFS 
$\mathcal{F}=\{w_1,w_2\}$ via 
$w_i:= r\circ (v_i,u_i)$, $i=1,2$.
Note that 
\[
w_j\circ w_i = (u_j\circ v_i,v_j\circ u_i)
\mbox{ for } i,j\in\{1,2\}.
\]
By induction
\begin{equation}\label{eq:w1w2n}
(w_2\circ w_1)^n = ((u_2\circ v_1)^n,
(v_2\circ u_1)^n), n\geq 1.
\end{equation}

By regulating $\kappa$'s and $\lambda$'s
we can observe two characteristic scenarios.

(i) $\kappa(u_i)\geq 1$ for $i=1,2$, 
yet
\[
\operatorname{Lip}(w_j\circ w_i) \leq 
\max_{i,j\in\{1,2\}}
\{\lambda(v_j) \cdot \lambda (u_i)\} < 1.
\] 
This means that although the IFS $\mathcal{F}$ 
is not contractive, its second iterate 
$\mathcal{F}^2=\{w_j\circ w_i: i,j\in\{1,2\}\}$
is contractive. In particular
$\mathcal{F}^2$ admits an attractor, so by
Remark \ref{remark:CFP} $\mathcal{F}$ also 
has an attractor.

(ii) $\kappa(u_i)\geq 1$, 
$\kappa(u_i)\cdot \kappa(v_{3-i}) >1$,
$\lambda(u_i)\cdot \lambda(v_i) < 1$ for
all $i=1,2$. This means that $w_i$'s are
eventually contractive ($w_i\circ w_i$ 
are contractive), but $\mathcal{F}$ does not admit
an attractor. Suppose, a~contrario, that
$A\subseteq X\times X$ is a compact attractor 
of $\mathcal{F}$. Pick $(a_x,a_y)\in A$, 
$(x,y)\neq (a_x,a_y)$. Then, using 
\eqref{eq:w1w2n}, we obtain
\begin{eqnarray*}
d_{X\times X}((w_2\circ w_1)^n(x,y),
(w_2\circ w_1)^n(a_x,a_y)) \geq \\ 
\max\{(\kappa(v_1)\kappa(u_2))^n \cdot d(x,a_x),
(\kappa(v_2)\kappa(u_1))^n \cdot d(y,a_y)\} 
\to\infty.
\end{eqnarray*}
\end{example}

\begin{remark}
Putting $X={\mathbb{R}}$, $u_1=u_2= \operatorname{id}$, 
$v_1(x)=\frac{x}{2}$, and 
$v_2(x) = \frac{x+1}{2}$ in Example \ref{ex:EC} 
yields \cite{Mendivil} Exercise 2.28.
\end{remark}

\begin{remark}
For $X={\mathbb{R}}$ Example \ref{ex:EC} takes 
particularly nice form. 
One can plug various values of 
$\tilde{v_i}$, $\hat{v_i}$, $\tilde{u_i}$
$\hat{u_i}$, $i=1,2$, and play with matrices
\[
w_i(x,y)= \begin{bmatrix}0&1\\
1&0\\\end{bmatrix} \cdot
\left(\begin{bmatrix}\tilde{v_i}&0\\
0&\tilde{u_i}\\\end{bmatrix}
\begin{bmatrix}x\\y\\\end{bmatrix} +
\begin{bmatrix}\hat{v_i}\\
\hat{u_i}\\\end{bmatrix}\right),
\; x,y\in {\mathbb{R}}.
\]
Then $v_i(x)=\tilde{v_i}\cdot x+\hat{v_i}$,
$\kappa(v_i)=\lambda(v_i)=\tilde{v_i}$,
and similarly
$u_i(x)=\tilde{u_i}\cdot x+\hat{u_i}$,
$\kappa(u_i)=\lambda(u_i)=\tilde{u_i}$.
\end{remark}

\subsection{Average contractive IFSs}

A probabilistic iterated function system $(\mathcal{F},\vec{p})$ 
comprising Lipschitz maps $w_1,...,w_N$ 
is said to be \emph{average contractive}, 
provided $\sum_{i=1}^{N} p_i\cdot\operatorname{Lip}(w_i)<1$, 
where $\operatorname{Lip}(w_i)$ stands for 
the Lipschitz constant of $w_i$.

\begin{remark}\label{rem:average}
Let $\mathcal{F}=\{w_1,...,w_N\}$ be an IFS
consisting of Lipschitz maps.
There exists a vector of positive
weigths $\vec{p}$ such that $(\mathcal{F},\vec{p})$
is average contractive if and only if
$\min\{\operatorname{Lip}(w_i):i=1,...,N\}<1$. 
The class of such weights can be easily determined. 
For each $i=1,...,N$, 
set $c_i:=\operatorname{Lip}(w_i)$ 
and assume for simplicity that 
$c_1=\min\{c_1,...,c_N\}$. 
Observe that $p_1+...+p_N=1$ 
is equivalent to $p_1=1-p_2-...-p_N$. 
Hence our aim is to find 
all positive values $p_2,...,p_N$ 
so that $p_2+...+p_N<1$ and 
$(1-p_2-...-p_N)c_1+p_2 c_2+...+p_N c_N<1$. 
All in all, the family of all desired 
vectors consists of all $N$-tuples
$(p_1,p_2,...,p_N)$ so that
\[
\begin{array}{ccc}
p_2\in(0,1),
\;\;p_2(c_2-c_1)<1-c_1,\\
p_3\in(0,1-p_2),
\;\;p_3(c_3-c_1)<1-c_1-p_2(c_2-c_1),\\
..................................................\\
p_N\in(0,1-p_2-...-p_{N-1}),
\;\;p_N(c_N-c_1)<1-c_1-p_2(c_2-c_1)-...-p_{N-1}(c_{N-1}-c_1),\\
p_1=1-p_2-...-p_N.
\end{array}
\]
\end{remark}

The interest in this kind of systems steams 
from image processing. The basis for 
the applications is the observation
that the Markov operator $M$ associated with 
an average contractive IFS on a complete
separable metric space has CFP (in other 
words, $M$ is \emph{asymptotically stable}); 
cf. \cite{MyjakSzarek} Fact 3.2. 
Even more, under suitable assumptions, 
$M$ is contractive with respect to 
the Monge-Kantorovitch metric $d_{MK}$;
cf. \cite{Mendivil} Theorem 2.60.
Having this in mind and recalling that $M$ 
arising from the Rakotch contractive IFS 
is weakly contractive with respect to $d_{MK}$, one can
look for a hybrid generalization of an
average contractive and a weakly contractive IFS.
We realize this idea below.

\begin{definition}\label{def:averageRakotch}
We say that a probabilistic IFS $(\mathcal{F},\vec{p})$ 
is \emph{average Rakotch contractive}, 
if there exist positive numbers $c_1,...,c_N$ 
such that
\begin{itemize}
\item[(i)] $\sum_{i=1}^Np_ic_i\leq 1$;
\item[(ii)] each $w_i$ is $\varphi_i$-contraction
for some comparison function of the form
$\varphi_i(t)=\lambda_i(t)t$, where 
$\lambda_i:{\mathbb{R}}_+\to{\mathbb{R}}_+$ is nonincreasing 
and $\lambda_i(t)<c_i$ for $t>0$.
\end{itemize}
\end{definition}

Let us exhibit a probabilistic IFS 
$(\mathcal{F},\vec{p})$ which is neither average contractive
nor Rakotch contractive, yet 
it is average Rakotch contractive.

\begin{example}\label{ex:averageRakotch}
Let $w_1,w_2:\left[0,\frac{\pi}{2}\right]
\to[0,\frac{\pi}{2}]$ be defined by 
\begin{equation*}
w_1(x):=2\sin(x),\;\;\;
w_2(x):=\frac{1}{2}\sin(x).
\end{equation*}
For every $0\leq x<y\leq\frac{\pi}{2}$, 
we have
\begin{equation*}
\sin(y)-\sin(x)\leq \sin(y-x)-\sin(0)=
\sin(y-x)=\frac{\sin(y-x)}{y-x}(y-x).
\end{equation*}
From this we see that $w_1$ is
a $\varphi_1$-contraction for 
$\varphi_1(t)=\lambda_1(t)t$, where 
$\lambda_1(t)=\frac{2\sin(t)}{t}$, 
and $w_2$ is a~$\varphi_2$-contraction for
$\varphi_2(t)=\lambda_2(t)t$, where 
$\lambda_2(t)=\frac{\sin(t)}{2t}$. 
In particular, $\operatorname{Lip}(w_1)=2$ and 
$\operatorname{Lip}(w_2)=\frac{1}{2}$. 
Hence the IFS $\mathcal{F}=\{w_1,w_2\}$ is 
not weakly contractive and 
the probabilistic IFS $(\mathcal{F},\vec{p})$ 
is not average contractive for
$\vec{p}=\left(\frac{1}{3},\frac{2}{3}\right)$.
Nevertheless, $(\mathcal{F},\vec{p})$ is 
average Rakotch contractive.
\end{example}

Careful examination of the IFS $\mathcal{F}$ from 
Example \ref{ex:averageRakotch} reveals
that although $(\mathcal{F},\vec{p})$ is 
not average contractive for 
$\vec{p}=\left(\frac{1}{3},\frac{2}{3}\right)$,
it is average contractive for many other 
vectors of weights (precisely for those 
$\vec{p}=(p_1,p_2)$ which satisfy 
$2\cdot p_{1} + \frac{1}{2}\cdot p_{2} <1$.)
This is not a coincidence. 

\begin{remark}
Let $\mathcal{F}=\{w_1,...,w_N\}$ be an IFS consisting 
of Lipschitz maps. 
The following conditions are equivalent:
\begin{itemize}
\item[(i)] $(\mathcal{F},\vec{p})$ is average 
Rakotch contractive for some $\vec{p}$;
\item[(ii)] either $\mathcal{F}$ is Rakotch contractive, or
$(\mathcal{F},\vec{p})$ is average contractive for some $\vec{p}$.
\end{itemize}
Indeed, Definition \ref{def:averageRakotch} (i)
implies that either $c_i < 1$ for some $i$ or  
$c_i= 1$ for all $i$. The first alternative 
is related to average contractivity 
due to Remark \ref{rem:average}. 
The second alternative is related to 
Rakotch contractivity.
\end{remark}

The following result is an extension of 
Theorem \ref{th:weakcontrMarkov} 
for average Rakotch contractive IFSs.

\begin{theorem}
Let $(\mathcal{F},\vec{p})$ be an average Rakotch
contractive IFS on a complete metric space $X$. 
Then the Markov operator 
$M:\mathcal{P}_1(X)\to\mathcal{P}_1(X)$
induced by $(\mathcal{F},\vec{p})$
is a Rakotch contraction with respect to the 
Monge-Kantorovich metric on the space 
$\mathcal{P}_1(X)$ of Radon 
probability measures with integrable distance. 
In particular, $M$ admits 
a unique invariant measure 
$\mu_{*}\in\mathcal{P}_1(X)$ which is 
the CFP of $M$, i.e., the iterations
$M^n(\mu)\to \mu_{*}= M(\mu_{*})$ 
converge weakly.
\end{theorem}
\begin{proof}
The proof is similar to the proof of 
Theorem \ref{th:weakcontrMarkov}.  
Hence we will only outline the additional steps 
required to complete the proof.

First we need to find strictly 
increasing, concave functions 
$\psi_1,...,\psi_N:{\mathbb{R}}_+\to{\mathbb{R}}_+$ 
such that each $w_i$ is a 
$\psi_i$-contraction and for every $t>0$, $\sum_{i=1}^Np_i\psi_i(t)<t$. 
This can be done with the aid of 
Lemma \ref{lemma:rakotch}. 
Indeed, we apply this lemma to 
$\tilde{\varphi}_i(t)=\frac{\varphi_i(t)}{c_i}$, 
obtaining maps $\tilde{\psi}_i(t)$ 
which satisfy $\frac{\varphi_i(t)}{c_i}\leq \tilde{\psi}_i(t)<t$ 
for $t>0$. Finally, we set 
$\psi_i(t):=c_i\tilde{\psi}_i(t)$.

Then we proceed as in the proof of 
Theorem \ref{th:weakcontrMarkov} and show that 
that $M$ is a $\varphi$-contraction for
$\varphi(t):=\sum_{i=1}^{N} p_i\psi_i(t)$. 
As $\varphi$ is strictly increasing, 
concave and $\varphi(t)<t$, we infer 
that $M$ is a Rakotch contraction. 
The result follows.
\end{proof}

Average contractive IFSs are amenable for 
techniques of ergodic theory of Markov processes,
yet allow for non-contractive maps to be employed.
This makes them a good proposal for the marriage
of theory and applications.
We will see that there are some clouds
on the horizon in this picturesque landscape.

It is a common belief that average 
contractivity explains behaviour of several
non-contractive IFSs experimented with
in computer graphics, e.g., 
\cite{LasotaMyjak1996}, \cite{Mendivil}. 
However, it should be stressed that average
contractive IFSs may lack attractors and
that running the random iteration may 
produce artifacts which can be confirmed
theoretically, so they do not occur simply 
because of poor numerics. The example below
offers some insight into this phenomenon.

\begin{example}[Lasota--Myjak semiattractor]\label{ex:linear2andHalf}
Let $\mathcal{F}= \{w_1,w_2\}$ be an IFS 
on the real line ${\mathbb{R}}$, 
$w_1(x) = \frac{x}{2}$, $w_2(x) = 2x$. 
The IFS $\mathcal{F}$ 
\begin{enumerate}
\item[(i)] does not have an attractor;
\item[(ii)] is average contractive;
\item[(iii)] induces for the vector of 
weigths $\vec{p}=(p_1,p_2)$, 
$0<p_2<\frac{1}{3}$, $p_1=1-p_2$,
a Markov operator 
$M=M_{(\mathcal{F},\vec{p})}:{\mathcal{P}}({\mathbb{R}})\to{\mathcal{P}}({\mathbb{R}})$ 
which has CFP and the attracting 
invariant probability measure is 
$\mu_{*}=\delta_{0}$;
\item[(iv)] has a unique compact 
$\mathcal{F}$-invariant set 
$A_{*}=\operatorname{supp} \mu_{*}=\{0\}$, but 
$A_{*}$ cannot be recovered by running 
the chaos game algorithm.
\end{enumerate}

To see (i) one simply notes that
\begin{equation}\label{eq:blowup}
\mathcal{F}^n(\{x\})\ni w_2^n(x) = 2^n \cdot x 
\to \pm\infty  \mbox{ for } x\neq 0. 
\end{equation}

Item (ii) is obvious by Remark \ref{rem:average}
($\operatorname{Lip}(w_1) = \frac{1}{2}$).

Item (iii). 
Denote by $s\cdot B  = \{s\cdot b: b\in B\}$ 
the scaling of $B\subseteq {\mathbb{R}}$ by a factor $s>0$. 
Then the Markov operator $M$ takes the following form
\[
M\mu(B) = p_1 \cdot \mu(2\cdot B)
+ p_2 \cdot \mu(\frac{1}{2}\cdot B) 
\mbox{ for } B\in{\mathcal{B}}({\mathbb{R}}), \mu\in{\mathcal{P}}({\mathbb{R}}).
\]
Since $0\in B$ iff $0\in s\cdot B$,
we have $M(\delta_{0}) = \delta_{0}$.
(Actually it is true for all weigths $p_i$.)
The weak convergence 
$M^n(\mu)\to \delta_{0}$
for any initial probability measure 
$\mu\in{\mathcal{P}}({\mathbb{R}})$ is ensured by Fact 3.2 
in \cite{MyjakSzarek}. 
Indeed, $p_1\cdot \frac{1}{2}+p_2\cdot 2<1$.

Finally we address (iv). That $A_{*}$
is the unique compact invariant set 
is obvious from \eqref{eq:blowup}.

Let $x_0\neq 0$. Define
$Z=\{2^{q}\cdot x_0: q\in \mathbb{Z}\}$;
$\mathbb{Z}$ -- integer numbers.
By regulating the number of repetitions 
of a given symbol, it is not hard 
to construct a disjunctive sequence 
$\sigma\in\{1,2\}^{\infty}$ so that the 
orbit $x_n$ starting at $x_0$ and 
driven by $\sigma$ has a ``diffused''
omega-limit
\begin{equation}\label{eq:latticeZ}
\omega((x_n))=Z\cup \{0\}\neq \{0\}.
\end{equation} 
Therefore the deterministic chaos game algorithm
fails to recover $A_{*}$.

Further, from the elementary 
theory of random walk (on the lattice $Z$; 
e.g. \cite{Knill} chap. 3.10),
it follows that the equation 
\eqref{eq:latticeZ} holds almost surely
for any sequence driving an orbit $x_n$,
which is generated by a Bernoulli scheme 
(see also Theorem \ref{th:DisjunctiveProcessViaMinorant}). 
Therefore random orbits starting nearby 
$A_{*}$ do not recover it.
\end{example}
\begin{remark}
The unique invariant set in the example 
above is called a semiattractor 
in the Lasota--Myjak sense; 
e.g. \cite{MyjakSzarek}.
\end{remark}

\section{Beyond contractivity}

We overview some possible extensions of 
the theory of iterated function systems to systems
comprising maps which are far from contractive. 
Our choice is very subjective 
and should be treated as an element of a larger landscape. 
For instance, 
we omit the Lasota-Myjak theory of semiattractors (e.g., \cite{MyjakSzarek}), 
the Conley theory for IFSs (e.g., \cite{BV-Developments}),
limit sets of Kleinian groups (e.g., \cite{IndrasPearls}),
limit sets of parabolic IFSs (e.g., \cite{MauldinUrbanski}), 
quantum IFSs (e.g., \cite{Jadczyk}) 
and non-conformal IFSs (e.g., \cite{KarolySimon}).

In the present section we will show that the existence of invariant sets 
for IFSs on compact spaces or, more generally, for condensing IFSs, 
can be inferred rather easily from general principles of set theory 
and nonlinear functional analysis. 
Very similar approach makes possible to mimic the theory
of global maximal attractors for semiflows on Banach spaces
in the case of iterated function systems. 
Finally, general principles of functional analysis are also suitable 
to prove the existence of invariant measures for Markov operators 
induced by IFSs.

In the course of our further presentation we will use the tool from 
nonlinear functional analysis, called the measure of noncompactness, 
e.g., \cite{AKPRS}. We will also employ the language
of set-valued analysis, e.g. \cite{Beer}. 
We find it more natural than then traditional language of relations
employed in topological dynamics, e.g., \cite{Akin}.

Before we move further, let us address shortly the limitations 
to the naive approach to extend the Hutchinson-Barnsley theory of IFSs 
by trying to apply the metric fixed point theory 
to the Hutchinson operator. The extension of the Banach fixed theorem 
to nonexpansive maps on Banach spaces employs the geometry of balls. 
It turns out that given a metric space, its natural hyperspaces of subsets,
metrized with the Hausdorff ($\ell^{\infty}$-type), 
Pompeiu ($\ell^{p}$-type) or the Borsuk metric of continuity, 
are very ``pokey" (as expressed by Ch. Bandt). For instance, 
two closed balls $B_1, B_2$ in a hyperspace have wide intersection,
like the balls in a Banach space with the $\ell^\infty$-norm; 
that is $\operatorname{diam} (B_1\cap B_2) = \operatorname{diam}(B_1)=2r$, when $B_1$ and $B_2$ 
have equal radii $r$ and have their centers at the distance $r$.
Formally, the so-called Lifshitz constant is equal $1$ 
for any practically interesting hyperspace; cf. \cite{Les2007Lifsh}.

Another imaginable trick to harness the Hutchinson operator could 
be an isometric embedding of the hyperspace into a Banach space 
and then to apply the metric fixed point theory on Banach spaces. 
It turns out that classic embedding theorems for hyperspaces are of no use. 
The Radstr\"{o}m-H\"{o}rmander embedding (\cite{Beer} Theorem 3.2.9) 
works for the hyperspace of convex sets which is against the philosophy 
of fractal geometry where jugged sets are privalent.
A less restrictive Kuratowski-Wojdys{\l}awski 
embedding (\cite{engelking} Problem 4.5.23(f))
allows for jugged sets. It works for 
the hyperspace of closed bounded subsets 
of any metric space. Unfortunately, 
the obstacle to apply the embedding theorems 
is of purely geometric nature. 
The hyperspace is always embedded into 
a metric space with the geometry 
of $\ell^{\infty}$-space. 

Overall, there is no obvious way to apply to 
the Hutchinson operator, fixed point theorems 
for nonexpansive maps like 
the Browder-Goehde-Kirk theorem, e.g.,
\cite{Granas} chap.I $\S$4.6 (C.1) p.76 and (C.5) p.77. 
In a sense, however, some sort of useful 
embedding of a hyperspace into a linear space 
is possible. Namely, viewing sets as supports 
of measures turns out to be a powerful technique 
of turning the nonlinear problem of finding 
fixed points of the Hutchinson operator into 
a linear problem of finding invariant measures 
of the Markov operator;
see section \ref{sec:KrylovBogolubov}.

\subsection{Multivalued IFSs}\label{sec:multivaluedIFS}

Let $X$ be a Hausdorff topological space. 
A mapping $W:X\to 2^X$ is called 
a \emph{multifunction}, or a multivalued map, 
or a set-valued map. The set $W(x)$ is called 
a \emph{value} of $W$ at $x\in X$. The set 
\[
W(S) := \bigcup_{x\in S} W(x)
\]
is called an \emph{image} of $S\subseteq X$ via $W$. 
Thus, nomen omen, $W(x)=W(\{x\})$.

A multifunction $W:X\to 2^{X}$ with nonempty 
values defines a \emph{multivalued IFS}. 
The \emph{Hutchinson operator} 
${\mathcal{F}}:2^{X}\to 2^{X}$ induced by $W$ 
is given by the formula
\[
{\mathcal{F}}(S) := \overline{\bigcup_{x\in S} W(x)} = \overline{W(S)} 
\mbox{ for } S\subseteq X;
\]
cf. \cite{Kieninger} Definition 3.1.1 p.64.

Let ${\mathcal{F}} = \{w_1,...,w_N\}$ be an IFS comprising functions 
$w_i:X\to X$, $i=1,...,N$. Define the multifunction $W:X\to 2^{X}$ by
\begin{equation}\label{eq:MultiFromIFS}
W(x) = \{w_i(x): i=1,...,N\} \mbox{ for } x\in X.
\end{equation}
Then the Hutchinson operator induced by 
the IFS ${\mathcal{F}}$ coincides with the Hutchinson 
operator induced by $W$. We can replace IFSs 
with multifunctions and speak within 
the realm of multifunctions about all concepts 
already defined for IFSs in terms of 
the Hutchinson operator, like for example an
$\mathcal{F}$-invariant set, e.g., \cite{BV-Developments},
\cite{Kieninger}, \cite{Mendivil}.
However, restricting the study of IFSs merely 
to study of the dynamics of the Hutchinson 
operator has its drawbacks. For instance, 
symbolic dynamics for fractals relies on 
the existence of the coding map, which is 
enabled by the representation of multifunction 
via single-valued contractive maps as given 
in equation \eqref{eq:MultiFromIFS}.

It should be underlined that no continuity 
assumption about a multifunction $W$ is 
made unless stated explicitly. In particular, 
Birkhoff type theorems on invariant sets
in sections \ref{sec:Birkhoff4Compact} and 
\ref{sec:Birkhoff4Condensing} hold for general IFSs,
regardless of whether the maps comprising an IFS 
are continuous or discontinuous. That is possible 
because the definition of the Hutchinson operator 
involves the closure operator and the Hutchinson 
operator has good order-theoretic properties.

\begin{lemma}
The Hutchinson operator $\mathcal{F}:2^{X}\to 2^{X}$ 
induced by a multifunction $W:X\to 2^{X}$ is 
\begin{enumerate}
\item[(a)] order-monotone with respect to the inclusion, i.e.,
\[
\mbox{if } S\subseteq S' \mbox{ then } 
\mathcal{F}(S)\subseteq \mathcal{F}(S'), 
\mbox{ for all } S,S'\subseteq X;
\]
\item[(b)] set-additive
\[
\mathcal{F}(S\cup S')= \mathcal{F}(S) \cup \mathcal{F}(S') 
\mbox{ for } S,S'\subseteq X.
\]
\end{enumerate}
\end{lemma}

Since some crucial theorems about IFSs involve 
continuity in their assumptions 
(e.g., the general statement of the chaos 
game algorithm, or invariance of the attractor), 
we note the following.

\begin{proposition}[Continuity of $\mathcal{F}$; \cite{BLesContinuity}]\label{th:ContinuityF}
Let $X$ be a normal topological space. 
Let ${\mathcal{F}} = \{w_1,...,w_N\}$ be an IFS comprising 
continuous functions $w_i:X\to X$, $i=1,..,N$.
The associated Hutchinson operator 
$\mathcal{F}:{\mathcal{K}}(X)\to{\mathcal{K}}(X)$ acting on the hyperspace 
${\mathcal{K}}(X)$ of nonempty compact subsets of $X$ 
is continuous with respect to the Vietoris topology. 
In particular, if $X$ is a metric space, then 
$\mathcal{F}:{\mathcal{K}}(X)\to{\mathcal{K}}(X)$ is continuous with respect to 
the Hausdorff distance.
\end{proposition}

\begin{remark}
The Hutchinson operator $\mathcal{F}$ is Vietoris continuous 
on the hyperspace of closed sets.
When $X$ is a metric space and $\mathcal{F}$ is considered 
on the hyperspace of nonempty closed bounded 
subsets of $X$, $\mathcal{F}$ typically fails 
to be continuous with respect to the Hausdorff distance; 
cf. \cite{BLesContinuity}. 
\end{remark}

\subsection{The Birkhoff theorem on invariant set for compact IFSs}\label{sec:Birkhoff4Compact}

The material in the present section is based on 
\cite{Kieninger, Ok, LesCEJM}; see also
\cite{Tarafdar} Theorem 2.41 p.92.

Let $X$ be a Hausdorff topological space. 
We say that a multifunction $W:X\to 2^X$ 
is \emph{compact}, provided the closure 
of the image $\overline{W(X)}\subseteq X$ 
is a compact set.

\begin{theorem}[Birkhoff theorem for compact IFSs]\label{th:Birkhoff4Compact}
Let $W:X\to 2^X$ be a compact multifunction 
with nonempty values, which acts on a Hausdorff
topological space $X$. Then $W$ admits:
\begin{enumerate}
\item[(a)] the greatest invariant set, and
\item[(b)] a minimal invariant set,
\end{enumerate}
which are compact.
\end{theorem}
\begin{proof}
Let $\mathcal{F}:\mathcal{K}(X)\to \mathcal{K}(X)$ be 
the Hutchinson operator induced by $W$. 
Since $\mathcal{F}(X)$ is compact, without loss of generality 
we may assume that $X$ is compact. 
(Any set $C\subset X$ with $\mathcal{F}(C)=C$ is a~subset of $\mathcal{F}(X)$.)

\emph{Proof via order-theoretic fixed point principles}.
It is enough to recall the following: 
(i) any chain in $(\mathcal{K}(X), \subseteq)$ 
admits an infimum,
(ii) the Hutchinson operator $\mathcal{F}$ is order-monotone,
(iii) $\mathcal{F}(X)\subseteq X$,
and then apply the Kleene principle to get (a) 
and the Knaster-Tarski principle to get (b).

\emph{Proof via order-theoretic consequences of 
the axiom of choice}. Let 
\[
\mathcal{S} = \{S\subseteq X: \mathcal{F}(S)\subseteq S =\overline{S} \neq\emptyset\}.
\]
The poset $(\mathcal{S},\subseteq)$ has the greatest element 
$X\in \mathcal{S}$.   
Moreover, any chain $\mathcal{C}$ in $(\mathcal{S},\subseteq)$ 
admits an infimum. 
Namely, $M=\bigcap \mathcal{C}\in \mathcal{S}$ is 
the greatest lower bound for $\mathcal{C}$. 
(Indeed, for all $C\in\mathcal{C}$: $M\subseteq C$
and $\mathcal{F}(M)\subseteq \mathcal{F}(C) \subseteq C$. 
Hence $\mathcal{F}(M)\subseteq M$.)
Therefore, we can define a transfinite sequence
\begin{equation*}
\left\{\begin{array}{ll}
S_{0} = X, & \mbox{} \\
S_{\alpha +1} = \mathcal{F}(S_{\alpha}) & 
\mbox{ for an isolated ordinal number } \alpha, \\
S_{\beta} = \bigcap_{\alpha<\beta} S_{\alpha} &
\mbox{ for a limit ordinal number } \beta, \\
\end{array}\right.
\end{equation*}
where $\alpha$ runs through all ordinals less than $\chi$,
the first ordinal of cardinality $(\operatorname{card} X)^{+}$ 
(the successor of the cardinal number of $X$). Obviously
$S_{\alpha_2}\subseteq S_{\alpha_1}$ 
for $\alpha_1<\alpha_2$.
It is impossible that $S_{\alpha +1}\neq S_{\alpha}$ 
for all $\alpha<\chi$.
Thus $\mathcal{F}(S_{\alpha^{*}}) = S_{\alpha^{*} +1} = S_{\alpha^{*}}$ 
for some $\alpha^{*}<\chi$.
By the construction, $S_{\alpha^{*}}$ is nonempty and compact, 
and it is the greatest $\mathcal{F}$-invariant set. We have established (a).

For (b) it is enough to note that any chain in $\mathcal{S}$ admits
a lower bound (namely its intersection), so the Zorn lemma gives the existence
of a minimal element in $\mathcal{S}$, denote it $S_{*}$. 
The definition of $\mathcal{S}$ says that $\mathcal{F}(S_{*})\subseteq S_{*}$.
Since $\mathcal{F}(\mathcal{F}(S_{*}))\subseteq \mathcal{F}(S_{*})$, due to order monotonicity of $\mathcal{F}$, 
and $S_{*}\neq \emptyset$, we also have that 
$\mathcal{F}(S_{*}) \in \mathcal{S}$.
Finally, minimality of $S_{*}$, enforces that 
$\mathcal{F}(S_{*}) = S_{*}$.
That is, $S_{*}$ is nonempty and compact and 
it is the minimal $\mathcal{F}$-invariant set. 
We have established (b).
\end{proof}

\begin{remark}
One could try to employ the Kantorovitch 
fixed point theorem instead of the 
Knaster--Tarski theorem, but it is 
a more demanding approach to the question 
of existence of invariant sets. In such 
a case one has to ensure order-continuity 
in addition to order-monotonicity of 
the Hutchinson operator, cf. 
\cite{Tarafdar} chap.3.6 
(and the references therein).
\end{remark}

\begin{example}
Let $X=[-1,1]\subseteq {\mathbb{R}}$ and $W:X\to 2^{X}$, 
$W(x)=\{-x\}$ for $x\in X$. 
Every set of the form $\{a,-a\}$, $a\in X$, is a minimal invariant set.
The whole space $X$ is the greatest invariant set.
\end{example}

\subsection{The Birkhoff theorem on invariant set for condensing IFSs}\label{sec:Birkhoff4Condensing}

The material in the present section is based on \cite{Ok}, \cite{AFGL}, \cite{LesCEJM}.

\begin{definition}\label{def:MNC}
An extended-valued nonnegative functional 
$\gamma:2^X\to[0,\infty]$, defined 
on subsets of a Hausdorff topological space $X$, 
is called a \emph{measure of noncompactness} 
(shortly MNC), if it satisfies the following axioms:
\begin{enumerate}
\item[($\gamma$-0)] $\gamma(\emptyset)=0$;
\item[($\gamma$-1)] $\gamma(\overline{S}) = \gamma(S)$;
\item[($\gamma$-2)] \emph{regularity}: 
if $\gamma(S) = 0$, then $\overline{S}$ is a compact set;
\item[($\gamma$-3)] \emph{monotonicity}:
if $S\subseteq S'$, then $\gamma(S)\leq \gamma(S')$;
\item[($\gamma$-4)] \emph{nonsingularity}:
$\gamma(S\cup\{x\}) = \gamma(S)$;
\end{enumerate}
where $S, S'\subseteq X$, $x\in X$.
\end{definition}

Monotonicity ($\gamma$-3) explains why 
an abstract functional $\gamma$ deserves the name of a measure 
(in the spirit of Choquet's capacities), 
while regularity ($\gamma$-2) explains why 
$\gamma$ measures noncompactness. 

\begin{remark}\label{eq:KuratowskiIntersection}
The set of axioms ($\gamma$-0)---($\gamma$-4)
is strong enough to yield 
the Kuratowski intersection property for $\gamma$ 
(\cite{LesCEJM} Theorem 3.4): 
if $S_n$, $n\geq 1$, is a decreasing sequence 
of nonempty closed subsets of a metric space $X$ 
such that $\gamma(S_n)\to 0$ as $n\to\infty$, 
then the intersection 
$S_{\infty}= \bigcap_{n=1}^{\infty} S_n$ 
is a nonempty compact set, and 
$\lim_{n\to\infty} d_{H}(S_n,S_{\infty})=0$.
\end{remark}

\begin{example}\label{ex:trivialMNC}(Trivial MNC).
Let $X$ be a Hausdorff topological space. 
Fix any $e\in(0,\infty]$.
For $S\subseteq X$ 
put $\gamma(S)= 0$ if $\overline{S}$ is a compact set, 
and $\gamma(S)=e$ otherwise. 
Then $\gamma: 2^{X}\to[0,\infty]$ is an MNC.
\end{example}

\begin{example}(Hausdorff MNC).
Let $(X,d)$ be a complete metric space. The functional
\[
\gamma(S) := \inf \left\{r>0: S\subseteq \bigcup_{j=1}^{k} B(x_j,r)
\mbox{ for some } x_j\in X, k\in{\mathbb{N}}\right\}
= \inf_{K\in{\mathcal{K}}(X)} d_{H}(K,S),
\]
is called the \emph{Hausdorff MNC}. 
If $X$ is not complete, then 
$\gamma$ is not an MNC in the sense of Definition \ref{def:MNC}.
Indeed, if $x_n$ is a Cauchy sequence which is not convergent,
then for $S=\{x_n\}_{n=1}^{\infty}$ we have that $\gamma(S)=0$
and $\overline{S}$ is not compact. 
Note that $\gamma(S)<\infty$ if and only if 
$S\subseteq X$ is bounded.
\end{example}

\begin{remark}\label{rem:Darbo}
The most important application of MNCs is 
the common generalization of two fixed point principles
on Banach spaces: the Banach and Schauder theorem, due 
to Darbo and Sadovski\u{\i}; 
cf. \cite{AKPRS} Theorem 1.5.11 or 
\cite{Granas} chap.II $\S$6.9.C p.133. 
This however involves an additional property of 
an MNC $\gamma$,
\begin{enumerate}
\item[($\gamma$-D)] \emph{Darbo formula}:
$\gamma(\overline{\operatorname{conv} S}) = \gamma(S)$ 
for $S\subseteq X$.
\end{enumerate}
The Darbo formula may be viewed as 
a quantitative generalization of the Mazur theorem 
on compactness of the convex hull.
The Hausdorff MNC in a Banach space obeys ($\gamma$-D). 
\end{remark}

\begin{definition}(Condensing multifunction).
Let $X$ be a Hausdorff space and $\gamma$ an MNC in it.
A~multifunction $W:X\to 2^{X}$ is said to be 
\emph{condensing} with respect to $\gamma$, if 
\begin{equation*}
\gamma(W(S)) \left\{\begin{array}{ll}
< \gamma(S), & \mbox{ when } 0<\gamma(S)<\infty \\
= 0, & \mbox{ when } \gamma(S)=0 
\end{array}\right.
\end{equation*}
for $S\subseteq X$.
\end{definition}

\begin{remark}
If $W:X\to 2^{X}$ is condensing with respect to an MNC 
$\gamma$ satisfying ($\gamma$-0)---($\gamma$-4), 
then its values $W(x)$, $x\in X$, are relatively compact. 
Indeed,
\[
\gamma(\{x\})=\gamma(\emptyset\cup \{x\})= 
\gamma(\emptyset)=0,
\]
so $\gamma(W(x))=0$.
\end{remark}

\begin{remark}\label{rem:ultraadditive}
Sometimes it is assumed that the MNC $\gamma$ 
is additive in the sense of max-plus algebra.
Formally
\begin{enumerate}
\item[($\gamma$-3')] \emph{ultra-additivity}:
$\gamma(S\cup S') = \max\{\gamma(S),\gamma(S')\}$
for $S, S'\subseteq X$.
\end{enumerate}
Ultra-additivity ($\gamma$-3') is a stronger property 
than monotonicity ($\gamma$-3).
The Hausdorff MNC is an example of an MNC 
satisfying ($\gamma$-3'). 
If $W_1, W_2:X\to 2^{X}$ are 
two multifunctions condensing with respect to $\gamma$ which 
is ultra-additive, then their set-theoretic union 
$W_1\cup W_2:X \to 2^{X}$,
$(W_1\cup W_2)(x) := W_1(x) \cup W_2(x)$, $x\in X$, 
is also condensing with respect to $\gamma$.
\end{remark}

The following is a Birkhoff theorem on minimal invariant set for IFSs.

\begin{theorem}[Birkhoff theorem for condensing IFSs]\label{th:Birkhoff4Condensing}
Let $X$ be a Hausdorff topological space and 
$\gamma$ an MNC in it.
Let $W:X\to 2^{X}$ be a multifunction with nonempty values, 
condensing with respect to $\gamma$. Assume that there exists 
a nonempty closed set $B\subset X$ such that 
$W(B)\subset B$ and $\gamma(W(B))<\infty$.
Then $W$ admits a nonempty minimal invariant set 
which is compact.
\end{theorem}
\begin{proof}
Without loss of generality we may assume that 
$\gamma(B)<\infty$. For, if not, then we can replace 
$B$ with $\overline{W(B)}$. 

Denote by $\mathcal{F}:2^{X}\to 2^{X}$ the Hutchinson operator
induced by $W$. Pick anyhow $b_0\in B$.
Consider the following family of sets
\[
\mathcal{S} = \{S\subseteq X: 
\mathcal{F}(S)\subseteq S=\overline{S}, 
b_0\in S\subseteq B\}.
\]
Obviously $B\in \mathcal{S}$, so 
$\mathcal{S}\neq \emptyset$.

Let $S_{*}=\bigcap \mathcal{S}$. It turns out that: 
(i) $S_{*}\in\mathcal{S}$, 
(ii) $S_{*}$ is the least element 
of $(\mathcal{S},\subseteq)$, and
(iii) $S_{*}$ is compact.

Let us verify (i). It is evident that 
\[
b_0\in S_{*}=\overline{S_{*}} \subseteq B.
\]
Moreover, for all $S\in \mathcal{S}$, we have
\[
\mathcal{F}(S_{*})\subseteq F(S)\subseteq S,
\]
so $\mathcal{F}(S_{*})\subseteq \bigcap\mathcal{S} =S_{*}$.

Property (ii) is obvious from the definition of the  intersection.

It left to verify (iii). 
Put $S_0 = \mathcal{F}(S_{*})\cup \{b_0\}$. Then
\[
\mathcal{F}(S_0)=\mathcal{F}(\mathcal{F}(S_{*}))\cup\mathcal{F}(\{b_0\}) \subseteq 
\mathcal{F}(S_{*})\cup\mathcal{F}(S_{*}) \subseteq 
\mathcal{F}(S_{*}) \cup \{b_0\} = S_0,
\]
because 
$b_0\in S_{*}$ and $\mathcal{F}(S_{*})\subseteq S_{*}$
(thanks to (i)).
Thus $S_0\in \mathcal{S}$.
Since $S_{*}$ is the least element of $\mathcal{S}$
(due to (ii)), we get that 
$S_{0}=\mathcal{F}(S_{*}) \cup \{b_0\}\supseteq S_{*}$.
Therefore we have
\[
\gamma(S_{*})\leq \gamma(\mathcal{F}(S_{*}) \cup \{b_0\}) =
\gamma(\mathcal{F}(S_{*}))\leq \gamma(S_{*})\leq \gamma(B) <\infty.
\]
From the assumption that $W$ is condensing with respect to $\gamma$ 
it follows that $\gamma(S_{*})=0$, so
$S_{*}=\overline{S_{*}}$ is compact.

Summarizing, $\mathcal{F}(S_{*})\subseteq S_{*}$ and 
$S_{*}$ is a nonempty compact subset of $X$.
We are in position to restrict the action of $\mathcal{F}$ 
from $2^{X}$ to ${\mathcal{K}}(S_{*})$. Application
of the Birkhoff theorem for compact IFSs finishes 
the proof.
\end{proof}

\subsection{Condensing maps vs compact maps, IFSs with condensation and weak contractions}
We explain in this section that IFSs of condensing maps 
embrace IFSs on compact spaces and weakly contractive IFSs, 
as well as a mix of both: IFSs with condensation; see also 
Remark \ref{rem:Darbo}. While the case of compact 
maps is readily embraced by condensing maps, 
the case of weakly contractive maps needs more elaboration.
Then the case of IFSs with condensation follows smoothly.

\begin{proposition}
If $W:X\to 2^{X}$ is a compact multifunction, then
it is condensing both 
with respect to the trivial MNC from Example \ref{ex:trivialMNC} 
and with respect to the Hausdorff MNC.
\end{proposition}

It left to discuss weakly contractive multifunctions.

\begin{definition}
A multifunction $W: X\to 2^{X}$ is a 
\emph{multivalued Browder contraction}, if 
\begin{equation*}
d_{H}(W(x_1),W(x_2)) \leq {\varphi}(d(x_1,x_2))
\mbox{ for } x_1,x_2\in X,
\end{equation*}
where $d_{H}$ stands for the Hausdorff distance 
and ${\varphi}$ is a modulus of continuity; 
cf. Definition \ref{def:contractions}.
\end{definition}
A multivalued Banach contraction (i.e.,
${\varphi}(t)=\lambda\cdot t$, $t\in{\mathbb{R}}_{+}$)
is often called the Nadler contraction.

\begin{lemma}\label{lem:WofD}
Let $W:X\to 2^{X}$ be a Browder contraction with
modulus of continuity ${\varphi}$. Then for every $r>0$,
$x\in X$, ${\varepsilon}>0$
\[
W(D(x,r)) \subseteq B(W(x),{\varphi}(r)+{\varepsilon}).
\]
\end{lemma}

\begin{proposition}[\cite{AndresFiser}, 
\cite{AFGL} Proposition 2] 
Let $X$ be a complete metric space.
\begin{enumerate}
\item[(a)] Let $\mathcal{F}=\{w_1,...,w_N\}$ be an IFS 
comprising Browder weak contractions on $X$. 
Then the multifunction $W:X\to 2^{X}$, 
associated with $\mathcal{F}$ according 
to formula \eqref{eq:MultiFromIFS}, 
is a multivalued Browder contraction.
\item[(b)] If a multifunction $W:X\to 2^{X}$ is 
a Browder contraction with compact values,
then it is condensing with respect to the Hausdorff MNC.
\end{enumerate}
\end{proposition}
\begin{proof}
Part (a) is just a particular case of Theorem \ref{th1}. 
Indeed, the Hutchinson operator $\mathcal{F}$ agrees
with $W$ on singletons: $\mathcal{F}(\{x\}) = W(x)$ for $x\in X$.

Part (b) follows from Lemma \ref{lem:WofD} by
handling carefully neighbourhoods and
the definition of the Hausdorff MNC $\gamma$. 

If $\gamma(S)=0$, then $\overline{S}$ is compact.
Since $W$ is continuous with compact values,
we have that 
$\overline{W(S)} \subseteq 
\overline{W(\overline{S})} = W(\overline{S})$ 
are compact;
cf. \cite{Beer} Theorem 6.2.9, Proposition 6.2.11.
Hence $\gamma(W(S))=0$.

Suppose now that $0< \gamma(S) <r$. By the
definition of $\gamma$ 
there exists a finite set $\{x_j\}_{j}$ with
$\bigcup_{j} B(x_j,r)\subseteq S$.
Fix ${\varepsilon}>0$. Let $K=\bigcup_{j} W(x_j)$. 
Compactness of $K$ ensures that there exists 
a finite set $\{y_l\}_{l}$ with 
$\bigcup_{j} B(y_l,{\varepsilon})\supseteq K$.
Therefore we have 
\begin{eqnarray*}
W(S) \subseteq \bigcup_{j} W(B(x_j,r)) \subseteq
\bigcup_{j} B(W(x_j),{\varphi}(r)+{\varepsilon}) = 
\\
= B(K,{\varphi}(r)+{\varepsilon}) \subseteq 
\bigcup_{l} B(y_l,{\varphi}(r)+2{\varepsilon}).
\end{eqnarray*}
Hence $\gamma(W(S))\leq {\varphi}(r)$ because ${\varepsilon}>0$ 
was arbitrary. Further,
\[
\gamma(W(S))\leq
\lim_{r\to \gamma(S)\,+} {\varphi}(r) = 
{\varphi}(\gamma(S)) < \gamma(S).
\]
\end{proof}

The result below allows to apply the theory 
of condensig IFSs to weakly contractive IFSs. 
Namely, one has to restrict a weakly contractive
system to a sufficiently large closed ball $D(x_0,r)$. 
Then the system is condensing with respect to 
the Hausdorff MNC $\gamma$ 
and $\gamma(D(x_0,r))<\infty$.
The same observation can be used to localize an 
attractor (and may be viewed as a ``distant relative'' 
of the so-called collage theorem).

\begin{proposition}[\cite{AFGL} Proposition 3]\label{prop:WofD2} 
Let $W: X\to 2^X$ be a multivalued Browder contraction with bounded values and 
the modulus of continuity ${\varphi}$ satisfying
\begin{equation}\label{eq:firdeepunderr}
\lim_{r\to\infty} (r-{\varphi}(r)) = \infty
\end{equation}
(in particular, ${\varphi}$ can be taken as in the definition of the Rakotch contraction).
Then for each $x_0\in X$ there exists 
sufficiently large radius $r_0>0$, so that
\[
W(D(x_0,r)) \subseteq D(x_0,r)
\]
for all \(r\geq r_0\).
\end{proposition}
\begin{proof}
Denote by ${\varphi}$ the modulus of continuity of $W$.
Fix \({\varepsilon}>0\). 
Thanks to the boundedness of values of $W$, 
there exists \(\rho>0\) such that 
\(W(x_0)\subseteq B(x_0,\rho)\).
Thanks to \eqref{eq:firdeepunderr}
we can find \(r_0\) so that
\begin{equation}\label{eqr-etar}
r-{\varphi}(r)> \rho+{\varepsilon} \mbox{ for all } 
r\geq r_0 .
\end{equation}
Combining Lemma \ref{lem:WofD} with 
\eqref{eqr-etar} gives:
\begin{eqnarray*}
W(D(x_0,r)) \subseteq 
B(W(x_0), {\varphi}(r)+{\varepsilon}) \subseteq 
B(B(x_0,\rho), {{\varphi}(r)+{\varepsilon}}) \subseteq 
\\ 
\subseteq B(x_0,\eta(r)+{\varepsilon}+\rho) 
\subseteq D(x_0,r).
\end{eqnarray*}
\end{proof}

Let us recall that an \emph{IFS with condensation} 
$\mathcal{F}_{K}$ is a weakly contractive IFS $\mathcal{F}=\{w_1,...,w_N\}$ 
on a complete metric space $X$ with 
a given nonempty compact subset $K\subseteq X$;
e.g., \cite{Mendivil} chap.2.6.1, \cite{OS1}. 
The Hutchinson operator for $\mathcal{F}_K$ is defined
by $\mathcal{F}_{K}:2^{X}\to 2^{X}$, 
$\mathcal{F}_{K}(S)=\bigcup_{i=1}^{N} 
\overline{w_i(S)} \cup K$.
The IFS $\mathcal{F}_{K}$ is a multivalued IFS 
induced by a multifunction $W_{K}:X\to 2^{X}$, 
$W_{K}(x) = W(x) \cup K$ for $x\in X$,
where $W$ is a multifunction associated with $\mathcal{F}$ 
according to \eqref{eq:MultiFromIFS}. Then we can 
check that $W_K$ is a multivalued weak contraction
(as a union of single-valued weak contractions and 
a constant multifunction) and we land in the realm 
of weakly contractive multivalued IFSs. 
Another approach could exploit the observation 
that weakly contractive and compact maps 
are condensing with respect to the Hausdorff MNC and 
so are their set-theoretic unions
by Remark \ref{rem:ultraadditive}.

\subsection{Global maximal attractor of the IFS}\label{sec:MaximalAttractor}

In the present subsection we are going to 
present for IFSs an adaptation of the 
classic theory of global maximal attractors 
for semigroups (e.g. \cite{Cheban}, 
\cite{Chueshov}, \cite{SellYou}). The adaptation 
of definitions is not faithful, but it is 
done so that to avoid some technicalities 
with the so-called absorbing sets and 
trapping regions. Throughout the subsection 
we assume that $X$ is a complete metric space.

\begin{definition}
Let $W:X\to 2^{X}$ be a multivalued IFS and 
$\mathcal{F}:2^{X}\to 2^{X}$ the Hutchinson operator 
induced by $W$.
We say that $A\subseteq X$ \emph{attracts}
$S\subseteq X$ under $\mathcal{F}$, provided
$\lim_{n\to\infty} e(\mathcal{F}^n(S),A)=0$; putting that
other way, for every ${\varepsilon}>0$ there exists 
$n_0\in {\mathbb{N}}$ s.t. $\mathcal{F}^n(S)\subseteq B(A,r)$ 
for all $n\geq n_0$. A~nonempty closed set
$A^{*}\subseteq X$ is called a 
\emph{global maximal attractor}, when $A^{*}$
is a minimal nonempty closed set 
attracting all subsets $S\subseteq X$.
\end{definition}

\begin{remark}
The set $A\subseteq X$ is attracting all subsets 
of $X$ if and only if $A$ attracts $X$.
\end{remark}

\begin{example}
Let $\mathcal{F}=\{w_1,...,w_N\}$ be an IFS acting on $X$.
Let $W:X\to 2^{X}$ be the multifunction associated 
with $\mathcal{F}$ according to \eqref{eq:MultiFromIFS}.
If $W(X)=X$, then $X$ is the global maximal attractor of $W$. 
This is the case, when $w_i$ are affine maps 
on the euclidean space $X$, which could make 
the theory of global attractors not interesting 
for the fractal geometry. 
However, the following is true. 
If $\mathcal{F} = \{w_1,...,w_N\}$ is a contractive IFS
on a complete metric space $X$ and 
$W(D(x_0,r))\subseteq D(x_0,r)$ for some $x_0\in X$
and $r>0$ (Proposition \ref{prop:WofD2}),
then the attractor of $\mathcal{F}$ is precisely the
global maximal attractor of $\mathcal{F}$ restricted to
$D(x_0,r)$,
$\mathcal{F}|_{D(x_0,r)} = \{w_1|_{D(x_0,r)},...,
w_N|_{D(x_0,r)}\}$.
\end{example}

Basic properties of the global attractor are 
collected below.

\begin{proposition}
Let $A^{*}$ be a global maximal attractor of 
$W:X\to 2^{X}$. Let $\mathcal{F}:2^{X}\to 2^{X}$ 
be the Hutchinson operator induced by $W$. 
Then the following hold:
\begin{enumerate}
\item[(a)] $A^{*}$ is the smallest nonempty closed 
set attracting all $S\subseteq X$;
\item[(b)] $A^{*}$ has the form 
$A^{*}= \bigcap_{n=1}^{\infty} \mathcal{F}^n(X)$;
\item[(c)] $\mathcal{F}(A^{*})\subseteq A^{*}$;
\item[(d)] $\lim_{n\to\infty}
d_{H}(\mathcal{F}^n(X),A^{*})=0$;
\item[(e)] if $\mathcal{F}(S)\supseteq S$ and 
$S\subseteq X$ is a nonempty closed set,
in particular, if $S$ is an $\mathcal{F}$-invariant set,
then $S\supseteq A^{*}$;
\item[(f)] if $A^{*}$ is compact and 
$W:X\to 2^{X}$ is an upper semicontinuous multifunction
(\cite{Beer} Definition 6.2.4 p.193), i.e., 
for each $x_0\in X$ and 
every open $V\supseteq W(x_0)$ there exists
an open $U\ni x_0$ s.t. $W(U)\subseteq V$,
then $A^{*}$ is the greatest $\mathcal{F}$-invariant set.
\end{enumerate}
\end{proposition}
\begin{proof}
Item (a). 
Let $A$ be a nonempty closed set attracting $X$ under $\mathcal{F}$. 
Fix $r>0$, $0<{\varepsilon}\leq r$. 
Then there exists $n_0$ s.t. 
$\mathcal{F}^n(X)\subseteq B(A^{*},{\varepsilon})\cap B(A,{\varepsilon})$ 
for all $n\geq n_0$.
Now observe that
\[
B(A^{*},{\varepsilon})\cap B(A,{\varepsilon}) \subseteq 
B(A^{*}\cap \overline{B(A,2r)},\;{\varepsilon}).
\]
Hence, as ${\varepsilon}>0$ was arbitrary, 
$A^{*}\cap \overline{B(A,2r)}$ is attracting $X$.
Since $A^{*}$ is a minimal attracting nonempty closed set, we have that 
$A^{*}\cap \overline{B(A,2r)}=A^{*}$. So
$A^{*}\subseteq \overline{B(A,2r)}$ for all $r>0$. 
Overall $A^{*}\subseteq \bigcap_{r>0} 
\overline{B(A,2r)} = \overline{A}=A$;
$A^{*}$ is contained in every nonempty closed attracting set $A$.

Item (b). 
Since $\lim_{n\to\infty} e(\mathcal{F}^n(X), A^{*})$, 
we have 
$\bigcap_{n=1}^{\infty} \mathcal{F}^n(X)\subseteq 
\bigcap_{{\varepsilon}>0} B(A^{*}, {\varepsilon})= \overline{A^{*}} = 
A^{*}$. On the other hand, each $\mathcal{F}^m(X)$, $m\in{\mathbb{N}}$,
is attracting $X$ (as the sequence $\mathcal{F}^n(X)$ is
decreasing with respect to $\subseteq$). Recalling that 
$A^{*}$ is the smallest nonempty closed set attracting $X$ (due to (a)), we arrive at 
$A^{*} \subseteq \mathcal{F}^m(X)$ for each $m$.

Item (c) follows immediately from (b):
\[
\mathcal{F}(A^{\*})\subseteq 
\mathcal{F}\left(\bigcap_{n=1}^{\infty} \mathcal{F}^n(X)\right) 
\subseteq \mathcal{F}^{n+1}(X)\subseteq \mathcal{F}^n(X) 
\mbox{ for all } n,
\]
so $\mathcal{F}(A^{*}) \subseteq 
\bigcap_{n=1}^{\infty} \mathcal{F}^n(X) = A^{*}$.

Item (d) can be seen by observing that
$e(A^{*},\mathcal{F}^n(X))=0$, because of (c).

Item (e). If $S\subseteq \mathcal{F}(S)$, then 
$S\subseteq \mathcal{F}^n(S)\subseteq \mathcal{F}^n(X)$ for all
$n$, so $S\subseteq A^{*}$ thanks to (b).

Item (f). Thanks to (c) and (e) it is enough 
to check that $\mathcal{F}(A^{*})\supseteq A^{*}$.
Fix ${\varepsilon}>0$. By the upper semicontinuity
of $W$ to every $x\in A^{*}$ there exists 
$\delta_{x}>0$ s.t. 
$W(B(x,\delta_{x}))\subseteq B(W(x),{\varepsilon})$. 
The open cover 
$\bigcup_{x\in A^{*}} B(x,\delta_{x}) 
\supseteq A^{*}$
of a compact set has a Lebesgue number $\delta>0$,
i.e., for each $a\in A^{*}$ there exists 
$x\in A^{*}$ s.t. 
$B(a,\delta)\subseteq B(x,\delta_{x})$;
cf. \cite{Beer} Theorem 2.3.1 p.54, 
\cite{engelking} Theorem 4.3.31.
Hence
\begin{equation}\label{eq:WuscOnA}
W(B(A^{*},\delta)) = \bigcup_{x\in A^{*}}
W(B(a,\delta)) \subseteq 
\bigcup_{x\in A^{*}} W(B(x,\delta_{x}))
\subseteq 
\bigcup_{x\in A^{*}} B(W(x),{\varepsilon}) = 
B(W(A^{*}),{\varepsilon}).
\end{equation}
Since $A^{*}$ attracts $X$, we have 
$\mathcal{F}^n(X)\subseteq B(A^{*},\delta)$ for large $n$.
Taking into account \eqref{eq:WuscOnA} yields
\begin{equation*}
\mathcal{F}^{n+1}(X)\subseteq \overline{W(B(A^{*},\delta))}
\subseteq \overline{B(W(A^{*}),{\varepsilon})} \subseteq
B(W(A^{*}),2{\varepsilon})
\mbox{ for large } n.
\end{equation*}
Overall $\mathcal{F}(A^{*})= W(A^{*})$ attracts $X$. 
Since $A^{*}$ is the smallest closed nonempty 
set attracting $X$, due to (a),
we finally have $\mathcal{F}(A^{*})\supseteq A^{*}$.
\end{proof}

The main criterion for the existence of global attractors in IFSs provides

\begin{theorem}
Let $\gamma$ be an MNC in a complete 
metric space $X$. 
Let $W:X \to 2^{X}$ be a multifunction. 
Assume that $\gamma(W(X))<\infty$ 
and either of the conditions holds:
\begin{enumerate}
\item[(i)] $W$ is a set-contraction with respect to
$\gamma$, i.e., there exists $\lambda<1$ s.t.
\[
\gamma(W(S))\leq \lambda\cdot\gamma(S)
\mbox{ for } S\subseteq X;
\]
\item[(ii)] $\gamma$ is the Hausdorff MNC and 
$W$ is condensing with respect to $\gamma$.
\end{enumerate}
Then $\bigcap_{n=1}^{\infty} \mathcal{F}^n(X)$ 
is a compact global maximal attractor of $W$,
where $\mathcal{F}$ is the Hutchinson operator 
associated with $W$.
\end{theorem}
\begin{proof}
Denote 
$A^{*}= \bigcap_{n=1}^{\infty} \mathcal{F}^n(X)$.
It is enough to observe that 
\begin{equation}\label{eq:FnXto0}
\lim_{n\to\infty} \gamma(\mathcal{F}^n(X)) = 0.
\end{equation}
Then by the Kuratowski intersection theorem 
(Remark \ref{eq:KuratowskiIntersection}) we have 
\[
e(\mathcal{F}^n(X),A^{*})\leq d_{H}(\mathcal{F}^n(X),A^{*}) \to 0.
\]
Thus $A^{*}$ is a compact global maximal 
attractor of $W$.

It left to ensure \eqref{eq:FnXto0}. 
Under assumption (i) we have:
\[
\gamma(\mathcal{F}^{n+1}(X))\leq \lambda\cdot \gamma(\mathcal{F}^n(X))
\leq {\lambda}^n \cdot \gamma(\mathcal{F}(X))\to 0.
\]
Under assumption (ii) the property \eqref{eq:FnXto0}
is covered by technical Lemma 5 in \cite{AFGL} 
(see also \cite{AKPRS} Lemma 1.6.11).
\end{proof}

\subsection{Invariant measures. The Krylov-Bogolyubov theorem for IFSs}\label{sec:KrylovBogolubov}

Let ${\mathcal{F}}=\{w_1,...,w_{N}\}$ be an IFS 
of continuous maps acting on a Hausdorff 
topological space $X$. Given a vector 
$\vec{p}=(p_1,...,p_N)$ of positive weights $p_i>0$,
$\sum_{i=1}^{N} p_i=1$, we can form 
a probabilistic IFS $(\mathcal{F},\vec{p})$
(with constant probabilities). 
The \emph{Markov operator} induced by $(\mathcal{F},\vec{p})$, 
$M_{(\mathcal{F},\vec{p})}:{\mathcal{M}}_{\pm}(X)\to{\mathcal{M}}_{\pm}(X)$, 
acts on signed Radon measures on $X$,
according to the formula \eqref{eq:operatorM}, i.e.,
\begin{equation*}
M_{(\mathcal{F},\vec{p})} = \sum_{i=1}^{N} 
p_i\cdot (w_i)_{\sharp}, 
\end{equation*}
where $(w_i)_{\sharp}$ is the push-forward 
of measures 
(see Appendix \ref{section:measures}). 

\begin{proposition}\label{prop:ertiesM}
Let $M:{\mathcal{M}}_{\pm}(X)\to{\mathcal{M}}_{\pm}(X)$
be the Markov operator induced by 
a probabilistic IFS on 
a normal topological space $X$. Then
\begin{enumerate}
\item[(a)] $M$ is linear;
\item[(b)] $M$ is continuous 
with respect to the weak topology;
\item[(c)] $M$ sends probability measures 
to probability measures, that is 
$M({\mathcal{P}}(X))\subseteq {\mathcal{P}}(X)$.
\end{enumerate}
\end{proposition}
\begin{proof}
Items (a) and (c) are readily verified. 
Item (b) follows from Proposition \ref{th:wtransport}.
\end{proof}

\begin{theorem}[Krylov-Bogolyubov theorem for IFSs]\label{th:KrylovBogolubov}
Let $X$ be a compact topological space. 
Let $M$ be a Markov operator associated 
with a probabilistic IFS comprising 
continuous maps on $X$.
Then there exists an invariant probability 
measure $\mu_{*}=M(\mu_{*})$.
\end{theorem}
\begin{proof}
The Markov operator 
$M:{\mathcal{M}}_{\pm}(X)\to {\mathcal{M}}_{\pm}(X)$ 
is linear and weakly continuous, 
the simplex of probability measures 
${\mathcal{P}}(X)\subseteq {\mathcal{M}}_{\pm}(X)$ is a nonempty 
convex weakly compact set, 
and $M({\mathcal{P}}(X))\subseteq {\mathcal{P}}(X)$.
Hence there are several ways to 
establish the theorem.

\emph{Proof via Schauder--Tikhonov principle}. 
The set up allows for a direct application 
of the Schauder-Tikhonov fixed point principle 
on topological vector spaces; e.g., \cite{Granas}
chap.II $\S$7.1.c Theorem (1.13) p.148.

\emph{Proof via Markov--Kakutani theorem}. 
The set up allows for a direct application 
of the Markov-Kakutani fixed point theorem 
on topological vector spaces; e.g., \cite{Granas}
chap.I $\S$3.3 Theorem (3.2) p.43.

\emph{Proof via Mann iteration}. 
Fix $\mu_0\in{\mathcal{P}}(X)$. Consider the orbit 
$M^n(\mu_0)$ and its averages
\[
\nu_n = \frac{1}{n}\cdot \sum_{k=0}^{n-1} M^k(\mu_0),
n\geq 1.
\]
The sequence $\nu_n\in{\mathcal{P}}(X)$ admits 
a weakly convergent subnet 
(not necessarily a subsequence unless 
${\mathcal{P}}(X)$ is a Fr\'{e}chet sequential space, 
cf. \cite{engelking} Exercise 1.6.D):
\begin{equation}\label{eq:accumul-nu}
\nu_{n_k} \to \mu_{*}\in {\mathcal{P}}(X).
\end{equation}
Observe now that due to Lemma \ref{lem:1nwcompact}
the sequence
\begin{equation*}
M(\nu_n) -\nu_n = \frac{1}{n}\cdot 
\left(\sum_{k=1}^{n} M^k(\mu_0)
- \sum_{k=0}^{n-1} M^k(\mu_0)\right) 
= \frac{1}{n} (M^n(\mu_0)-\mu_0) \to 0
\end{equation*}
is weakly convergent to the null measure 
$0\in{\mathcal{M}}_{\pm}(X)$. In particular,
$M(\nu_{n_k}) -\nu_{n_k}\to 0$.
Combining this with \eqref{eq:accumul-nu}
yields
\[
M(\nu_{n_k})\to M(\mu_{*}), M(\nu_{n_k})\to \mu_{*}.
\]
Thus $M(\mu_{*})=\mu_{*}$ (because weak limit is unique).
\end{proof}

\begin{remark}
The Mann iteration is a general iterative scheme 
for finding fixed points; cf. \cite{Berinde} chap.4.
The method of proof via Mann iteration is 
employed for instance in \cite{MyjakSzarek},
Theorem 4.3.
\end{remark}

\begin{remark}
Another way of proving Theorem \ref{th:KrylovBogolubov}
could be an application of the Hahn--Banach theorem
along the lines of the proof of 
the Markov--Kakutani theorem in \cite{Granas} 
(chap.I $\S$3.3 Theorem (3.1) p.43).
\end{remark}

\begin{remark}
Theorem \ref{th:KrylovBogolubov} can be viewed as 
a measure counterpart of the Birkhoff theorem 
for compact IFSs (Theorem \ref{th:Birkhoff4Compact}).
One could ask for a measure counterpart of 
the Birkhoff theorem for condensing IFSs 
(Theorem \ref{th:Birkhoff4Condensing}).
Theorems of this kind take some sort of 
compactness of orbits $\{M^n(\mu)\}_{n=1}^{\infty}$ 
as an assumption, instead of the condensation 
with respect to an MNC, e.g.,
\cite{MyjakSzarek} Theorems 4.3 and 4.11,
\cite{Wisnicki}, \cite{Foguel}, \cite{Skorohod}. 
For comparison, note that 
if $w:X\to X$ is a single-valued map 
condensing with respect to some MNC $\gamma$, 
then every orbit
$S=\{w^n(x)\}_{n=0}^{\infty}$ with $\gamma(S)<\infty$
is relatively compact. 
\end{remark}

The Krylov-Bogolyubov theorem can be used 
to establish the existence of invariant sets, 
because supports of invariant measures 
are invariant sets. Thus, instead of 
studying sets one can study invariant measures,
which is often a preferred approach,
as the theory of Markov processes is 
very rich and well-developed.

\begin{theorem}[\cite{Mendivil} Exercise 2.64, \cite {KleptsynNalskii} Proposition 5]\label{th:invariantsupp}
Let ${\mathcal{F}} = \{w_1,...,w_{N}\}$ be an IFS 
of continuous maps on a Hausdorff topological 
space $X$. Let $\vec{p}=(p_1,...,p_N)$ be 
a fixed vector of positive weights $p_i> 0$. 
Let $M:{\mathcal{P}}(X)\to{\mathcal{P}}(X)$ be the Markov 
operator corresponding to the 
probabilistic IFS $(\mathcal{F},\vec{p})$.
If $\mu_{*}=M\mu_{*}$ is an invariant 
Radon probability measure, then 
$A_{*}=\operatorname{supp} \mu_{*}$ is a closed 
invariant set for ${\mathcal{F}}$, i.e., 
$A_{*}={\mathcal{F}}(A_{*})$.
\end{theorem}
\begin{proof}
Using Proposition \ref{prop:support} we have
\begin{gather*}
{\mathcal{F}}(A_{*}) = \overline{\bigcup_{i=1}^{N} 
w_i(\operatorname{supp}\mu_{*})} =
\\
= \bigcup_{i=1}^{N} 
\overline{w_i(\operatorname{supp}\mu_{*})} =
\bigcup_{i=1}^{N} 
\operatorname{supp} ((w_i)_{\sharp}\mu_{*}) =
\operatorname{supp}\left( \sum_{i=1}^{N} p_i 
\cdot (w_i)_{\sharp}\mu_{*} \right) =
\\
= \operatorname{supp} M\mu_{*} = \operatorname{supp}\mu_{*} = A_{*}.
\end{gather*}
The assumption that $p_i>0$ is needed to have 
$\operatorname{supp} ((w_i)_{\sharp}\mu_{*}) =
\operatorname{supp} (p_i \cdot (w_i)_{\sharp}\mu_{*}).$
\end{proof}

\subsection{Attractors of non-contractive IFSs}

In Section \ref{sec:MaximalAttractor} we have 
presented the theory of global maximal attractors.
These are minimal limits with respect to the upper Hausdorff
metric topology (described by the excess functional, 
cf. \cite{Beer} chap.4.2 p.114) of the Hutchinson iterates.
Still one can ask about attractors of non-contractive 
IFSs understood as CFP's of the Hutchinson operator, 
see Definition \ref{def:Attractor}. Recall that
a nonempty compact set $A_{\mathcal{F}}\subseteq X$ is an attractor
of the IFS $\mathcal{F}$ comprising continuous maps acting 
on a topological space $X$, provided
$\mathcal{F}^n(S)\to A_{\mathcal{F}}$ with respect to the Vietoris topology 
for all nonempty compact $S\subseteq X$. 
(As usual, $\mathcal{F}$ stands for both --- the IFS and 
the induced Hutchinson operator.) Let us remark
that there exists a more refined notion of the 
attractor, called strict attractor, which was proposed
by Barnsley and Vince, and could be roughly described
as a local attractor. More precisely, a nonempty
compact subset $A\subseteq X$ is a 
\emph{strict attractor} of $\mathcal{F}$, if there exists
an open neighbourhood $U\supseteq A$ with 
$\mathcal{F}(U)\subseteq U$ s.t. $A=A_{\mathcal{F}\vert U}$ 
is an attractor of the IFS $\mathcal{F}\vert U$, i.e., 
$\mathcal{F}$ restricted to $U$. (If $\mathcal{F}=\{w_1,...w_N\}$,
$w_i:X\to X$, then 
$\mathcal{F}\vert U= \{w_{1}\vert U,...w_{N}\vert U\}$,
$w_{i}\vert U(x)=w_i(x)$ for $x\in U$.)

The theory of (strict) attractors of non-contractive 
IFSs is rather general and not many results can 
be lifted from the contractive case. For instance
the following theorem has no counterpart 
in the realm of non-contractive IFSs.

\begin{theorem}[Hata's connectedness principle;
\cite{BLesR2018fastbasin} Theorem 1]\label{th:HataConnectedness}
If $\mathcal{F}$ is a TIFS and its attractor $A_{\mathcal{F}}$ is 
connected, then $A_{\mathcal{F}}$ is necessarily locally
connected and arcwise connected.
\end{theorem}

Despite these kind of phenomena met outside 
the contractive realm, some fundamental results
on attractors of contractive IFSs are also valid 
for attractors of non-contractive IFSs. The most
notable case seems to be the chaos game algorithm
on the representation of the attractor by 
$\omega$-limits of typical orbits. Probabilistic
version of the chaos game holds for any IFS on
a topological space; cf. \cite{BLesR2016chaos}.
Derandomization of the chaos game is more demanding;
see the comments at the end of Section 
\ref{sec:ChaosGame}.

Since the theory of (strict) attractors of 
non-contractive IFSs is still under initial
development, we present only a couple of 
characteristic examples.

First example explains why Hata's principle does not
work outside contractive realm.

\begin{example}[Connected non-arcwise connected attractor;
\cite{BLesR2018fastbasin} Example 2]
Let $X=\{0\}\times [-1,1] \cup 
\{(t,\sin\left(\frac{1}{t}\right)): t\in(0,1]\}
\subset \mathbb{R}^2$ 
be the Warsaw sine curve. It is a connected but not arcwise 
connected set. It turns out that there exists a continuous 
map $w_1:X\to X$ and a point $x_0\in X$ s.t. the
orbit $\{w_1^n(x_0)\}_{n=0}^{\infty}$ is dense in $X$, 
e.g., \cite{Sivak} Example 12. 
Putting $w_2:X\to X$ to be the constant map 
$w_2(x)=x_0$ for all $x\in X$, yields an IFS 
$\{w_1,w_2\}$ for which the Warsaw sine is an attractor;
cf. \cite{BLesR2016chaos} Example 4.
\end{example}

Second example shows that CFP of the Hutchinson
operator corresponding to a collection of maps 
does not enforce the CFP for individual maps.

\begin{example}[Non-contractive IFS with attractor]\label{ex:irrot}
Let $X=[0,1]/\{0,1\}$ be the circle (unit interval 
with glued ends). Let $w_1:X\to X$ be the irrational 
rotation, i.e., $w_1(x)=x+r \mod 1$ for $x\in X$ and 
a fixed irrational $r$. Let $w_2:X\to X$ be 
the identity map, i.e., $w_2(x)=x$ for $x\in X$.
Then the IFS of isometries $\{w_1,w_2\}$ induces 
the Hutchinson operator with the CFP 
being the whole circle.
\end{example}

Third example shows that unlike in the case 
of attractors of TIFS, attractors in general
IFSs need not be metrizable.

\begin{example}[Non-metrizable attractor; 
\cite{BLesR2016chaos} Example 6]
Let $X=\bigcup_{j\in\{0,1\}} ([0,1]\setminus\{j\})\times \{j\}
\subset \mathbb{R}^{2}$ be the Alexandrov two arrows space, 
e.g., \cite{Bogachev} Example 6.1.20 or 
\cite{engelking} Exercise 3.10.C p.212.
The space $X$ is a compact non-metrizable Hausdorff 
topological space.
Define continuous maps $w_i:X\to X$, $i=1,2,3$,
according to the formulas
\begin{eqnarray*}
w_{1}(t,j):= \left(\frac{t}{2},j\right),\\ 
w_{2}(t,j):= \left(\frac{t+1}{2},j\right),\\ 
w_{3}(t,j):= (1-t,1-j), 
\end{eqnarray*}
for $j=0,1$, $t\in[0,1]$, $t\neq j$. 
Then the IFS $\{w_{1},w_{2},w_{3}\}$
admits the double arrow space as an attractor.
\end{example}

It is known that the Hilbert cube $[0,1]^{{\mathbb{N}}}$ is
not homeomorphic to a topologically contractive IFS.
On the other hand, it is an open problem whether
Hilbert cubes $[0,1]^{{\mathbb{N}}}$ and $[0,1]^{{\mathbb{R}}}$ 
(infinite product of either countable or continuum
number of copies of the unit interval $[0,1]$)
are attractors of some IFS. We know only that
a Hilbert cube of weight higher than continuum
is not separable, so it cannot be an attractor
of an IFS (\cite{BLesR2016chaos} Proposition 5).

\section{Instead of The End}

Infinite iterated function systems have 
been not discussed in our survey, 
except Remark \ref{rem:generalization}. 
We would like to make some further remarks 
in connection with multivalued IFSs. We do it 
in such a~manner that, hopefully, 
a more unified view on various matters 
will be achieved by the reader.
The basic observation is that some aspects of
the dynamics of infinite IFSs can be captured
by turning an infinite IFS into a multivalued IFS.

If $\{w_i:i\in I\}$ is an infinite system of maps
(i.e., $I$ is infinite) acting on a topological
space $X$, then it induces a multifunction
$W:X\to 2^X$, $W(x):=\{w_i(x): i\in I\}$,
and we end up in the framework of multivalued 
IFSs (see Section \ref{sec:multivaluedIFS}).
But there is a price to pay for this reduction.
Often it is too bold to yield sufficient insight.
For instance, to study the structure of 
the invariant set and, in particular its dimension 
(e.g.,  \cite{Hille}, \cite{Kaenmaki},
\cite{MauldinUrbanski}, \cite{Mantica})
or the chaos game algorithm (\cite{Les2015chaos}),
one needs to access individual maps comprising 
an infinite IFS, rather than look at the 
rough description of a collective dynamical 
behaviour of the IFS encoded by a single 
multifunction. Although selection theorems allow
to decompose a multifunction into individual 
single-valued maps, these decompositions 
(whenever exist) suffer several drawbacks. 
For instance, Lipschitz constants of 
selectors depend not only on the Lipschitz 
constant of a multifunction under 
decomposition, but also upon the 
dimension of the ambient space, cf.
\cite{AubinCellina} chap.1.9
(see also \cite{Les2015chaos} Proposition 9
where the role of equicontinuity is addressed).
Furthermore, it should be stressed out that the 
so-called inverse problem of fractal 
geometry has trivial solution for infinite 
and multivalued IFSs (cf. \cite{Mendivil} 
chap. 2.6.4.1).

Let us suppose that we accept all the 
aforementioned drawbacks and we 
reduce infinite IFSs to multivalued IFSs. 
Another question arises: when an infinite 
IFS of weakly contractive or condensing 
maps gives rise to a weakly contractive 
or, respectively, condensing multivalued IFS?
For the weakly contractive case we discussed 
this in Remark \ref{rem:generalization}. 
For the condensing case the following 
formula for a measure of noncompactness 
of an infinite union of sets
addresses the raised question 
(at least partially):
\[
\sup_{t\in T} \gamma(S_t) + 
\gamma^{\sharp}(\{S_t\}_{t\in T}) \leq 
\gamma\left(\bigcup_{t\in T} S_t\right) \leq
\sup_{t\in T} \gamma(S_t) + 
2\cdot \gamma^{\sharp}(\{S_t\}_{t\in T}),
\]
where $X$ is a metric space, 
$\gamma$ is the Hausdorff MNC in $X$,
$\{S_t\}_{t\in T}$ is an arbitrary family
of subsets $S_t\subseteq X$, and 
$\gamma^{\sharp}$ is the Hausdorff MNC 
with respect to the Hausdorff distance $d_{H}$
in the power set $2^{X}$. (Be aware 
that $d_{H}$ is only an 
extended-valued semimetric in $2^{X}$.) 
Thus, in a suitable function space, a 
relatively compact collection of maps which 
are condensing with respect to the Hausdorff MNC 
gives rise to a condensing set-theoretic 
union of maps; cf. \cite{Les-InfiniteIFSs}.

Another problem is related to weakly 
contractive multifunctions with non-compact 
values. Such multifunctions may come from 
bounded infinite IFSs (see Remark
\ref{rem:generalization} (3)).
Since condensing multifunctions have relatively
compact values, weakly contractive multifunctions
are not reducable to the condensing case in general.
To override this obstacle we are led to
consider hyper-condensing multifunctions,
that is, multifunctions which are condensing 
with respect to an MNC in the hyperspace; 
cf. \cite{LesHyperCondensing2}.

Usually it is demanded that attractors are 
compact sets. Nevertheless, it is worth 
to consider non-compact attractors.
Already infinite and multivalued IFSs comprising 
weakly contractive maps lead to non-compact 
closed bounded attractors 
(e.g., \cite{AndresFiser}, \cite{Wicks}).
Quite extraordinarily, another proposal,
the Mauldin--Urba\'{n}ski limit set, can 
be non-closed invariant set with a 
complicated descriptive topology (cf.
\cite{MauldinUrbanski} Example 5.2.2 p.140).
Finally, unbounded fractal sets offer 
an interesting and fruitful excursion 
outside the realm of compacta.
These include the Lasota--Myjak semiattractors 
like the Sierpi\'{n}ski chessboard
(\cite{LasotaMyjak1996} Example 6.2), 
the Barnsley--Vince fast basins 
like the Kigami web
(\cite{BLesR2018fastbasin} Section 5 Fig.4)
and fractal tilings
(e.g., \cite{BV-Developments} Section 10).

New extensions of the framework of IFSs 
are constantly proposed, e.g., 
generalized IFSs (GIFS) which comprise 
mappings defined on a finite Cartesian 
product $X^m$ (or even infinite product) 
with values in $X$, 
cf. \cite{MicMih} and \cite{JMS}.
Numerous approaches to self-similarity 
and frameworks related to IFSs are 
explored by researchers to this day: 
cocycles and non-autonomous 
dynamical systems (e.g., \cite{Guzik},
\cite{Cheban}, \cite{Akin}),
abstract self-similarity via topology, 
category theory and algebra (e.g., \cite{Wicks}, 
\cite{Charatonik}, \cite{Leinster}, \cite{Furstenberg}, 
\cite{SanchezGranero}, \cite{Kuhlmann}),
infinite products of matrices and chains 
(e.g., \cite{Hartfiel}, \cite{Mikosch}),
which is just a small sample 
to move the imagination of the reader.
Iterated function systems constitute only
one --- though sparkling creativity --- 
view on the rich landscape of dynamics.

\section{Appendix: Topology and measure}
We collect in this section 
the rudimentary notation and terminology 
from topology and measure theory.

Let $(X,d)$ be a metric space. 
For $x\in X$, $S\subset X$ and $r>0$, 
we define
\begin{itemize}
\item an \emph{open ball}, 
$B(x,r) = \{y\in X: d(y,x)<r\}$;
\item a \emph{closed ball}, 
$D(x,r) = \{y\in X: d(y,x)\leq r\}$; 
\item an \emph{$r$-neighbourhood} of $S$, 
$B(S,r) = \bigcup_{x\in S} B(x,r)$;
\item \emph{diameter} of $S$,
$\operatorname{diam}(S) = \sup_{x,x'\in S} d(x,x')$.
\end{itemize}

The closure of a subset $S\subseteq X$ of 
a metric or topological space $X$ 
is denoted by $\overline{S}$.

A set which is both closed and open 
is shortly called \emph{clopen}.
A set is \emph{perfect} if it is closed 
and has no isolated points.

\subsection{Nets}
In topological spaces the convergence 
of countable sequences is not enough 
to describe topology. For this reason, the notion of nets 
(a.k.a. Moore--Smith sequences) was introduced; e.g. 
\cite{Granas} Appendix: Preliminaries B p.593.

A \emph{directed set} is a pair $(N,\succeq)$, 
where $N$ is a nonempty set and $\succeq$ is 
a binary relation satisfying: 
(i) (\emph{reflexivity}) 
$n\succeq n$ for all $n\in N$;
(ii) (\emph{transitivity}) 
for every $n_1, n_2, n_3\in N$, 
if $n_3\succeq n_2\succeq n_1$, 
then $n_3\succeq n_1$;
(iii) (\emph{direction}) 
for every $n_1,n_2\in N$ there exists $n_3\in N$
s.t. $n_3\succeq n_1$, $n_3\succeq n_2$.

A \emph{net} of elements from $X$ is a function
$x:(N,\succeq)\to X$, denoted $(x_n)_{n\in N}$
or simply $x_n$. The net $(x_n)_{n\in N}$ 
of elements from a topological space $X$ is
\emph{convergent} to $x\in X$, written $x_n\to x$,
if for every open neighbourhood $U\ni x$ there exists
$n_0\in N$ such that $x_n\in U$ for all $n\succeq n_0$.

Let $(N,\succeq)$ and $(K,\gg)$ be directed sets. 
Let $\vec{n}:K\to N$ satisfy: 
(i) (\emph{monotonicity}) for every $k_1,k_2\in K$,
if $k_2\gg k_1$, then 
$\vec{n}(k_2)\succeq \vec{n}(k_1)$;
(ii) (\emph{cofinality}) for every $n\in N$ 
there exists $k\in K$ such that $\vec{n}(k)\succeq n$.
Given a net $(x_n)_{n\in N}$ and 
$\vec{n}:K\to N$,
the net $(x_{n_k})_{k\in K}$, 
where $n_k:=\vec{n}(k)$, is called
a~\emph{subnet} of $x_n$.

We have the following properties:
\begin{itemize}
\item every sequence $(x_n)_{n=1}^{\infty}$
is a net with a directed set of indices $({\mathbb{N}},\geq)$;
\item every subsequence is a subnet, although
there exist on some compact spaces 
countable sequences with plenty 
of convergent subnets yet no convergent subsequence;
\item in a Hausdorff topological space a limit of 
the net is unique;
\item if a net $x_n$ converges to $x$, 
then all its subnets $x_{n_k}$ converge to $x$;
\item a subset $S\subseteq X$ is compact precisely
when every net over $S$, $x_n\in S$, 
admits a convergent subnet $x_{n_k}\to x\in S$;
\item a map $w:X\to Y$ is continuous precisely 
when from the convergence of 
the net $x_n\to x$ it follows that the net $w(x_n)\to w(x)$ converges.
\end{itemize}

\subsection{Baire category}\label{section:baire}

If $X$ is a metric (or topological) space, 
then we say that $M\subset X$ is called:
\begin{itemize}
\item \emph{nowhere dense}, 
provided $\operatorname{Int}(\overline{M})=\emptyset$,
where $\operatorname{Int}$ stands for the interior; 
equivalently, under assumption that 
$X$ is a metric space:
for every $x\in X$ and $R>0$, 
there exist $y\in X$ and $r>0$ such that
\begin{equation}\label{eq:nowheredenseviaball}
B(y,r)\subseteq B(x,R)\setminus M;
\end{equation}
\item of the \emph{first Baire category} 
or \emph{meager}, if $M$ is a countable 
union of nowhere dense sets;
\item \emph{residual}, if $X\setminus M$ 
is of the first Baire category.
\end{itemize}
A famous Baire category theorem states 
that if $X$ is a complete metric space, 
then each meager set has empty interior.
Also, there are sets which have empty 
interior but are not meager (for example 
the set of irrationals on the real line ${\mathbb{R}}$). 
Meager sets in a complete metric space 
are considered to be small, and, in turn, 
residual sets are big. 
In particular, it is common to say that 
a \emph{typical element} from a complete 
space $X$ has some property, say (P), 
if the set 
$\{x\in X:x\mbox{ has property } (P)\}$ 
is residual.

Finally, let us note that, 
at least within metric spaces,
there exist approaches to topologic smallness 
other than the Baire category.
Namely, in the ``ball characterization" of 
nowhere density \eqref{eq:nowheredenseviaball} 
one can require that the smaller ball $B(y,r)$ 
is not too small with respect to a bigger 
one $B(x,R)$. Thus we obtain sets with sufficiently
large holes, called \emph{porous}. 
This idea can be formalized in many ways, 
cf. \cite{Z2}. By substituting meager sets
with countable unions of porous sets, 
called \emph{$\sigma$-porous},
we get a~strengthened analogue 
of the Baire category (which could be 
termed the Denjoy category).

\subsection{Hyperspaces}
Let $X$ be a metric or topological space. 
We distinguish the following families of sets:
\begin{itemize}
\item the family of all subsets of $X$, 
denoted $2^{X}$;
\item the family of all nonempty compact 
subsets of $X$, denoted ${\mathcal{K}}(X)$;
\item the family of all nonempty closed bounded
subsets of $X$, denoted $\mathcal{CB}(X)$ 
(provided $X$ is a~metric space).
\end{itemize}
Once a family of sets is topologized, we call 
it a \emph{hyperspace}.

If $(X,d)$ is a metric space, then 
for $S,S'\subseteq X$ we define
\begin{itemize} 
\item $e(S,S') = \sup_{x\in S} \inf_{x'\in S'}\; d(x,x')$;
\item the \emph{Hausdorff distance}:
$d_{H}(S,S') = \max\{e(S,S'), e(S',S)\}$.
\end{itemize}
(Conveniently $\inf\emptyset =\infty$.)
The geometric sense of $e$ and $d_{H}$ exhibits
the following description.

\begin{proposition}
Let $X$ be a metric space. 
For $S,S'\subseteq X$ it holds: 
\begin{enumerate}
\item[(a)] $e(S,S')= \inf \{r>0: S\subseteq B(S',r)\}$;
\item[(a)] $d_{H}(S,S')= \inf \{r>0: S\subseteq B(S',r), 
S'\subseteq B(S,r)\}$.
\end{enumerate}
\end{proposition}

If $X$ is a Hausdorff topological space, then 
in ${\mathcal{K}}(X)$ (or more generally in the family 
of all nonempty closed subsets of $X$) 
we introduce the \emph{Vietoris topology}
by declaring that the following sets form 
its open subbase: 
\begin{equation*}
V^{+} = \{K\in{\mathcal{K}}(X): K\subseteq V\},
V^{-} = \{K\in{\mathcal{K}}(X): K\cap V\neq\emptyset\}
\end{equation*}
where $V$ runs over open subsets of $X$.
The Vietoris topology in ${\mathcal{K}}(X)$ satisfies 
the Hausdorff separation; 
cf. \cite{engelking} Problem 3.12.27(b).

\begin{theorem}[\cite{Beer} Definition 3.2.1, 
Theorem 3.2.4, Exercise 3.2.9, 
or \cite{engelking} Problem 4.5.23]\label{th:HyperspaceCompleteCompact}
Let $X$ be a metric space.
\begin{enumerate}
\item[(a)] The pair $(\mathcal{CB}(X), d_{H})$
is a metric space, while $({\mathcal{K}}(X), d_{H})$
is its closed subspace.
\item[(b)] If $X$ is a complete space, 
then $\mathcal{CB}(X)$ and ${\mathcal{K}}(X)$ are complete
with respect to $d_{H}$.
\item[(c)] If $X$ is a compact space, then 
${\mathcal{K}}(X)=\mathcal{CB}(X)$ is compact with respect to $d_{H}$. 
\item[(d)] 
The Vietoris topology in ${\mathcal{K}}(X)$ coincides with 
the Hausdorff topology (that is, induced by $d_{H}$).
In particular, topologically equivalent metrics in $X$
yield topologically equivalent Hausdorff metrics 
in ${\mathcal{K}}(X)$.
\end{enumerate}
\end{theorem}

We end this section with a lemma 
which seems to be a folklore (see \cite{Myjak}).
Recall that a \emph{Cantor space} is 
a topological space homeomorphic to 
the Cantor ternary set. Equivalently 
(e.g., \cite{engelking}, Exercise 6.2.A (c) and 
Theorems 6.2.1 and 6.2.9), 
it is any compact metrizable topological space 
without isolated points
which is \emph{totally disconnected} 
in the sense that it has no nontrivial 
connected subsets.

\begin{lemma}\label{lemma:cantor}
If $X$ is a complete metric space without 
isolated points, then the set 
\[
\mathcal{C}(X) :=\{K\in{\mathcal{K}}(X): K \mbox{ is a Cantor space}\}
\]
is residual in ${\mathcal{K}}(X)$.
\end{lemma}
\begin{proof}
For every $n\in{\mathbb{N}}$, let 
\begin{equation*}
G_n:= \left\{K\in{\mathcal{K}}(X):
\forall_{x\in K}\;\exists_{y\in K}\;
0<d(x,y)<\frac{1}{n}\right\}.
\end{equation*}
It is easy to see that arbitrarily close to 
any $K\in{\mathcal{K}}(X)$ we can find a finite set 
which belongs to $G_n$ (multiplying, 
if needed, some points in initially 
chosen finite set). Therefore each $G_n$ 
is dense in ${\mathcal{K}}(X)$. Also, it is open. 
Indeed, choose any $K\in G_n$ and 
for any $0<\alpha<\beta<\frac{1}{n}$, set 
\[
K_{\alpha,\beta}:= \{x\in K: 
\exists_{y\in K}\;\alpha<d(x,y)<\beta\}.
\] 
Then each set $K_{\alpha,\beta}$ is 
open in $K$ and
$K=\bigcup_{\alpha,\beta} K_{\alpha,\beta}$. 
By compactness of $K$, we can find
$0<\alpha_0<\beta_0<\frac{1}{n}$ so that 
$K=K_{\alpha_0,\beta_0}$. Then $G_n$ contains 
an open ball around $K$ with radius 
$\frac{1}{2} \min\{\alpha_0,\frac{1}{n}-\beta_0\}$. 
Hence $G_n$ is open in $({\mathcal{K}}(X),d_{H})$. 

Now for every $n\in{\mathbb{N}}$, let
\begin{equation*}
D_n:=\left\{K\in{\mathcal{K}}(X):\operatorname{diam}(P)<\frac{1}{n} 
\mbox{ for any connected } P\subseteq K\right\}.
\end{equation*}
Each set $D_n$ is dense because, 
similarly as for $G_n$,
arbitrarily close to any $K\in {\mathcal{K}}(X)$ 
we can find a~finite set, which clearly 
belongs to $D_n$. We will prove that it is open 
by showing that ${\mathcal{K}}(X)\setminus D_n$ is closed. 

Choose a sequence $K_k\subseteq {\mathcal{K}}(X)\setminus D_n$ 
convergent to some $K\in {\mathcal{K}}(X)$, as $k\to\infty$. 
Then for every $k\in{\mathbb{N}}$, we can find 
a connected set $P_k\subseteq K_k$ 
such that $\operatorname{diam}(P_k)\geq \frac{1}{n}$. 
Since the set 
\[
C:=K\cup \bigcup_{k\in{\mathbb{N}}} K_k =
\overline{\bigcup_{k\in{\mathbb{N}}} K_k}
\]
is compact, the sequence $P_k\in {\mathcal{K}}(C)$, 
admits a convergent subsequence 
$P_{m_k}\to P\in{\mathcal{K}}(C)$ thanks to 
Theorem \ref{th:HyperspaceCompleteCompact} (c)). 
Now it is routine to check that its limit $P$ 
is a subset of $K$, 
$\operatorname{diam}(P)\geq \frac{1}{n}$ and 
that $P$ is connected, i.e., $K\notin D_n$. 
All in all, the set $D_n$ is open.

Taking the above onto account, we infer that the set 
$(\bigcap_{\in{\mathbb{N}}} G_n)\cap(\bigcap_{n\in{\mathbb{N}}} D_n)$ 
is residual. Obviously, it equals $\mathcal{C}(X)$ 
and the result follows.
\end{proof}

\subsection{Measures}\label{section:measures}

Let $X$ be a Hausdorff topological space. 
The \emph{Borel $\sigma$-algebra}, denoted ${\mathcal{B}}(X)$, 
is the smallest $\sigma$-algebra containing 
open sets in $X$. Its elements are called 
\emph{Borel sets}. A function $w:X\to Y$ between 
two Hausdorff topological spaces $X,Y$ is called 
\emph{Borel measurable}, if 
$w^{-1}(B)\in{\mathcal{B}}(X)$ for all $B\in{\mathcal{B}}(Y)$. 
In particular a continuous map $w$ is such.

A functional $\mu:{\mathcal{B}}(X)\to (-\infty,\infty)$, 
for which $\mu(\emptyset)=0$, is called
\begin{itemize}
\item a \emph{signed Borel measure} if it is
countably additive, i.e.,
$\mu(\bigcup_{k=1}^{\infty} B_k) = 
\sum_{k=1}^{\infty} \mu(B_k)$
and
$\sum_{k=1}^{\infty} |\mu(B_k)|<\infty$
for every countable family of disjoint 
Borel sets $B_k\in {\mathcal{B}}(X)$, 
$B_k\cap B_m = \emptyset$ for $k\neq m$, $k,m\in{\mathbb{N}}$;
\item a \emph{Borel measure} if $\mu$ is 
a signed Borel measure and
$\mu(B)\geq 0$ for all $B\in{\mathcal{B}}(X)$; 
we write then $\mu:{\mathcal{B}}(X)\to [0,\infty)$;
\item a \emph{Radon measure}
if $\mu:{\mathcal{B}}(X)\to [0,\infty)$ is a Borel measure and
\[
\mu(B) = \sup \{\mu(K): K\subseteq B, 
K\in{\mathcal{K}}(X)\cup\{\emptyset\}\}
\mbox{ for all } B\in{\mathcal{B}}(X); 
\]
equivalently, for every ${\varepsilon}>0$, $B\in{\mathcal{B}}(X)$, 
there exists a compact subset $K\subseteq B$ 
such that $\mu(B\setminus K)<{\varepsilon}$;
\item a \emph{signed Radon measure}, 
if $\mu=\mu_1-\mu_2$ for two Radon measures
$\mu_1,\mu_2:{\mathcal{B}}(X)\to[0,\infty)$;
\item a \emph{probability measure}, if 
$\mu$ is a Radon measure for which $\mu(X)=1$.
\end{itemize}
We consider only finite (signed) measures, 
so the adjective \emph{finite} is usually omitted.

Note that on a complete separable metric space, 
every (signed) Borel measure is necessarily 
a~(signed) Radon measure 
(\cite{Bogachev} Theorem 7.1.7), while 
there exists a Borel measure on a compact 
topological space which is not a Radon measure
(\cite{Bogachev} Example 7.1.3).

We denote the following collections of measures:
\begin{itemize}
\item the collection of all signed Radon 
measures on $X$, ${\mathcal{M}}_{\pm}(X)$;
\item the collection of Radon 
probability measures on $X$, ${\mathcal{P}}(X)$.
\end{itemize}
Overall we have 
$\delta_{x}\in {\mathcal{P}}(X)\subseteq {\mathcal{M}}_{\pm}(X)$, where
$\delta_{x}$ is the Dirac measure at $x\in X$.

\begin{proposition}
Signed Radon measures form a vector space 
${\mathcal{M}}_{\pm}(X)$ under the addition of measures,
multiplication of a measure by a scalar, 
and with the null measure $0\in{\mathcal{M}}_{\pm}(X)$ 
as a zero vector. 
Radon probablility measures ${\mathcal{P}}(X)$ form 
a convex subset of ${\mathcal{M}}_{\pm}(X)$.
\end{proposition}
\begin{proof}
It is enough to check that if $\mu$, $\nu$ 
are (nonnegative!) Radon measures and 
$p>0$, then $\mu+\nu$ and $p\cdot \mu$
are Radon measures. The rest is obvious.

Fix ${\varepsilon}>0$, $B\in{\mathcal{B}}(X)$ and compact subsets 
$K,K'\subseteq B$ s.t. 
$\mu(B\setminus K)<{\varepsilon}$,
$\nu(B\setminus K')<{\varepsilon}$.
Then $K\cup K'\subseteq B$ is compact, 
$(\mu+\nu)(B\setminus (K\cup K')) \leq 
\mu(B\setminus K) + \nu(B\setminus K') <2{\varepsilon}$,
$p\cdot \mu(B\setminus K)< p\cdot{\varepsilon}$.
\end{proof}

Every signed Borel measure 
$\mu:{\mathcal{B}}(X)\to(-\infty,\infty)$
is a difference $\mu=\mu_{1}-\mu_{2}$ 
of two nonnegative Borel measures 
$\mu_{1},\mu_{2}:{\mathcal{B}}(X)\to[0,\infty)$.
In particular, 
$\int_{X} f \;d\mu = 
\int_{X} f \;d\mu_{1} - \int_{X} f \;d\mu_{2}$
for a~bounded Borel measurable function $f:X\to {\mathbb{R}}$.
The \emph{Jordan-Hahn decomposition} 
$\mu=\mu^{+}-\mu^{-}$, 
$\mu^{+},\mu^{-}:{\mathcal{B}}(X)\to[0,\infty)$,
is the difference being minimal in the sense
that for any other decomposition
$\mu=\mu_{1}-\mu_{2}$ 
into nonnegative Borel measures $\mu_1,\mu_2$,
it holds $\mu_1\geq\mu^{+}$, $\mu_2\leq\mu^{-}$;
e.g., \cite{Bogachev} chap.3.1.
Given the Jordan-Hahn decomposition 
$\mu=\mu^{+}-\mu^{-}$ of a signed measure $\mu$
we can define the \emph{total variation measure}
of $\mu$, $|\mu|:{\mathcal{B}}(X)\to[0,\infty)$, 
$|\mu|:= \mu^{+}+\mu^{-}$. 
It should be remarked that 
a signed Borel measure $\mu$ is Radon if and only 
if its total variation measure $|\mu|$ is Radon.
We also have
\begin{equation}\label{eq:integraltotalvar}
\left| \int_{X} f \;d\mu \right| \leq
\sup_{x\in X} |f(x)| \cdot |\mu|
\end{equation}
for a bounded Borel function $f:X\to {\mathbb{R}}$

Let $X$ be a Hausdorff topological space.
Denote by ${\mathcal{C}}_b(X)$ the space of all bounded 
continuous real-valued functions $f:X\to{\mathbb{R}}$.
We endow the vector space ${\mathcal{M}}_{\pm}(X)$ 
of signed Radon measures on $X$ 
with the \emph{weak topology} generated
by the following subbasis of open sets 
\begin{equation*}
V_{f,{\varepsilon}}(\mu) := \left\{ \nu\in{\mathcal{M}}_{\pm}(X):
\left|\int_{X} f \;d\nu - \int_{X} f \;d\mu\right| 
<{\varepsilon} \right\},
\end{equation*}
where $f\in{\mathcal{C}}_b(X)$, ${\varepsilon}>0$, $\mu\in{\mathcal{M}}_{\pm}(X)$. 
So endowed ${\mathcal{M}}_{\pm}(X)$ constitutes 
a locally convex topological vector space:
the addition of measures and muliplication 
of a measure by a scalar are continuous 
in the weak topology and 
the subbasic neighbourhoods are convex. 

In terms of functional analysis, the weak topology 
in ${\mathcal{M}}_{\pm}(X)$ is the weak* topology 
transported from the predual space 
$(\mathcal{C}_b(X), \|\cdot\|_{\infty})$ via duality pairing 
\[\mathcal{C}_b(X)\times {\mathcal{M}}_{\pm}(X) \ni
\langle f, \mu \rangle \mapsto \int_{X} f \;d\mu.\]
It is a standard fact that the weak* topology 
in the space of continuous linear functionals
$(\mathcal{C}_b(X))^{*}$ obeys Hausdorff separation without
any assumptions on $X$. To identify measures with functionals
requires that some sort of the Riesz-Skorokhod representation
theorem holds true, which is quite restrictive. 
(If $X$ is not locally compact, then not all elements 
of $(\mathcal{C}_b(X))^{*}$ arise as integrals with respect to measures
and distinct measures may yield the same functional; 
say all functions in $\mathcal{C}_b(X)$ are constant, so 
$\int_{X} f \;d \delta_{x_1} = \int_{X} f \;d \delta_{x_2}$ 
for $x_1\neq x_2\in X$, $f\in \mathcal{C}_b(X)$.) 
Therefore we deliver a direct proof that ${\mathcal{M}}_{\pm}(X)$ 
is Hausdorff when $X$ is normal.

\begin{proposition}\label{prop:WeakTopIsHausdorff}
If $X$ is a normal topological space, 
then the space ${\mathcal{M}}_{\pm}(X)$ 
of signed Radon measures on $X$ 
equipped with the weak topology
is a Hausdorff space.
\end{proposition}
\begin{proof}
Let $\mu,\eta\in {\mathcal{M}}_{\pm}(X)$, $\mu\neq\eta$.
We shall find $\delta >0$ and $f\in\mathcal{C}_b(X)$ s.t.
$V_{f,\delta}(\mu)\cap V_{f,\delta}(\eta) =\emptyset.$

Since $\mu$ and $\eta$ are Radon, 
there exists a compact set $K\subseteq X$ 
s.t. $\mu(K)\neq \eta(K)$. 
Set ${\varepsilon}:=|\mu(K)-\eta(K)|>0$ and 
let $\delta:=\frac{1}{4}{\varepsilon}$. 
By the definition of a Radon measure, we can find 
compact sets $S',S''\subseteq X\setminus K$ 
such that $|\mu|((X\setminus K)\setminus S')<\delta$, 
and $|\eta|((X\setminus K)\setminus S'')<\delta$.  

Define $f(x):=0$ for $x\in S'\cup S''$, 
and $f(x):=1$ for $x\in K$. 
Then $f:S'\cup S''\cup K\to[0,1]$ is continuous. 
By the Tietze-Urysohn theorem 
(\cite{engelking} Theorem 2.1.8), 
$f$ admits a continuous extension $f:X\to[0,1]$. 
Let us observe that using standard properties 
of the integral, in particular inequality 
\eqref{eq:integraltotalvar}, we have
\begin{gather*}
{\varepsilon}=|\mu(K)-\eta(K)|=
\left|\int_K{f}\;d\mu-\int_K{f}\;d\eta\right| =
\\
= \left|\int_X{f}\;d\mu-\int_X{f}\;d\eta -
\int_{(X\setminus K)}{f}\;d\mu +
\int_{(X\setminus K)}{f}\;d\eta\right| \leq 
\\
\leq \left|\int_X{f}\;d\mu-\int_X{f}\;d\eta\right|+
\left|\int_{(X\setminus K)}{f}\;d\mu\right|+
\left|\int_{(X\setminus K)}{f}\;d\eta\right| =
\\
= \left|\int_X{f}\;d\mu-\int_X{f}\;d\eta\right|+
\left|\int_{(X\setminus K)\setminus S'}{f}\;d\mu\right|+
\left|\int_{(X\setminus K)\setminus S''}{f}\;d\eta\right| \leq
\\
\leq \left|\int_X{f}\;d\mu-\int_X{f}\;d\eta\right| +
\left|\mu\right|((X\setminus K)\setminus S')
+\left|\eta\right|((X\setminus K)\setminus S'') <
\\
< \left|\int_X{f}\;d\mu-\int_X{f}\;d\eta\right|+2\delta.
\end{gather*}
Hence
\[
\left|\int_X{f}\;d\mu-\int_X{f}\;d\eta\right|\geq {\varepsilon}-2\delta.
\]

Suppose that there is 
${\nu}\in V_{f,\delta}(\mu)\cap V_{f,\delta}(\eta)$. 
Then 
\[
{\varepsilon}-2\delta\leq \left|\int_X{f}\;d\mu-\int_X{f}\;d\eta\right|
\leq \left|\int_X{f}\;d\mu-\int_X{f}\;d{\nu}\right|+
\left|\int_X{f}\;d{\nu}-\int_X{f}\;d\eta\right|<2\delta,
\]
which yields a contradiction with ${\varepsilon}= 4\delta$.
\end{proof}

If a net (or a sequence) $\mu_n\in{\mathcal{M}}_{\pm}(X)$ 
converges to $\mu\in{\mathcal{M}}_{\pm}(X)$ 
in the weak topology, that is 
$\int_{X} f \;d\mu_n \to \int_{X} f \;d\mu$
for all $f\in{\mathcal{C}}_b(X)$, then we say that 
$\mu_n$ \emph{weakly converges} to $\mu$.
The weak limit is unique when $X$ is normal. 
The weak convergence is usually not interesting 
when $X$ is not normal. For instance, 
if $\mathcal{C}_b(X)$ merely consists of constant 
functions (e.g., \cite{engelking} Problem 2.7.17 p.119),
then any given sequence of probability measures 
weakly converges to all probability measures at once.

\begin{lemma}\label{lem:1nwcompact}
Let $\Delta\subseteq {\mathcal{M}}_{\pm}(X)$ be a set of
signed Radon measures which is compact with respect to 
the weak topology. Then
\begin{enumerate}
\item[(a)] $\Delta'= \{\int_{X} f \;d\mu: 
\mu\in\Delta\}\subseteq {\mathbb{R}}$ is compact 
(hence bounded) for each $f\in{\mathcal{C}}_b(X)$;
\item[(b)] every sequence $\mu_n\in\Delta$ has 
the property that $\frac{1}{n}\cdot\mu_n\to 0$
in the weak topology.
\end{enumerate}
\end{lemma}
\begin{proof}
For (a) observe that for each fixed $f\in{\mathcal{C}}_b(X)$ 
the functional
\[
{\mathcal{M}}_{\pm}(X)\ni \mu \stackrel{T}{\mapsto} 
\int_{X} f \;d\mu \in{\mathbb{R}}
\]
is continuous in the weak topology. Hence
$T(\Delta)=\Delta'$ is compact.

Item (b) follows from (a), because for every 
$f\in{\mathcal{C}}_b(X)$ one has
\[
\left|\int_{X} f 
\;d\left(\frac{1}{n}\cdot\mu_n\right)\right| =
\frac{1}{n}\cdot \left|\int_{X} f 
\;d\mu_n\right| \leq\frac{1}{n}\cdot 
\sup_{\delta\in\Delta'} |\delta|\to 0.
\]
\end{proof}

\begin{theorem}[Alexandrov--Prokhorov; \cite{Bogachev} Theorem 8.9.3 (i)]
If $X$ is a compact topological space, then
the set ${\mathcal{P}}(X)$ of probability measures on $X$
is a compact subset of ${\mathcal{M}}_{\pm}(X)$ with respect to
the weak topology.
\end{theorem}

Given a Borel measurable $w:X\to Y$ 
and a signed Borel measure 
$\mu:{\mathcal{B}}(X)\to(-\infty,\infty)$,
we define the \emph{push forward measure} through $w$
to be $w_{\sharp}\mu: {\mathcal{B}}(Y)\to(-\infty,\infty)$,
\[
w_{\sharp}\mu(B) := \mu(w^{-1}(B)) 
\mbox{ for all } B\in{\mathcal{B}}(Y).
\]
We also write $\mu\circ w^{-1}$ for $w_{\sharp}\mu$.
Obviously the operator $w_{\sharp}$ is linear:
\begin{equation*}
w_{\sharp}(p_1\cdot \mu_1+ p_2\cdot \mu_2) = 
p_1\cdot w_{\sharp}\mu_1+ p_2\cdot w_{\sharp}\mu_2,
\end{equation*}
for all measures $\mu_i$ and scalars $p_i$, $i=1,2$.

\begin{proposition}[\cite{Bogachev} Theorem 9.1.1(i), Theorem 3.6.1, chap.8.10(v)]\label{th:wtransport}
Let $w:X\to Y$ be a continuous map between 
Hausdorff topological spaces $X, Y$. Then 
\begin{enumerate}
\item[(a)] $w_{\sharp}\mu$ is a 
(signed or nonnegative) Radon measure in $Y$ 
whenever $\mu$ is a 
(signed or respectively, nonnegative) 
Radon measure in $X$;
\item[(b)] for all bounded Borel measurable 
functions $g:Y\to{\mathbb{R}}$ 
\[
\int_{X} g \;d(w_{\sharp}\mu) = 
\int_{X} g\circ w \;d\mu,
\]
where $\mu$ is a signed Borel measure in $X$;
\item[(c)]
the transport of measure
$w_{\sharp}:{\mathcal{M}}_{\pm}(X)\to {\mathcal{M}}_{\pm}(Y)$,
$w_{\sharp}(\mu) = \mu\circ w^{-1}$ 
for $\mu\in{\mathcal{M}}_{\pm}(X)$,
is continuous in the weak topology.
\end{enumerate}
\end{proposition}
\begin{proof}
For (a) fix ${\varepsilon}>0$, $B\in{\mathcal{B}}(Y)$ and 
a nonnegative Radon measure $\mu$. 
Since $\mu$ is Radon there exists a compact subset 
$K\subseteq w^{-1}(B)$ s.t. 
$\mu(w^{-1}(B)\setminus K)<{\varepsilon}$.
Simple set-algebra shows that
\[
w^{-1}(B\setminus w(K))\subseteq 
w^{-1}(B)\setminus K.
\]
Hence $w_{\sharp}\mu(B\setminus K') <{\varepsilon}$
for the compact set $K'=w(K)\subseteq B$.
For the case of signed $\mu$ it is enough 
to use the Jordan-Hahn decomposition 
$w_{\sharp}\mu= w_{\sharp}\mu^{+} - w_{\sharp}\mu^{-}$. 

We check (b) for $g=\chi_{B}$, 
the characteristic function of the set $B\in{\mathcal{B}}(Y)$.
Recall that $\chi_B(y)=1$ if $y\in B$, 
and $0$ otherwise. Observe that 
$\chi_{B}\circ w=\chi_{w^{-1}(B)}$.
Thus we have
\[
\int_{X} \chi_{B} \;d(w_{\sharp}\mu) =
w_{\sharp}\mu(B) = \mu(w^{-1}(B)) = 
\int_{X} \chi_{w^{-1}(B)} \;d\mu =
\int_{X} \chi_{B}\circ w \;d\mu.
\]
One readily extends (b) to simple functions $g$, 
their monotone limits and differences of these limits.

For (c) it is enough to check that for every 
$\mu\in M_{\pm}(X)$, ${\varepsilon}>0$ and $g\in {\mathcal{C}}_b(Y)$
there exists $f\in {\mathcal{C}}_b(X)$ s.t. 
\begin{equation}\label{eq:wtransport}
w_{\sharp} (V_{f,{\varepsilon}}(\mu)) \subseteq 
V_{g,{\varepsilon}}(w_{\sharp}\mu)
\end{equation}
where $V$ (with parameters) stand for neighbourhoods
generating the weak topologies in ${\mathcal{M}}_{\pm}(Y)$
and ${\mathcal{M}}_{\pm}(X)$.
As shown below, $f=g\circ w$ verifies 
the condition \eqref{eq:wtransport}. Indeed,
for $\nu\in V_{f,{\varepsilon}}(\mu)$ we have
\begin{eqnarray*}
\left| \int_{X} g \;d(w_{\sharp}\nu) - 
\int_{X} g \;d(w_{\sharp}\mu) \right| =
\left| \int_{X} g\circ w \;d\nu - 
\int_{X} g\circ w \;d\mu \right| < {\varepsilon}.
\end{eqnarray*}
\end{proof}

\begin{definition}
For a nonnegative Borel measure 
$\mu:{\mathcal{B}}(X)\to[0,\infty)$ on 
a Hausdorff topological space $X$ 
we define the \emph{support} of $\mu$ to be
\[
\operatorname{supp} \mu = \{x\in X: \mu(U) >0 
\mbox{\;for all open neighbourhoods \;} 
U\ni x\}.
\]
Alternatively, $C=\operatorname{supp} \mu$ is the largest 
closed set $C$ s.t. $\mu(U)>0$ for every 
open $U$ with $U\cap C\neq \emptyset$. 
\end{definition}

Supports, especially of Radon measures, 
obey some nice properties.

\begin{lemma}\label{lem:Radonsupport}
Let $\mu:{\mathcal{B}}(X)\to[0,\infty)$ be a nonnegative Radon 
measure on a Hausdorff topological space $X$. 
\begin{enumerate}
\item[(i)] If $B\in{\mathcal{B}}(X)$ is such that 
$B\subseteq X\setminus \operatorname{supp}\mu$,  
then $\mu(B) = 0$. Conversely,
if $\mu(B)>0$, then 
$B\cap \operatorname{supp}\mu\neq\emptyset$.
\item[(ii)] If $\mu$ is nontrivial ($\mu\neq 0$),
then $\operatorname{supp}\mu\neq\emptyset$.
\end{enumerate}
\end{lemma}
\begin{proof}
Item (ii) is immediate from (i). 
For item (i) let $B\in{\mathcal{B}}(X)$ be such that 
$B\subseteq X\setminus \operatorname{supp}\mu$ and suppose,
a contrario, that $\mu(B)>0$. Since $\mu$ is Radon,
there exists a compact $K\subseteq B$ s.t. 
$\mu(K)>0$. Since $K\subseteq X\setminus \operatorname{supp}\mu$,
for each $y\in K$ there exists an open $U_y\ni y$
s.t. $\mu(U_y)=0$. On taking a finite subcover
$\bigcup_{j} U_{y_j}\supseteq K$ we 
arrive at contradiction:
\[
\mu(K)\leq 
\mu\left(\bigcup_{j} U_{y_j}\right)\leq
\sum_{j} \mu(U_{y_j})=0.
\]
\end{proof}

\begin{proposition}\label{prop:support}
Let $X,Y$ be Hausdorff topological spaces. 
Let $\mu,\nu$ be finite nonnegative Borel measures  
on $X$ and $p>0$. 
For (c) assume additionally that $\mu$ is 
a Radon measure and $w:X\to Y$ is continuous.
\begin{enumerate}
\item[(a)] $\operatorname{supp} (p\cdot \mu) = \operatorname{supp} \mu$,
\item[(b)] $\operatorname{supp} (\mu + \nu) = \operatorname{supp}\mu \cup \operatorname{supp}\nu$,
\item[(c)] $\operatorname{supp} w_{\sharp} \mu = 
\overline{w(\operatorname{supp}\mu)}$.
\end{enumerate}
\end{proposition}
\begin{proof}
Parts (a) and (b) are easy to check. 
For (c) put $C=\operatorname{supp} \mu$. We shall verify that 
$\operatorname{supp} (w_{\sharp}\mu) = \overline{w(C)}$. 

First we check that  
$\operatorname{supp} (w_{\sharp}\mu) \supseteq w(C)$. 
Take an open $V\subseteq Y$ s.t. 
$V\cap w(C)\neq\emptyset$. 
Then $w^{-1}(V)$ is open and 
$w^{-1}(V)\cap C\neq\emptyset$. 
Hence $w_{\sharp}\mu(V) = \mu(w^{-1}(V)) >0$ 
by definition of $\operatorname{supp} \mu$.   

Now, by contradiction, we suppose that there exists 
$y\in\operatorname{supp} w_{\sharp}\mu \setminus\overline{w(C)}$. 
We can separate $y$ from $\overline{w(C)}$ by some 
open neighbourhood $V\ni y$, 
$V\cap \overline{w(C)}=\emptyset$. 
(Namely $V:= X\setminus \overline{w(C)}$.)
Put $U:=w^{-1}(V)$. Then $U$ is open and 
$U\cap C =\emptyset$. 
(Indeed, $x\in w^{-1}(V)\cap C$ implies 
$w(x) \in V\cap w(C) =\emptyset$.)
Hence $U\subseteq X\setminus \operatorname{supp}\mu$. 
Since $\mu$ is Radon we can apply 
Lemma \ref{lem:Radonsupport}.
Finally, we get 
$w_{\sharp}\mu(V) = \mu(w^{-1}(V)) = \mu(U) = 0$. 
As $V\ni y$, this means that 
$y\not\in\operatorname{supp} w_{\sharp}\mu$. 
\end{proof}

Now we discuss the special case when 
$(X,d)$ is a complete metric space. 
By $\mathcal{P}_1(X)$ we denote 
the collection of Radon probability 
measures with integrable distance, 
i.e., measures $\mu\in{\mathcal{P}}(X)$ such 
that for some (equivalently -- for all) 
$x_0\in X$, the integral 
$\int_{X} d(x,x_0)\;d\mu(x)<\infty$. 
We endow the space $\mathcal{P}_1(X)$ with
\begin{itemize}
\item
the \emph{Monge-Kantorovich metric} 
(cf. \cite{Mendivil} chap. B.5.1)
\begin{equation*}
d_{MK}(\mu,\nu):= \sup\left\{\int_{X} f\;d\mu 
- \int_{X} f\;d\nu:\, 
f:X\to{\mathbb{R}}, \operatorname{Lip}(f)\leq 1\right\};
\end{equation*}
\item
the \emph{Fortet-Mourier metric} (cf. \cite{MyjakSzarek})
\begin{equation*}
d_{FM}(\mu,\nu):= \sup\left\{\left|\int_{X} f\;d\mu 
- \int_{X} f\;d\nu\right|:\, 
f:X\to{\mathbb{R}}, \operatorname{Lip}(f)\leq 1, 
\sup_{x\in X} |f(x)|\leq 1\right\}.
\end{equation*}
\end{itemize}
In the above, $\mu,\nu\in{\mathcal{P}}_1(X)$, and
\[
\operatorname{Lip}(f) = \sup_{x_1\neq x_2} 
\frac{\,|f(x_1)-f(x_2)|\,}{d(x_1,x_2)}
\]
is the minimal Lipschitz constant of $f:X\to{\mathbb{R}}$.

\begin{remark}
(i) If $X$ is bounded, then
$\mathcal{P}_1(X)=\mathcal{P}(X)$.

(ii)
If $X$ is a separable complete metric space, 
then $\mathcal{P}_1(X)$ constitutes 
the space of all Borel probability measures 
$\mu$ on $X$ which satisfy 
$\int_{X} d(x,x_0)\;d\mu<\infty$ 
for some $x_0\in X$ (see \cite{Mendivil}). 
\end{remark}

\begin{lemma}\label{lemma:measuresonmetric}
Let $\mathcal{P}_1(X)$ be the set of 
Radon probability measures with integrable distance 
on a~complete metric space $X$. 
\begin{itemize}
\item[(i)] The Monge-Kantorovich metric $d_{MK}$ 
is a complete metric on $\mathcal{P}_1(X)$.
\item[(ii)] The convergence with respect to 
$d_{MK}$ implies weak convergence. 
If additionally $X$ is bounded, then 
the metric $d_{MK}$ generates the weak topology.
\item[(iii)] For every $\mu,\eta\in \mathcal{P}_1(X)$,
\begin{equation*}
d_{MK}(\mu,\eta)=\min\left\{\int_{X\times X} d(x,y)
\;d\lambda(x,y):\lambda\in \Lambda(\mu,\eta)\right\},
\end{equation*}
where $\Lambda(\mu,\eta)$ is the set of 
all couplings of $\mu$ and $\eta$, i.e.,
Radon probability measures 
$\lambda$ on $X\times X$
with marginals $\lambda(B\times X)=\mu(B)$ 
and $\lambda(X\times B)=\eta(B)$ 
for $B\in {\mathcal{B}}(X)$.
\end{itemize}
\end{lemma}
\begin{proof}
We only explain how the items (i)-(iii) follow 
from much more general results that can 
be found in \cite{Bogachev}. We note a specific notation
and terminology used in \cite{Bogachev}.
$d_{MK}(\mu,\eta)=\|\mu-\eta\|_{0}^{*}$
(modified Kantorovich-Rubinshtein metric),
$d_{FM}(\mu,\eta)=\|\mu-\eta\|_{0}$
(Kantorovich-Rubinshtein metric).

By \cite{Bogachev} Theorem 8.3.2,
the metric $d_{FM}$ on the space of all nonnegative, 
so called \emph{$\tau$-additive} measures generates 
the weak topology. However, 
by \cite{Bogachev} Proposition 7.2.2, 
each Radon measure is $\tau$-additive. 
Hence the metric $d_{FM}$ restricted 
to $\mathcal{P}_1(X)$ generates 
the weak topology on $\mathcal{P}_1(X)$.
As stated in \cite{Bogachev} chap.8.10(viii) on p. 234, 
for Borel measures $\mu$ on $X$ with 
integrable distance, 
the following inequalities hold:
\begin{equation}\label{ineq:Monge}
d_{FM}\leq d_{MK}\leq\max\{1,\operatorname{diam}(X)\}d_{FM}.
\end{equation}
This gives us (ii). 

On the same page we find a statement 
that the metric $d_{MK}$ is complete 
on the set $\mathcal{A}(X)$ of all Borel 
probability $\tau$-additive measures on $X$ 
with integrable distance. 
Since $\mathcal{P}_1(X)\subseteq \mathcal{A}(X)$, 
to see (i) it remains to show that 
$\mathcal{P}_1(X)$ is a closed subset of $\mathcal{A}(X)$.
By \cite{Bogachev} Theorem 8.10.43, 
the space $\mathcal{P}(X)$ of all Radon 
probability measures is complete with respect to $d_{FM}$,
so $\mathcal{P}(X)$ is a closed subset of 
the family of all nonnegative $\tau$-additive measures 
with resopect to $d_{FM}$. 
Assume that a~sequence $\mu_n\in\mathcal{P}_1(X)$ 
converges to some $\mu\in\mathcal{A}(X)$ w.r.t. $d_{MK}$. 
By inequalities in \eqref{ineq:Monge}, we also have $d_{FM}(\mu_n,\mu)\to 0$, 
so $\mu\in \mathcal{P}(X)$. Therefore, 
$\mu\in\mathcal{P}(X)\cap\mathcal{A}(X) =\mathcal{P}_1(X)$ 
and we are done.

Finally, (iii) is a straight consequence 
of \cite{Bogachev} Theorem 8.10.45.
\end{proof}

\subsection{Code space}\label{sec:code space}
Let $N\in {\mathbb{N}}$. A finite set of symbols $I=\{1,..,N\}$ 
is often called an \emph{alphabet}, 
while its elements --- \emph{symbols}, 
in the context of string operations on 
finite and infinite sequences of symbols, 
called \emph{words} on such occasion. 
The empty word is denoted by $\varnothing$.

We distinguish the following sets:
\begin{itemize}
\item the set of finite words of length $k\in{\mathbb{N}}$, which 
is simply $I^k$, the Cartesian product of 
$k$ copies of $I$; conveniently 
$I^0 =\{\varnothing\}$;
\item the set of all finite words
$I^{<\infty} =\bigcup_{k\in{\mathbb{N}}\cup\{0\}} I^k$;
\item the set of all infinite words $I^{\infty}$,
which is a countable Cartesian product of $I$.
\end{itemize}

Let $k\in{\mathbb{N}}$,
let $\alpha=(\alpha_1,...,\alpha_k) \in I^{k}$,
be a finite word and 
$\beta=(\beta_1,\beta_2,...)\in I^{<\infty}\cup I^{\infty}$
be a finite or an infinite word over $I$.
We define:
\begin{itemize}
\item the \emph{concatenation} of $\alpha$ and $\beta$ by
$\alpha\widehat{\;\;}\beta := 
(\alpha_1,...,\alpha_k,\beta_1,\beta_2,...)$;
\item the \emph{$k$-prefix} of $\beta$ to be 
$\beta_{\vert k} = (\beta_1,...,\beta_k)\in I^{k}$;
conveniently $\beta_{\vert 0} = \varnothing$.
\end{itemize}

The \emph{Baire metric} in $I^{\infty}$ 
is defined as follows:
\begin{equation*}
d_{B}(\alpha,\beta):=2^{-m},
\end{equation*}
where $\alpha=(\alpha_n)_{n=1}^{\infty}$, 
$\beta=(\beta_n)_{n=1}^{\infty}\in I^\infty$ and 
$m:=\inf\{n\in{\mathbb{N}}:\alpha_n\neq\beta_n\}$, 
provided $2^{-\inf\emptyset}=0$. 
The space $(I^{\infty}, d_{B})$ is called 
the \emph{code space}. 
It is a compact ultrametric space,
it bears the topology of a~countable Tikhonov
product of discrete spaces, all equal $I$,
and it is a Cantor space, e.g., \cite{Edgar}. 

\begin{lemma}\label{lem:denseincodespace}
A subset $\Omega\subset I^{\infty}$
is dense with respect to $d_{B}$ if prefixes of elements from 
$\Omega$ exhaust all possible finite words, i.e., 
for every $k\in{\mathbb{N}}$ and $\alpha\in I^k$ 
there exists $\omega\in\Omega$ s.t.
$\omega_{\vert k}=\alpha$.
\end{lemma}

The Borel $\sigma$-algebra of $I^{\infty}$ is 
a countable product of full $\sigma$-algebras of $I$, 
$\mathcal{B}(I^{\infty}) = \bigotimes_{n\in {\mathbb{N}}} 2^{I}$.
(Recall that if $I$ is discrete, then 
$\mathcal{B}(I)= 2^{I}$.)

\begin{definition}
Let $\vec{p}=(p_1,...,p_N)$ be a pobability vector
in $I=\{1,..,N\}$, $\sum_{i=1}^{N} p_i=1$, 
$p_i = \Pr(i)>0$ for $i\in I$.
The \emph{Bernoulli measure}
is a Borel probability measure on $I^{\infty}$,
$\mu_{b}: \mathcal{B}(I^{\infty})\to [0,1]$,
uniquely identified on cylinders
\[
\mu_{b}(\{(\alpha_1,...,\alpha_k)\}\times I^\infty) = 
\prod_{j=1}^{k} p_{\alpha_{j}},
\mbox{ where } \alpha_1,...,\alpha_k\in I, k\in{\mathbb{N}}.
\]
\end{definition}

\subsection{Ordered sets}
Let $(L,\preceq)$ be an abstract poset 
(partially ordered set). Let $C\subseteq L$. 
We say that $m\in L$ is
\begin{itemize}
\item the \emph{greatest element} of $C$, if $m\in C$
and $c\preceq m$ for all $c\in C$;
\item a \emph{lower bound} of $C$, when 
$m\preceq c$ for all $c\in C$; symbolically $m\preceq C$;
\item a \emph{minimal element}, if $m\in C$ and 
there is no $c\in C$ such that $c\preceq m \neq c$;
\item  an \emph{infimum} of $C$, written $m=\inf C$, 
provided $m$ is the greatest lower bound of $C$, i.e.,
$m$ is the greatest element of the set 
$\{m'\in L: m'\preceq C\}$; 
in yet other words, $m=\inf C$ if and only if  
(i) $m\preceq C$, and 
(ii) for every $m'$, if $m'\preceq C$, then $m'\preceq m$. 
\end{itemize}
Dual concepts: the least element, an upper bound, supremum,
and a maximal element are defined similarly 
(by reversing the order).

A linearly ordered subset $C\subseteq L$, i.e., 
for every $c_1,c_2\in C$ either $c_1\preceq c_2$ or
$c_2\preceq c_1$, is called a~\emph{chain} in $L$.

A mapping $\mathcal{F}: L\to L$ is \emph{order-monotone}, provided
\begin{equation*}
\mbox{for all } z_1,z_2\in L, 
\mbox{if } z_1\preceq z_2, \mbox{ then } \mathcal{F}(z_1)\preceq \mathcal{F}(z_2). 
\end{equation*}

Although chains in abstract posets may lack infima and 
lower bounds, it is no longer the case for suitable hyperspaces.

\begin{lemma}
Every chain of sets $\mathcal{C}$ in the poset of 
nonempty compact sets $({\mathcal{K}}(X), \subseteq)$
admits an infimum, which is the intersection 
$M=\bigcap \mathcal{C}$. 
\end{lemma}
\begin{proof}
By the Riesz characterization of compactness 
via centered families we know that $M\in{\mathcal{K}}(X)$.
Obviously $M\subseteq C$ for all $C\in \mathcal{C}$,
so $M$ is the lower bound for $\mathcal{C}$.
On the other hand, if $M'\subseteq C$ for all 
$C\in \mathcal{C}$, then $M'\subseteq M$. That is,
$M$ is the greatest lower bound for $\mathcal{C}$.
\end{proof}

\begin{theorem}[\cite{LesCEJM} Theorems 2.1 and 2.4]\label{th:Knaster}
Let $(L,\preceq)$ be a poset with the following property:
each chain in $L$ admits an infimum.
Let $\mathcal{F}: L\to L$ be an order-monotone mapping.
\begin{enumerate}
\item[(I)] (Knaster--Tarski)
If there exists $z_0\in L$ with $\mathcal{F}(z_0)\preceq z_0$, then
$\mathcal{F}$ admits a minimal fixed point $z_{*}=\mathcal{F}(z_{*})$.
\item[(II)] (Kleene)
If $L$ has the greatest element,
then $\mathcal{F}$ admits the greatest fixed point $z^{*}=\mathcal{F}(z^{*})$.
\end{enumerate}
\end{theorem}

Theorem \ref{th:Knaster} is stated in a greatly
simplified fashion compared to 
\cite{LesCEJM} Theorems 2.1 and 2.4; 
see also \cite{Granas} chap.1 $\S$2 p.25 
or \cite{Tarafdar} chap.3.2 and chap.3.3.4.
Its proof exploits Zorn's lemma and 
transfinite induction, much in the spirit 
of a direct proof of the Birkhoff theorem 
on invariant set, see section \ref{sec:Birkhoff4Compact}.
Many problems fit naturally within the 
order-theoretic framework, for example, 
the existence of fixed points of 
the Hutchinson operator. However, 
researchers often prefer to use 
in their proofs Zorn's lemma and transfinite
induction directly, e.g., \cite{Ok}.
Finally, note that the application of 
order-theoretic techniques, say, 
to analysis and differential equations, 
except some notable cases 
(e.g. \cite{Tarafdar} chap.3.5), 
may involve rather technical 
assumptions, cf. \cite{Heikkila}.

\section*{Acknowledgements}
We would like to thank the referee 
for her/his careful reading of this long paper 
and giving us valuable suggestions 
which improved the presentation.

\section*{Publishing information}
The final version of this paper will be published in Journal of Difference Equations and Applications.

\end{document}